\documentclass{amsart} 

\usepackage{amsmath, amsthm, url, booktabs}
\usepackage{tikz} 
\usepackage{tikz-cd} 
\usetikzlibrary{matrix,decorations.pathmorphing, positioning}
\usetikzlibrary{arrows,automata,backgrounds,petri}
\usetikzlibrary{shapes.multipart}
\usetikzlibrary{calc, trees}

\usepackage[]{amsrefs}
\usepackage{mathrsfs} % Use \mathscr
\usepackage{enumitem}

\theoremstyle{plain}
\newtheorem{theorem}{Theorem}[section]
\newtheorem{corollary}[theorem]{Corollary}
\newtheorem{proposition}[theorem]{Proposition}
\newtheorem{lemma}[theorem]{Lemma}
\newtheorem*{theorem*}{Theorem}
\newtheorem*{lemma*}{Lemma}
\newtheorem*{proposition*}{Proposition}
\newtheorem*{corollary*}{Corollary}

\theoremstyle{definition}
\newtheorem{definition}[theorem]{Definition}

\newtheorem{remark}[theorem]{Remark}

\newtheorem*{definition*}{Definition}
\newtheorem*{example*}{Example}
\newtheorem*{remark*}{Remark}

\newcommand{\QQQ}{\mathscr{Q}}

\newcommand{\TTT}{\mathcal{T}}

\newcommand{\LLL}{\mathscr{L}}
\newcommand{\AAA}{\mathscr{A}}

\DeclareMathOperator{\Tr}{Tr}
\DeclareMathOperator{\Ht}{Ht}

\DeclareMathOperator{\Cyl}{Cyl}

\title[Intrinsic Diophantine approximation on the unit circle]{Intrinsic Diophantine approximation on the unit circle and its Lagrange spectrum}

\author{Byungchul Cha}
\address{Muhlenberg College, 2400 Chew st, Allentown, PA, 18104, USA}
\email[B. Cha]{cha@muhlenberg.edu}

\author{Dong Han Kim}
\address{Department of Mathematics Education, Dongguk University - Seoul, 30 Pildong-ro 1-gil, Jung-gu, Seoul, 04620 Korea}
\email[D.H. Kim]{kim2010@dongguk.edu}
\thanks{Supported by the National Research Foundation of Korea (NRF-2018R1A2B6001624).}



\begin{document}

\begin{abstract}
    Let $\LLL(S^1)$ be the Lagrange spectrum arising from intrinsic Diophantine approximation on the unit circle $S^1$ by its rational points. 
    We give a complete description of the structure of $\LLL(S^1)$ below its smallest accumulation point.
    To this end, we use digit expansions of points on $S^1$, which were originally introduced by Romik in 2008 as an analogue of simple continued fraction of a real number. 
    We prove that the smallest accumulation point of $\LLL(S^1)$ is 2. 
    Also we characterize the points on $S^1$ whose Lagrange numbers are less than 2 in terms of Romik's digit expansions. 
    Our theorem is the analogue of the celebrated theorem of Markoff on badly approximable real numbers. 
\end{abstract}

\maketitle

\section{Introduction}\label{SecIntro}
 
\subsection{Motivation}
Call $(a, b)$ a \emph{Pythagorean pair} if $a$ and $b$ are nonnegative coprime integers such that $a^2 + b^2$ is a square. 
Suppose that we draw a half-line $\ell$ from the origin $O$ into the first quadrant of an affine coordinate plane
and we aim to make $\ell$ stay as far away as possible from all but finitely many Pythagorean pairs.
What is the greatest possible margin by which $\ell$ misses all but finitely many Pythagorean pairs?
What is the second greatest?

To formulate this question more precisely, we parametrize such a half-line by a point $P$ in the unit quarter circle $\QQQ$, defined by
\[
	\QQQ = \{ (x, y) \in \mathbb{R}^2 \mid x^2 + y^2 = 1, \text{ and } x, y\ge 0 \}.
\]
Write $\hat\delta (P; (a, b))$ for the shortest (Euclidean) distance between a Pythagorean pair $(a, b)$ and the half-line $\overrightarrow{OP}$.
Then we are interested in minimizing   
\[
	L(P) = \limsup_{(a, b)} \hat\delta (P; (a, b))^{-1},
\]
where Pythagorean pairs $(a, b)$ are ordered by their Euclidean norms $\sqrt{a^2 + b^2}$. 

Theorem~\ref{MainTheoremI} provides an answer to the questions posed in the first paragraph of this paper.
Note that, for each value of $L(P)$ in Table~\ref{tab:Top10}, there are (infinitely) many $P$ in $\QQQ$ which produce the same $L(P)$ and the table lists only one of them. For instance, both $(\frac{\sqrt3}2, \frac12)$ and $(\frac12, \frac{\sqrt3}2)$ give the same value $L(P) = \sqrt3$. 

\begin{theorem}\label{MainTheoremI}
The 10 smallest values of $L(P)$, together with corresponding $P$'s, are as in Table~\ref{tab:Top10}.
\end{theorem}

\begin{table}
\caption{Top 10 smallest values of $L(P)$ and corresponding $P$.}
	\[
	   \begin{array}{@{} ll @{}}
			\toprule
			L(P) & P \\
			\midrule
			 \sqrt2 = 1.414213562\dots 
			 &		\left(\frac1{\sqrt2}, \frac1{\sqrt2}\right)  \\
			 \sqrt3 = 1.732050808\dots 
			 &		\left(\frac12, \frac {\sqrt3}{2}\right)  \\
			\frac{\sqrt{34}}{3} = 1.943650632\dots 
			&
			\left(\frac{3}{34} \sqrt{34},\,\frac{5}{34} \sqrt{34}\right)
			\\
			\frac{3\sqrt{11}}{5} = 1.989974874\dots 
			&
			\left(\frac{9}{50} \sqrt{11} - \frac{2}{25},\,\frac{6}{25} \sqrt{11} + \frac{3}{50}\right)
			\\
			\frac{\sqrt{482}}{11} = 1.995863491\dots 
			&
			\left(\frac{11}{482} \sqrt{482},\,\frac{19}{482} \sqrt{482}\right)
			\\
			\frac{\sqrt{1154}}{17} = 1.998269147\dots 
			& 
			\left(\frac{7}{390} \sqrt{1154} - \frac{6}{65},\,\frac{3}{130} \sqrt{1154} + \frac{14}{195}\right)
			\\
			\frac{\sqrt{6722}}{41} = 1.999702536\dots 
			& 
			\left(\frac{41}{6722} \sqrt{6722},\,\frac{71}{6722} \sqrt{6722}\right)
			\\
			\frac{\sqrt{3363}}{29} = 1.999702713\dots 
			& 
			\left(\frac{9}{853} \sqrt{3363} - \frac{161}{1706},\,\frac{23}{1706} \sqrt{3363} + \frac{63}{853}\right)
			\\
			\frac{\sqrt{13922}}{59} = 1.999856358 \dots  
			& 
			\left(\frac{71}{14066} \sqrt{13922} - \frac{570}{7033},\,\frac{95}{14066} \sqrt{13922} + \frac{426}{7033}\right)
			\\
			\frac{\sqrt{16899}}{65} = 1.999940828\dots 
			&
			\left(\frac{33}{8450} \sqrt{16899} - \frac{28}{4225},\,\frac{28}{4225} \sqrt{16899} + \frac{33}{8450}\right)
			\\
			\bottomrule
		\end{array}
		\]
		\label{tab:Top10}
\end{table}
We note here that our main theorem (Theorem~\ref{MainTheoremDescDiscretePart}) gives an arbitrarily long list of ranking of $L(P)$, not just the top 10. 
Another implication of Theorem~\ref{MainTheoremDescDiscretePart} is given in the following theorem.
\begin{theorem}\label{MainTheoremII}
	The smallest accumulation point of the set
\[
	\{ L(P) \in \mathbb{R} \mid P \in \QQQ \}
\]
is 2.
\end{theorem}

\subsection{General setting and description of main result}
We begin with a general set-up for intrinsic Diophantine approximation.
As will be explained later, our main theorem 
(Theorem~\ref{MainTheoremDescDiscretePart}) is the analogue of the celebrated theorem of Markoff in \cite{Mar79} and \cite{Mar80} on badly approximable numbers by rationals.
In the present paper, we will not attempt to give a comprehensive review on the vast body of existing literature regarding Markoff's theorem and subsequent developments. 
Instead, we refer interested readers to a book \cite{CF89} written by Cusick and Flahive and a survey paper \cite{Mal77} by Malyshev.

Let $(\mathcal{X}, d(\cdot, \cdot))$ be a complete metric space and let $\mathcal{Y}$ be a closed subset of $\mathcal{X}$.
Assume that $\mathcal{Y}$ is contained in the closure of a countable subset $\mathcal{Z}$ of $\mathcal{X}$.
In addition, we assume that there is a \emph{height function} $H: \mathcal{Z} \longrightarrow \mathbb{R}_{\ge0}$, whose inverse image of any finite set is finite.
Given the data
$(\mathcal{X}, \mathcal{Y}, \mathcal{Z}, H)$,
we define \emph{the Lagrange number} $L(P)$ of $P\in \mathcal{Y}-\mathcal{Z}$ to be
\[
	L(P) = \limsup_{Z\in \mathcal{Z}}\frac1{H(Z)d(P, Z)}
\]
and \emph{the Lagrange spectrum} to be
\[
	\LLL(\mathcal{Y}) = \{ L(P) \mid P \in \mathcal{Y}-\mathcal{Z}, \quad L(P) <\infty \}.
\]

A classical Lagrange spectrum studied by Markoff in the papers \cite{Mar79} and \cite{Mar80} is concerned with 
$(\mathcal{X}, \mathcal{Y}, \mathcal{Z}, H) = (\mathbb{R}, \mathbb{R},
\mathbb{Q}, H)$
with $H(p/q) = |q|^2$ for coprime integers $p$ and $q$.
Particularly relevant to the present paper is \emph{intrinsic Diophantine approximation on $n$-spheres},
$(\mathcal{X}, \mathcal{Y}, \mathcal{Z}, H) = (\mathbb{R}^{n+1}, S^n, 
S^n \cap \mathbb{Q}^{n+1}, H)$,
which is studied in \cite{KM15} and \cite{FKMS}.
Here, $S^n$ is a unit $n$-sphere in $\mathbb{R}^{n+1}$ centered at the origin 
and the height function $H$ is defined by
$H(\mathbf{p}/q) =|q|$ with primitive $\mathbf{p}\in\mathbb{Z}^{n+1}$, meaning that all coefficients of $\mathbf{p}$ have no common divisor $>1$.
Generally speaking, much less is known about the spectrum $\LLL(S^n)$ 
than the classical Lagrange spectrum of Markoff.
In \cite{KM15} Kleinbock and Merrill show that
$\LLL(S^n)$ is bounded away from 0
for every $n\ge1$.
Kopetzky \cite{Kop80} appears to be the first to determine the minimum of $\LLL(S^1)$. 
Later, Moshchevitin independently discovered the (same) minimum of $\LLL(S^1)$ in \cite{Mos16}.

A point  $P\in \mathcal{Y}-\mathcal{Z}$ is commonly called \emph{badly approximable} if $L(P) < \infty$.
In this paper, we will say that $P\in \mathcal{Y}-\mathcal{Z}$ is \emph{very badly approximable} if $L(P)$ is less than the smallest accumulation point of $\LLL(\mathcal{Y})$.  

In the classical case  
$(\mathcal{X}, \mathcal{Y}, \mathcal{Z}, H) = 
(\mathbb{R}, \mathbb{R}, \mathbb{Q}, H)$, it is well-known that the spectrum $\LLL(\mathbb{R})$ is a closed subset of $(0, \infty)$, 
$\LLL(\mathbb{R})$ contains a discrete part on the lower end and a (closed) interval $[c, \infty)$ on the higher end, and the smallest accumulation point of $\LLL(\mathbb{R})$ is 3.
A celebrated theorem of Markoff in \cite{Mar79} and \cite{Mar80} gives a complete description of all very badly approximable points for
$(\mathbb{R}, \mathbb{R}, \mathbb{Q}, H)$.
As a direct analogue of this, our main theorem (Theorem~\ref{MainTheoremDescDiscretePart}) 
gives a complete description of  very badly approximable points  for
$(\mathbb{R}^{2}, S^1,  S^1 \cap \mathbb{Q}^{2}, H)$.

After an initial version of the present paper was posted in the arxiv server, 
Moshchevitin brought to our attention a paper 
\cite{Kop85} by Kopetzky.
In this paper, Kopetzky connects prior results of A.~Schmidt in \cite{Sch75a} and \cite{Sch75b} with intrinsic Diophantine approximation of $S^1$ and deduces a statement which implies the same result as our main theorem.
See the first paragraph in \S\ref{SecLiterature} below.

One of the differences between Kopetzky's methods and ours is that we can explicitly obtain digit expansions of badly approximable points.
For example, our method can be used to find 
maximal gaps beyond its smallest accumulation point in $\LLL(S^1)$. 
Also, it is recently shown in \cite{CCGW} that tools developed here can be adapted to give a complete description of very badly approximable points in 
$(\mathbb{C}, S^1, S^1 \cap \mathbb{Q}(\sqrt{-3}), H)$.

To give a more detailed explanation of our main theorem, let $P=(\alpha, \beta)$ be a point in the unit quarter circle $\QQQ$ and choose a rational point $Z = (\frac ac, \frac bc)$ in $\QQQ$. 
We define
\begin{equation}\label{DefinitionDeltaPZ}
	\delta(P; Z) = c\cdot \sqrt{%
		\left( \alpha - \frac ac \right)^2 + 
		\left( \beta - \frac bc \right)^2
	}.
\end{equation}
Although $\delta(P; Z)$ is not the same as $\hat\delta(P; (a, b))$, they become arbitrarily close as $c = \sqrt{a^2 + b^2}$ becomes large
and therefore 
\begin{equation}\label{DefinitionDelta}
	L(P) = \liminf_{c\to\infty} \delta(P; Z)^{-1},
\end{equation}
which is the Lagrange number with respect to 
$(\mathbb{R}^{2}, S^1,  S^1 \cap \mathbb{Q}^{2}, H)$.
Note that we are using the usual Euclidean distance in $\mathbb{R}^2$ in the definition \eqref{DefinitionDeltaPZ}.

In the present paper, we will call $(x; y_1, y_2)$ a \emph{Markoff triple} if $(x; y_1, y_2)$ is a positive integer triple satisfying
\begin{equation}\label{MarkoffEquationIntro}
2x^2 + y_1^2 + y_2^2 = 4 x y_1 y_2.
\end{equation}

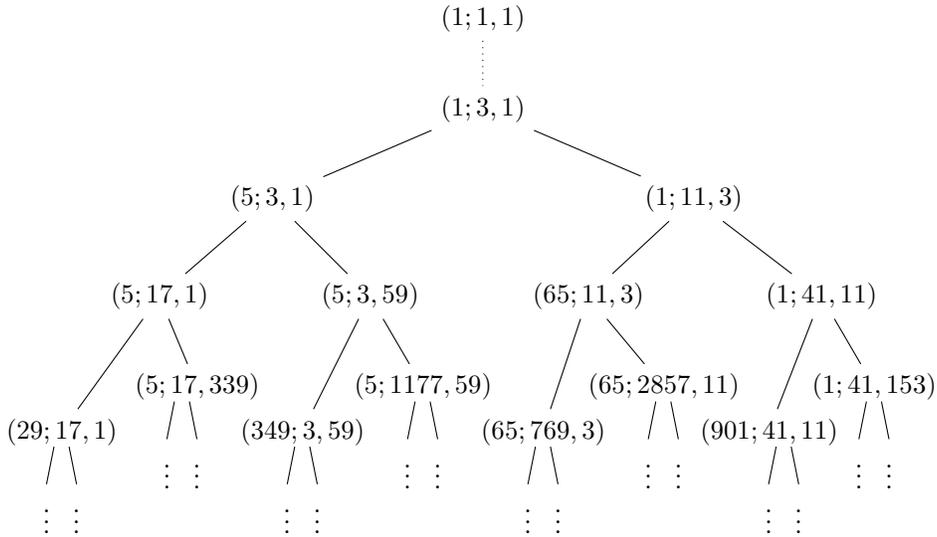
\begin{figure}
\begin{center}
\tikzset{triarrow/.pic={
    \draw (-0.1, 0) -- (-0.2, -0.5) node[below]{$\vdots$};
    \draw (0.1, 0) -- (0.2, -0.5) node[below]{$\vdots$};
  }
}
\begin{tikzpicture}[xscale=0.9]%[->, >=stealth', auto, xscale=1.2]
\node (base) at (0, 5.2) {$(1; 1, 1)$};

\node (0) at (0, 4) {$( 1; 3, 1)$};

\node (f) at (-2.8, 2.8) {$(1; 3, 11)$};
\node (g) at (2.8, 2.8) {$(5; 3, 1)$};

\node (f0) at (-4.3, 1.5) {$(1; 41, 11)$};
\node (f1) at (-1.5, 1.5) {$(65; 3, 11)$};

\node (g0) at (1.4, 1.5) {$(5; 3, 59)$};
\node (g1) at (4.5, 1.5) {$(5; 17, 1)$};

\node (f01) at (-5.6, -.3) {$(1; 41, 153)$};
\pic at (-5.6, -0.5) {triarrow};
\node (f00) at (-3.8, .3) {$(901; 41, 11)$};
\pic at (-4, 0.1) {triarrow};

\node (f10) at (-2.4, -.3) {$(65; 2857, 11)$};
\pic at (-2.4, -0.5) {triarrow};
\node (f11) at (-.8, .3) {$(65; 3, 769)$};
\pic at (-0.8, 0.1) {triarrow};

\node (g00) at (.8, -.3) {$(349; 3, 59)$};
\pic at (0.8, -0.5) {triarrow};
\node (g01) at (2.4, .3) {$(5; 1177, 59 )$};
\pic at (2.4, 0.1) {triarrow};

\node (g10) at (3.8, -.3) {$(5; 17, 339)$};
\pic at (4, -0.5) {triarrow};
\node (g11) at (5.2, .3) {$(29; 17,1)$};
\pic at (5.2, 0.1) {triarrow};

\draw[dotted] (base) to (0);

\draw (0) to node[font=\footnotesize]{} (f);
\draw (0) to node[below right, font=\footnotesize]{} (g);

\draw (f) to node[font=\footnotesize]{} (f0); 
\draw (f) to node[right, font=\footnotesize]{} (f1); 

\draw (g) to node[font=\footnotesize]{} (g0); 
\draw (g) to node[right, font=\footnotesize]{} (g1); 

\draw (f0) to node[font=\footnotesize]{} (f00); 
\draw (f0) to node[right, font=\footnotesize]{} (f01); 

\draw (f1) to node[font=\footnotesize]{} (f10); 
\draw (f1) to node[right, font=\footnotesize]{} (f11); 

\draw (g0) to node[font=\footnotesize]{} (g00); 
\draw (g0) to node[right, font=\footnotesize]{} (g01); 

\draw (g1) to node[font=\footnotesize]{} (g10); 
\draw (g1) to node[right, font=\footnotesize]{} (g11); 
\end{tikzpicture}
\end{center}\caption{The Markoff tree for $2x^2 + y_1^2 + y_2^2 = 4 xy_1y_2$\label{Mtree}}
\end{figure}
The triple $(1; 1, 1)$ is said to be the \emph{singular} Markoff triple and all others are said to be \emph{nonsingular} Markoff triples. 
As in the classical theory of Markoff in \cite{Mar79} and \cite{Mar80}, all nonsingular Markoff triples form an infinite and complete binary tree whose root is $(1; 3, 1)$ (see Figure~\ref{Mtree}). 
A set of recursive rules for generating this tree is stated in \S\ref{SubSectionMarkoffTriples}.
Using these rules, we obtain two sequences 
\begin{equation}\label{DefMx}
    \mathcal{M}_x = \{ x \mid
(x; y_1, y_2) \text{ is a Markoff triple}
\} 
=\{ 1, 5, 29, 65, 169, 349, \dots\},
\end{equation}
and
\begin{equation}\label{DefMy}
\begin{aligned}
\mathcal{M}_y &= \{ \max\{y_1, y_2\}\mid 
(x; y_1, y_2)
\text{ is a Markoff triple}
\}\\
&=\{ 1, 3, 11, 17, 41, 59, \dots\}. 
\end{aligned}
\end{equation}
We are now ready to state our main theorem.
\begin{theorem}[Main Theorem. Also \cite{Kop85} and \cite{Sch75b}]\label{MainTheoremDescDiscretePart}
Let $\LLL(S^1)$ be the Lagrange spectrum with respect to $(\mathbb{R}^{2}, S^1,  S^1 \cap \mathbb{Q}^{2}, H)$.
Then
\[
\LLL(S^1) \cap [0, 2)
= \\
\left\{ \sqrt{ 4 - \frac{1}{x^2}}\, 
| \, x \in \mathcal{M}_x\right \} 
 \cup
\left \{ \sqrt{ 4 - \frac{2}{y^2}} \, | \, 
y \in \mathcal{M}_y \right\}.
\]
\end{theorem}

\subsection{Outline of proof and organization of the paper}
As mentioned above,
Theorem~\ref{MainTheoremDescDiscretePart} is an unmistakable analogue of Markoff's classical theorem on $\LLL(\mathbb{R})$.
The theory of continued fraction plays a central role in Markoff's theory.
Therefore a natural starting point is to define digit expansions for points $P\in\QQQ$, which will be our counterpart to continued fraction expansions of real numbers.
This is done by Romik in \cite{Rom08}.
In essence, to each $P\in\QQQ$, one can attach an infinite sequence with values in $\{1,2,3\}$, which we will call 
\emph{a Romik digit expansion}, or simply, a \emph{digit expansion of} $P$.
This construction will be thoroughly reviewed in \S\ref{SubsecRomikDigitExpansion}.

Markoff's theorem in \cite{Mar79} and \cite{Mar80} can be rephrased by saying that a real number $\gamma$ is very badly approximable if any only if the continued fraction expansion of $\gamma$ is eventually periodic and its minimal period can be written as a \emph{Christoffel word} on an alphabet $\{a, b \}$ under the substitution
\[
a = 2\, 2 \quad \text{and} \quad
b = 1\, 1.
\]
See \S\ref{SecChristoffel} to recall definition and basic properties of Christoffel words.
This formulation of Markoff's theorem using Christoffel words can be traced back to Cohn \cite{Coh55}. 
More recently, Bombieri presented an elegant and self-contained exposition of this approach in \cite{Bom07}. Shortly after this, Reutenauer independently gave a short proof of this in \cite{Reu09} using known properties of Christoffel words and the theory of Sturmian words. Also, see \cite{BLRS} and \cite{Aig13}.

To state our results for Romik digit sequences, let $w$ be a finite word on a two-letter alphabet $\{a, b\}$ (see \S\ref{SecInfiniteSeqDoublyInfiniteSeq}).
We will say that $w$ is \emph{even} if $w$ contains an even number of $b$'s and \emph{odd} if $w$ contains an odd number of $b$'s.
For each $w$, we define a finite word $\jmath(w)$ on a three-letter alphabet $\{ a, b, a^{\vee} \}$ as follows.
Write $w = l_1 \cdots l_k$ with $l_j \in \{ a, b \}$.
For each $j = 1, 2, \dots, k$, we let $t = t(j)$ be the 
the number of occurrences of $b$ in the sequence $l_1, \dots, l_{j-1}$. 
Then we define $\jmath(w) = l'_1 l'_2 \cdots l_k'$ with $l_j' \in \{ a, b, a^{\vee} \}$ where 
\[
l'_j = 
\begin{cases}
a & \text{ if $l_j = a$ and $t(j)$ is even},\\
a^{\vee} & \text{ if $l_j = a$ and $t(j)$ is odd},\\
b & \text{ if } l_j = b.\\
\end{cases}
\]
For example, $\jmath(ababba) = aba^{\vee}bba^{\vee}$ and $\jmath(babbba) = ba^{\vee}bbba$.
From this definition, it follows that
$\jmath(w)$ does not contain any of the following words as a subword:
    \begin{equation}\label{ForbiddenWords0}
    ab^{2k}a^{\vee}, \quad
    a^{\vee} b^{2k}a,\quad
    ab^{2k+1}a,\quad
    a^{\vee}b^{2k+1}a^{\vee}
    \end{equation}
    for any $k\ge0$.

The key step in proving our main theorem is to characterize very badly approximable points by their periods in digit expansions in terms of Christoffel words.
More precisely, we will show in Theorem~\ref{BombieriThm15} that $P= (\alpha, \beta) \in\QQQ$ is very badly approximable if and only if 
	\begin{enumerate}[font=\upshape, label=(\roman*)]
	\item the digit expansion of $P$
	ends with either $2^{\infty}:= 222\cdots$ or
	$(31)^{\infty}:= 313131\cdots$, or
	\item  there exists a Christoffel word $w$ in $\{a, b\}$ such that 
the digit expansion of either $P$ or $P^{\vee}:= (\beta, \alpha)$ is eventually periodic and its minimal period is equal to
\begin{equation}\label{EqPeriodChristoffel}
\begin{cases}
\jmath(w) & \text{ if $w$ is even} \\
\jmath(w)(\jmath(w))^{\vee} & \text{ if $w$ is odd} \\
\end{cases}
\end{equation}
via the substitution
\begin{equation}\label{SubstitutionRule0}
a = 3\, 1, \quad
b = 2,  \quad
a^{\vee} = 1\, 3.
\end{equation}
Here, $(\jmath(w))^{\vee}$ is by definition the word in $\{ a, b, a^{\vee}\}$ obtained by attaching $\vee$ to each of the letters in $\jmath(w)$ subject to the rule
\[
    b^{\vee} = b \quad \text{ and } \quad (a^{\vee})^{\vee} = a.
\]
	\end{enumerate}

We present in \S\ref{SecRomikSequence} and \S\ref{SecCombinatorics} an adaptation of Bombieri's masterful exposition in \cite{Bom07}.
In \S\ref{SecRomikSequence}, we study combinatorial properties of a doubly infinite digit sequence that arises from a very badly approximable point. 
Based on these properties, we deduce in \S\ref{SecCombinatorics} that such a doubly infinite digit sequence can be always associated under the substitution rule \eqref{SubstitutionRule0} to a purely periodic doubly infinite word on $\{ a, b, a^{\vee}\}$ with its period given in the form \eqref{EqPeriodChristoffel}.

After we characterize periods of very badly approximable points, 
we review in \S\ref{SubSectionMarkoffTriples} some known structure of the Markoff tree and the Christoffel tree. 
We use the fact that these two trees are isomorphic (as graphs) to prove that all the periods of very badly approximable points come from Christoffel words, which is a critical step in proving Theorem~\ref{MainTheoremDescDiscretePart}.
\begin{table}
    \caption{The 10 smallest values of $L(P)$}
    \begin{tabular}{@{} llll @{}} 
    \toprule
	   $w$ &
	   \text{Minimal period} &
	   \text{Markoff number} &
	   $L(P)$
	   \\ 
   \midrule
   $b$ &
   2 &
   $y = 1$ &
  $ \sqrt2 = 1.414213562 \dots$
	   \\
	$ a $ &
	  31 &
	$  x = 1$ &
	$   \sqrt3 = 1.732050808 \dots $\\
	$    ab $ &
	    312 132  &
	$    y = 3$ &
	$   \frac{\sqrt{34}}{3} = 1.943650632 \dots$ \\
	$    abb$ &
	   3122 & 
	$    x = 5$ &
	$   \frac{3\sqrt{11}}{5} =1.989974874 \dots $ \\
	$    	aab$ &
	    31312 13132&
	$   y = 11$ &
	$   \frac{\sqrt{482}}{11} = 1.995863491 \dots $\\
	$    abbb$ &
	   31222 13222 &
	$    y = 17$ &
	$   \frac{\sqrt{1154}}{17} = 1.998269147 \dots $\\
	$    aaab$ &	
	   3131312 1313132 &
	$    y = 41$ &
	$   \frac{\sqrt{6722}}{41} = 1.999702536 \dots  $\\
	$    abbbb$ &
	   312222 &
	$    x = 29$ &
	$   \frac{\sqrt{3363}}{29} = 1.999702713 \dots $ \\
	$    ababb$ & 
	   3121322 1323122 &
	$    y = 59$ &
	$   \frac{\sqrt{13922}}{59} = 1.999856358 \dots $ \\
	$    aabab$ &
	   31312132 &
	$    x = 65$ &
	$   \frac{\sqrt{16899}}{65} = 1.999940828 \dots $ \\
      \bottomrule
   \end{tabular}
   \label{tab:10smallestL}
\end{table}

Table~\ref{tab:10smallestL} shows the 10 smallest Lagrange numbers and related quantities.
When we convert minimal periods in Table~\ref{tab:10smallestL} to the corresponding points $P\in \QQQ$, we obtain the list in Table~\ref{tab:Top10}.

\subsection{Related literature}\label{SecLiterature}
  In \cite{Sch75a}, A.~Schmidt develops a new approach to the problem of generalizing the theory of Diophantine approximation to complex numbers.
 He studies, among others, a certain version of \emph{a Markoff spectrum}, which is the set of normalized minimum values, 
 called \emph{$C$-minimum}, 
 of indefinite binary quadratic forms with real coefficients on a lattice in $\mathbb{R}^2$.
Then he proves that (its initial discrete part of) his Markoff spectrum is given by the same expression as in the right hand side in Theorem~\ref{MainTheoremDescDiscretePart}. 
See Chapter 5 in \cite{Sch75a} and \cite{Sch75b}.
In addition, a similar Markoff spectrum was studied by Vulakh, which was later called \emph{the Markoff spectrum on the sublattice of index 2} by Malyshev in \cite{Mal77}.
Vulakh proved that the initial discrete part of this spectrum is also the same as the right hand side in Theorem~\ref{MainTheoremDescDiscretePart}.
It is Kopetzky in \cite{Kop85} who made the connection explicit between $\LLL(S^1)$ and the Markoff spectrum of \emph{$C$-minimal forms} of A.~Schmidt.

According to \cite{Ser85a}, one can think of partial quotients of continued fractions as cutting sequences arising from geodesics on the modular surface $\mathbb{H}/\mathrm{SL}_2(\mathbb{Z})$. 
Similarly, one can interpret Romik's digit sequence of $P\in \QQQ$ as a cutting sequence on $\mathbb{H}/\Gamma(2)$.
This has already been observed in Theorem 5 in \cite{Rom08} by Romik.
When interpreted this way, the Romik digit expansion can be shown to be related to \emph{even continued fractions}. 
This point of view is emphasized in a recent work \cite{KLL20}.
For general discussion on even continued fractions and related cutting sequences, see \cite{SW16} and \cite{BM18}.

It is also possible to interpret Romik's digit sequences as cutting sequences of geodesics in the \emph{hyperboloid model} of the hyperbolic surface.
For instance, the three matrices $M_1, M_2, M_3$ defined in \eqref{DefinitionMs} are elements of the orthogonal group $O(2,1)$, which acts as an (orientation-preserving or orientation-reversing) isometry group of the hyperboloid model $x^2 + y^2 - z^2 = -1$. In this sense, the present paper is a natural continuation of our prior work in \cite{CNT} and \cite{CK18}.
A recent paper \cite{Pan19} by Panti is another example where the hyperboloid model is emphasized over the upper half plane model.

Finally, we mention a series of papers \cite{AA13}, \cite{AAR16}, and \cite{AR17} written by Abe, Aitchison, and Rittaud. Using geometric and combinatorial means, they obtained certain Lagrange spectra which are very much related to ours. 

\subsection*{Acknowledgements}
The authors are thankful to Brittany Gelb, whose careful reading improved an initial version of the paper.
They are also grateful to Nikolay Moshchevitin for informing them of Kopetzy's work \cite{Kop85}
and to Yann Bugeaud for his comments and French translation of the title and abstract.
Finally, they wish to thank the anonymous referee for providing many helpful comments.

\section{Romik's dynamical system and Perron's formula}\label{RomikWay}

\subsection{Romik's digit expansions}\label{SubsecRomikDigitExpansion}
As in the introduction, we let
\[
	\QQQ = \{ (x, y) \in \mathbb{R}^2 \mid x^2 + y^2 = 1, \text{ and } x, y\ge 0 \}.
\]
Following  \cite{Rom08},
we define a map $\mathcal{T}:\QQQ \longrightarrow \QQQ$
by
\begin{equation}\label{DefinitionT}
\TTT(x, y) = \left(
\frac{|2 - x - 2y|}{3 - 2x - 2y},
\frac{|2 - 2x - y|}{3 - 2x - 2y}
\right).
\end{equation}
To each $P = (x, y) \in \QQQ$, we assign a \emph{Romik digit} $d(P)$ to be
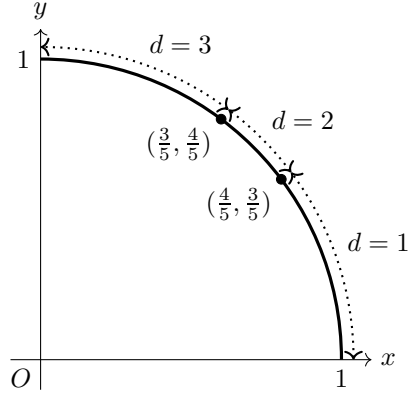
\begin{figure}
	\begin{center}
		\begin{tikzpicture}[scale=0.8]
			\node[below left] at (0, 0) {$O$};
			\draw[->] (-0.5, 0) -- (5.5, 0) node[right] {$x$};
			\draw[->] (0, -0.5) -- (0, 5.5) node[above] {$y$};
			\draw[very thick] (5, 0) node [below] {$1$}
			arc (0:90:5) node[left] {$1$};

			\draw[fill] (3, 4) circle (0.08) node[below left] {$(\frac35, \frac45)$};
			\draw[fill] (4, 3) circle (0.08) node[below left] {$(\frac45, \frac35)$};
			\draw[<->, dotted, thick] (5.2, 0) 
			arc (0: 36.8:5.2) node[midway, above right] {$d=1$};
			\draw[<->, dotted, thick] ([shift=(36.8:5.2)]0, 0) 
			arc (36.8: 53.1: 5.2) node[midway, above right] {$d=2$};
			\draw[<->, dotted, thick] ([shift=(53.1:5.2)]0, 0) 
			arc (53.1: 90: 5.2) node[midway, above right] {$d=3$} ;
		\end{tikzpicture}
		\caption{Romik digit of $P$ \label{DigitPicture}}
	\end{center}
\end{figure}
\begin{equation}\label{DefinitionDigit}
d(P) = \begin{cases}
1 & \text{ if } \frac45 \le x \le 1, \\
2 & \text{ if } \frac35 \le x \le \frac45, \\
3 & \text{ if } 0 \le x \le \frac35, \\
\end{cases}
\end{equation}
(see Figure~\ref{DigitPicture})
and the \emph{$j$-th digit of $P$} is subsequently defined to be
$d_j = d(\TTT^{j-1}(P))$ for $j=1, 2, \dots.$ 
The resulting sequence $\{ d_j \}_{j=1}^{\infty}$ will be called \emph{the Romik digit expansion} of $P$ and we write
\begin{equation}\label{RomikDigitExpansion}
	P = (x, y) = [d_1, d_2, \dots ]_{\QQQ}.
\end{equation}
The map $\TTT$ defines a dynamical system $(\QQQ, \mathcal{T})$, which we call \emph{Romik's system}, and $\mathcal{T}$ shifts each digit sequence to the left, so that
\[
	\TTT^k(P) =  (\overbrace{\TTT \circ \cdots \circ \TTT}^{k \text{ times}}) (P)
	= [d_{k+1}, d_{k+2}, \dots ]_{\QQQ}.
\]
For instance, 
\[
	(\tfrac1{\sqrt2}, \tfrac1{\sqrt2})= [2, 2, \dots]_{\QQQ} \ \text{ and } \ (\tfrac12, \tfrac{\sqrt3}2) = [3, 1, 3, 1, \dots]_{\QQQ}.
\]
We denote by $1^{\infty}$ and $3^{\infty}$ the infinite successions of 1's and 3's. 
Since the points $(1, 0)$ and $(0, 1)$ are fixed by $\TTT$, we have
\[
	(1, 0) = [1, 1 ,1, \dots]_{\QQQ}=[1^{\infty}]_{\QQQ} 
	\ \text{ and } \  (0, 1) = [3, 3 ,3, \dots]_{\QQQ}= [3^{\infty}]_{\QQQ}.
\]

We allow each of the two boundary points $(\frac45, \frac35)$ and $(\frac35, \frac45)$ to have \emph{two} valid digits $\{ 1,2\}$ and $\{2,3\}$ respectively.
As a result, all rational points on $\QQQ$ except for $(1, 0)$ and $(0, 1)$ will have two valid digit expansions of the forms    
\[
[\dots, 2, 1^{\infty}]_{\QQQ} \ \text{ and } \ [\dots, 3, 1^{\infty}]_{\QQQ},
\]
or
\[
[\dots, 1, 3^{\infty}]_{\QQQ} \ \text{ and } \ [\dots, 2, 3^{\infty}]_{\QQQ}.
\]
For example,
\[
	(\tfrac35, \tfrac45) = [2, 1^{\infty}]_{\QQQ} \ \text{ and } \ 
	 [3, 1^{\infty}]_{\QQQ},
\]
and
\[
	(\tfrac{12}{13}, \tfrac{5}{13}) = [1, 1, 3^{\infty}]_{\QQQ} \ \text{ and } \ 
	 [1, 2, 3^{\infty}]_{\QQQ}.
\]

The map $\TTT$ originates from an old theorem on trees of \emph{primitive Pythagorean triples}, that is, triples $(a, b, c)$ of (pairwise) coprime positive integers $a, b, c$ with $a^2 + b^2 = c^2$, which is often attributed to 
Berggren \cite{Ber34} and Barning \cite{Bar63}.
The theorem says that, if $(a, b, c)$ is a primitive Pythagorean triple, there exists a unique sequence $[d_1, \dots, d_k]$ of digits $d_j \in \{ 1, 2, 3 \}$ such that
\[
\begin{pmatrix}
a \\ b \\ c
\end{pmatrix}
=
M_{d_1}
\cdots
M_{d_k}
\begin{pmatrix}
3 \\ 4 \\ 5
\end{pmatrix}
\text{ or }
\begin{pmatrix}
a \\ b \\ c
\end{pmatrix}
=
M_{d_1}
\cdots
M_{d_k}
\begin{pmatrix}
4 \\ 3 \\ 5
\end{pmatrix}
\]
where $M_1, M_2, M_3$ are defined to be
\begin{equation}\label{DefinitionMs}
M_1=
\begin{pmatrix}
-1 & 2 & 2 \\
-2 & 1 & 2 \\
-2 & 2 & 3 \\
\end{pmatrix},
\quad
M_2=
\begin{pmatrix}
1 & 2 & 2 \\
2 & 1 & 2 \\
2 & 2 & 3 \\
\end{pmatrix},
\quad
M_3=
\begin{pmatrix}
1 & -2 & 2 \\
2 & -1 & 2 \\
2 & -2 & 3 \\
\end{pmatrix}.
\end{equation}
As a result of this theorem,  the set of all primitive Pythagorean triples forms directed ternary trees 
(see Figure~\ref{PythagoreanTree})
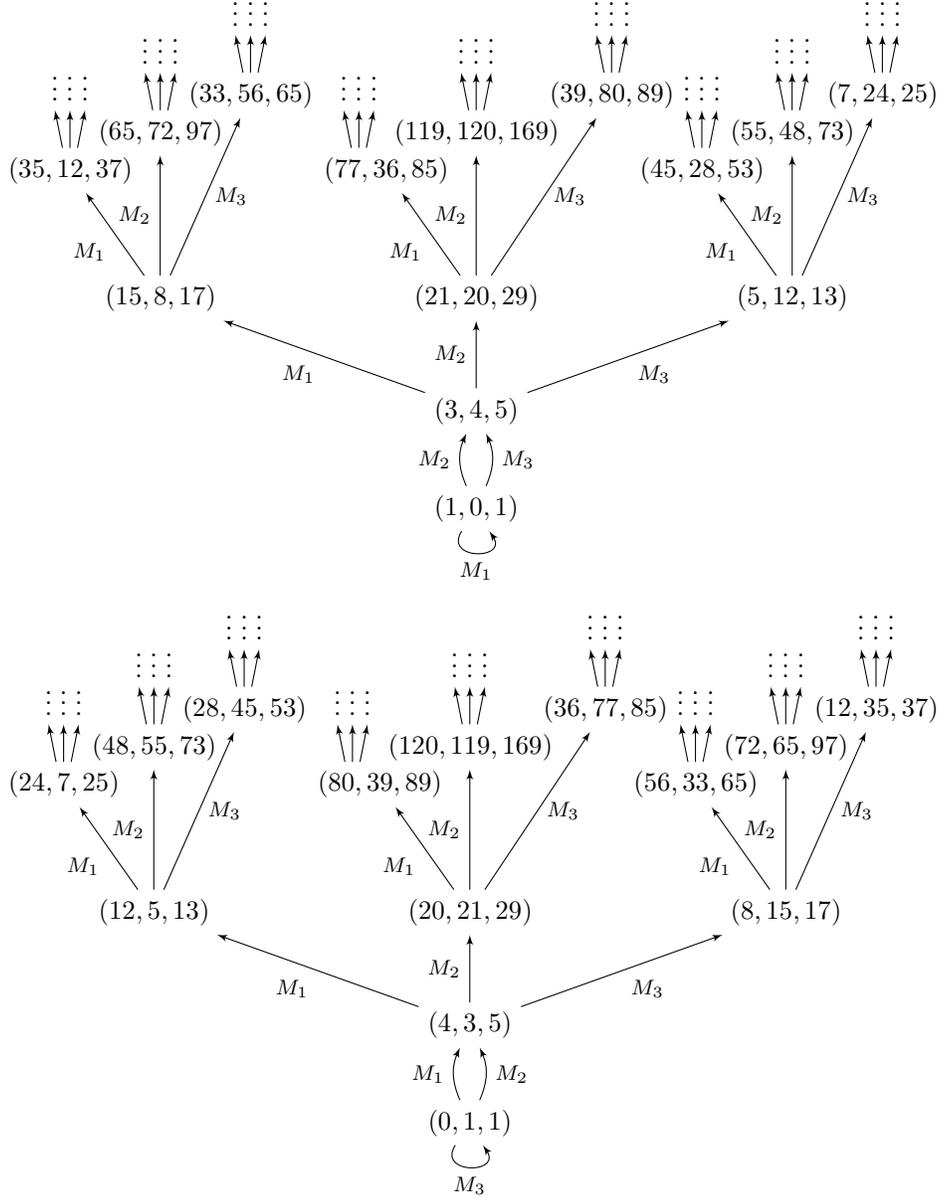
\begin{figure}
\begin{center}
\tikzset{triarrow/.pic={
    \draw (0, 0) -- (0, 0.5) node[above]{$\vdots$};
    \draw (-0.1, 0) -- (-0.2, 0.5) node[above]{$\vdots$};
    \draw (0.1, 0) -- (0.2, 0.5) node[above]{$\vdots$};
  }
}
\begin{tikzpicture}[->, >=latex', auto, xscale=1]

\node (subbase) at (0, -1.5) {$(1, 0, 1)$};
\node (base) at (0, -0.2) {$(3, 4, 5)$};

\node (1f1) at (-3.5, 1.3) {$(15, 8, 17)$};
\node (1f2) at (0, 1.3) {$(21, 20, 29)$};
\node (1f3) at (3.5, 1.3) {$(5, 12, 13)$};

\pic at (-4.5, 3.3) {triarrow};
\pic at (-3.5, 3.8) {triarrow};
\pic at (-2.5, 4.3) {triarrow};

\node (2f1) at (-4.5, 3) {$(35, 12, 37)$};
\node (2f2) at (-3.5, 3.5) {$(65, 72, 97)$};
\node (2f3) at (-2.5, 4) {$(33, 56, 65)$};

\node (2f4) at (-1, 3) {$(77, 36, 85)$};
\node (2f5) at (0, 3.5) {$(119, 120, 169)$};
\node (2f6) at (1.5, 4) {$(39, 80, 89)$};

\pic at (-1.3, 3.3) {triarrow};
\pic at (0, 3.8) {triarrow};
\pic at (1.5, 4.3) {triarrow};

\node (2f7) at (2.5, 3) {$(45, 28, 53)$};
\node (2f8) at (3.5, 3.5) {$(55, 48, 73)$};
\node (2f9) at (4.5, 4) {$(7, 24, 25)$};

\pic at (2.5, 3.3) {triarrow};
\pic at (3.5, 3.8) {triarrow};
\pic at (4.5, 4.3) {triarrow};

\draw (subbase) to[bend left=20] node[left, font=\footnotesize]{${M_2}$} (base);
\draw (subbase) to[bend right=20] node[right, font=\footnotesize]{${M_3}$} (base);
\draw (subbase) to [out=-120,in=-60,loop,looseness=4] 
	node[below, font=\footnotesize]{${M_1}$} (subbase);

\draw (base) to node[font=\footnotesize]{${M_1}$} (1f1);
\draw (base) to node[font=\footnotesize]{${M_2}$} (1f2);
\draw (base) to node[below right, font=\footnotesize]{${M_3}$} (1f3);

\draw (1f1) to node[font=\footnotesize]{${M_1}$} (2f1); 
\draw (1f1) to node[font=\footnotesize]{${M_2}$} (2f2); 
\draw (1f1) to node[right, font=\footnotesize]{${M_3}$} (2f3); 

\draw (1f2) to node[font=\footnotesize]{${M_1}$} (2f4); 
\draw (1f2) to node[font=\footnotesize]{${M_2}$} (2f5); 
\draw (1f2) to node[right, font=\footnotesize]{${M_3}$} (2f6); 

\draw (1f3) to node[font=\footnotesize]{${M_1}$} (2f7); 
\draw (1f3) to node[font=\footnotesize]{${M_2}$} (2f8); 
\draw (1f3) to node[right, font=\footnotesize]{${M_3}$} (2f9); 
\end{tikzpicture}
\begin{tikzpicture}[->, >=latex', auto, xscale=1]
\node (subbase) at (0, -1.5) {$(0, 1, 1)$};
\node (base) at (0, -0.2) {$(4, 3, 5)$};

\node (1f1) at (-3.5, 1.3) {$(12, 5, 13)$};
\node (1f2) at (0, 1.3) {$(20, 21, 29)$};
\node (1f3) at (3.5, 1.3) {$(8, 15, 17)$};

\node (2f1) at (-4.5, 3) {$(24, 7, 25)$};
\node (2f2) at (-3.5, 3.5) {$(48, 55, 73)$};
\node (2f3) at (-2.5, 4) {$(28, 45, 53)$};

\pic at (-4.5, 3.3) {triarrow};
\pic at (-3.5, 3.8) {triarrow};
\pic at (-2.5, 4.3) {triarrow};

\node (2f4) at (-1, 3) {$(80, 39, 89)$};
\node (2f5) at (0, 3.5) {$(120, 119, 169)$};
\node (2f6) at (1.5, 4) {$(36, 77, 85)$};

\pic at (-1.3, 3.3) {triarrow};
\pic at (0, 3.8) {triarrow};
\pic at (1.5, 4.3) {triarrow};

\node (2f7) at (2.5, 3) {$(56, 33, 65)$};
\node (2f8) at (3.5, 3.5) {$(72, 65, 97)$};
\node (2f9) at (4.5, 4) {$(12, 35, 37)$};

\pic at (2.5, 3.3) {triarrow};
\pic at (3.5, 3.8) {triarrow};
\pic at (4.5, 4.3) {triarrow};

\draw (subbase) to[bend left=20] node[left, font=\footnotesize]{${M_1}$} (base);
\draw (subbase) to[bend right=20] node[right, font=\footnotesize]{${M_2}$} (base);
\draw (subbase) to [out=-120,in=-60,loop,looseness=4] 
	node[below, font=\footnotesize]{${M_3}$} (subbase);

\draw (base) to node[font=\footnotesize]{${M_1}$} (1f1);
\draw (base) to node[font=\footnotesize]{${M_2}$} (1f2);
\draw (base) to node[below right, font=\footnotesize]{${M_3}$} (1f3);

\draw (1f1) to node[font=\footnotesize]{${M_1}$} (2f1); 
\draw (1f1) to node[font=\footnotesize]{${M_2}$} (2f2); 
\draw (1f1) to node[right, font=\footnotesize]{${M_3}$} (2f3); 

\draw (1f2) to node[font=\footnotesize]{${M_1}$} (2f4); 
\draw (1f2) to node[font=\footnotesize]{${M_2}$} (2f5); 
\draw (1f2) to node[right, font=\footnotesize]{${M_3}$} (2f6); 

\draw (1f3) to node[font=\footnotesize]{${M_1}$} (2f7); 
\draw (1f3) to node[font=\footnotesize]{${M_2}$} (2f8); 
\draw (1f3) to node[right, font=\footnotesize]{${M_3}$} (2f9); 
\end{tikzpicture}
\end{center}
\caption{Trees of Pythagorean triples\label{PythagoreanTree}} 
\end{figure}
with an edge from $(a', b', c')$ to $(a, b, c)$ whenever
\begin{equation}\label{Edge}
\begin{pmatrix} a \\ b \\ c \end{pmatrix}
	=
	M_d
\begin{pmatrix} a' \\ b' \\ c' \end{pmatrix}
\end{equation}
for $d\in\{1, 2, 3 \}$.
We find it convenient to add $(1, 0, 1)$ and $(0, 1, 1)$ to the trees, even though they are not primitive Pythagorean triples.

A simple calculation shows that $(a, b, c)$ and $(a', b', c')$ satisfy \eqref{Edge} if and only if $\TTT(\frac ac, \frac bc)=(\frac{a'}{c'}, \frac{b'}{c'})$ and, when this happens, we have $d=d(\frac{a}{c}, \frac{b}{c})$ (see Proposition~\ref{DigitsAndMatrixMultiplication} below).
As a consequence, we can ``read off'' the digit expansion of $(\frac ac, \frac bc)$ by locating $(a, b, c)$ in a tree in Figure~\ref{PythagoreanTree}.
For example, suppose that $(a, b, c)$ is connected to $(3, 4, 5)$ in the tree via $M_{d_1}, \dots, M_{d_k}$: 
\begin{center}
\begin{tikzpicture}[->, >=latex', auto, xscale=1.2]
	\node [] (A) {$(a, b, c)$};
	\node [right=of A] (B) {$\cdots$};
	\node [right=of B] (C) {$(3, 4, 5)$};
	\node [right=of C] (D) {$(1, 0, 1)$};
\draw (B) to node[above, font=\footnotesize]{${M_{d_1}}$} (A);
\draw (C) to node[above, font=\footnotesize]{${M_{d_k}}$} (B);
\draw (D) to [bend right=20] node[above, font=\footnotesize]{${M_{2}}$} (C);
\draw (D) to [bend left=20] node[below, font=\footnotesize]{${M_{3}}$} (C);
\draw (D) to [out=20,in=-20,loop,looseness=4] 
	node[right, font=\footnotesize]{${M_1}$} (D);
\end{tikzpicture}
\end{center}
Then the digit expansion of $(\frac ac, \frac bc)$ is obtained by tracing $(a, b, c)$ back to $(3, 4, 5)$ and then to $(1, 0, 1)$:
\[
	(\tfrac ac, \tfrac bc) = 
	[d_1, \dots, d_k, 2, 1^{\infty}]_{\QQQ}
	\text{ and }
	[d_1, \dots, d_k, 3, 1^{\infty}]_{\QQQ}.
\]

The following proposition explicitly relates $\TTT$ and actions of $M_1, M_2, M_3$, which will be used frequently later.
\begin{proposition}\label{DigitsAndMatrixMultiplication}
	Let
	\[
		P = (\alpha, \beta) = [d_1, d_2, \dots, d_k, \dots]_{\QQQ},
	\]
	and $P' = \TTT^k(P) = (\alpha', \beta')$ for some $k\ge 1$.
	Then, the two vectors
	\[
		M_{d_k}^{-1} \cdots M_{d_1}^{-1}
	\begin{pmatrix} 
		\alpha \\ \beta \\ 1
	\end{pmatrix}
	\text{ and } 
	\begin{pmatrix} 
		\alpha' \\ \beta' \\ 1
	\end{pmatrix}
	\]
    are non-zero (positive) scalar multiples of one another.
\end{proposition}
\begin{proof}
    It suffices to show this for $k=1$, as the general case will then follow from an easy induction. 
	Suppose that $d = d(P) = 1$.
	Then,
	\[
		M_1^{-1}
	\begin{pmatrix} \alpha \\ \beta \\ 1 \end{pmatrix}
		= 
	\begin{pmatrix}
		-1 & -2 & 2 \\
	2 & 1 & -2 \\
	-2 & -2 & 3
	\end{pmatrix}
	\begin{pmatrix} \alpha \\ \beta \\ 1 \end{pmatrix}
	=
	(-2\alpha - 2\beta + 3)
	\begin{pmatrix}
	\alpha' \\ \beta' \\ 1
	\end{pmatrix},
	\]
	because 
	\[
		 (\alpha', \beta') = \TTT(\alpha, \beta) =
		 \left(
			 \frac{-\alpha- 2\beta+2}{-2\alpha - 2\beta+3},
			 \frac{2\alpha +\beta-2}{-2\alpha - 2\beta+3}
		 \right).
	 \]
	 This proves the case $d(P)=1$. The cases for $d(P)=2$ and $d(P)=3$ are similarly straightforward.
\end{proof}
\subsection{Cylinder sets and their boundary points}\label{SecCylinder}
Fix a sequence $\{d_1, \dots, d_k\}$ of $k$ Romik digits $d_j \in \{1, 2, 3 \}$.
Define a \emph{cylinder set} $\Cyl(d_1, \dots, d_k)$ \emph{of length} $k$ by 
\begin{equation}\label{CylindersetDefinition}
	\Cyl(d_1, \dots, d_k) = \{ P \in \QQQ: d_j = d(\TTT^{j-1}(P)) \text{ for } j = 1, \dots, k \}.
\end{equation}
(It is convenient to think of $\QQQ$ to be the cylinder set of length 0.)
The above definition needs to be interpreted carefully when $P$ has more than one valid digit expansion, that is, when $P$ is rational.
In that case, we say that $P \in \Cyl(d_1, \dots, d_k)$ whenever one of the two digit expansions satisfies the condition.
Topologically, $\Cyl(d_1, \dots, d_k)$ is a closed subarc of $\QQQ$, that is, a connected, closed subset of $\QQQ$. Figure~\ref{FigCylinderPicture} shows all cylinder sets of length 2.
\begin{figure}
	\begin{center}
		\begin{tikzpicture}[scale=0.5]
			\node[below left] at (0, 0) {$O$};
			\draw[->] (-1, 0) -- (11, 0) node[right] {$x$};
			\draw[->] (0, -1) -- (0, 11) node[above] {$y$};
			\draw[very thick] (10, 0) node [below] {$1$}
			arc (0:90:10) node[left] {$1$};

			\draw[fill] (6, 8) circle (0.1) node[below left] {$(\frac35, \frac45)$};
			\draw[fill] (8, 6) circle (0.1) node[below left] {$(\frac45, \frac35)$};

\draw[<->, dotted, thick] ([shift=(0.0:10.4)]0, 0) arc (0.0: 22.61986:10.4)
node[midway,  right] {\footnotesize{$\Cyl(1, 1)$}};
\draw[<->, dotted, thick] ([shift=(22.61986:10.4)]0, 0) arc (22.61986: 28.07249:10.4)
node[midway,  right] {\footnotesize{$\Cyl(1, 2)$}};
\draw[<->, dotted, thick] ([shift=(28.07249:10.4)]0, 0) arc (28.07249: 36.8699:10.4)
node[midway, above right] {\footnotesize{$\Cyl(1, 3)$}};
\draw[<->, dotted, thick] ([shift=(36.8699:10.4)]0, 0) arc (36.8699: 43.60282:10.4)
node[midway, above right] {\footnotesize{$\Cyl(2, 3)$}};
\draw[<->, dotted, thick] ([shift=(43.60282:10.4)]0, 0) arc (43.60282: 46.39718:10.4)
node[midway, above right] {\footnotesize{$\Cyl(2, 2)$}};
\draw[<->, dotted, thick] ([shift=(46.39718:10.4)]0, 0) arc (46.39718: 53.1301:10.4)
node[midway, above right] {\footnotesize{$\Cyl(2, 1)$}};
\draw[<->, dotted, thick] ([shift=(53.1301:10.4)]0, 0) arc (53.1301: 61.92751:10.4)
node[midway, above right] {\footnotesize{$\Cyl(3, 1)$}};
\draw[<->, dotted, thick] ([shift=(61.92751:10.4)]0, 0) arc (61.92751: 67.38014:10.4)
node[midway, above right] {\footnotesize{$\Cyl(3, 2)$}};
\draw[<->, dotted, thick] ([shift=(67.38014:10.4)]0, 0) arc (67.38014: 90:10.4)
node[midway, above ] {\footnotesize{$\Cyl(3, 3)$}};
	\end{tikzpicture}
	\caption{Cylinder sets of length 2\label{FigCylinderPicture}}
	\end{center}
\end{figure}
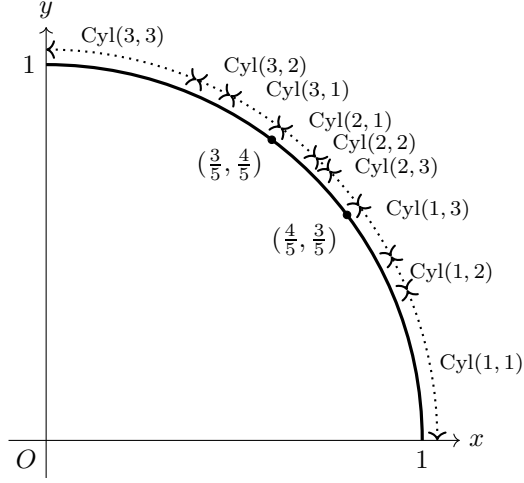

Denote by
$\mathbf{u}^{(1, 0)}$ and
$\mathbf{u}^{(0, 1)}$
the vectors
\begin{equation}\label{EqDefU}
	\mathbf{u}^{(1, 0)} =
	\begin{pmatrix} 1 \\ 0 \\ 1 \\ \end{pmatrix}
		\quad
		\text{ and }
		\quad
	\mathbf{u}^{(0, 1)} =
	\begin{pmatrix} 0 \\ 1 \\ 1 \\ \end{pmatrix}.
\end{equation}
Then the boundary points of $\Cyl(d_1, \dots, d_k)$ can be readily computed from the sequence $\{ d_1, \dots, d_k\}$ as follows.
Define
$\mathbf{z}^{(1, 0)}$
and 
$\mathbf{z}^{(0, 1)}$
(both of these depend on the sequence $\{ d_1, \dots, d_k\}$, which we suppress to lighten notations)
to be
\begin{equation}\label{V10}
	\mathbf{z}^{(1, 0)}
	=
\begin{pmatrix} 
a^{(1, 0)} \\
b^{(1, 0)} \\
c^{(1, 0)} \\
\end{pmatrix}
=
M_{d_1} M_{d_2} \cdots M_{d_k}
	\mathbf{u}^{(1, 0)} 
\end{equation}
and
\begin{equation}\label{V01}
	\mathbf{z}^{(0, 1)} =
\begin{pmatrix} 
a^{(0, 1)} \\
b^{(0, 1)} \\
c^{(0, 1)} \\
\end{pmatrix}
=
M_{d_1} M_{d_2} \cdots M_{d_k}
	\mathbf{u}^{(0, 1)} 
\end{equation}
Then the two points
\begin{equation}\label{DefinitionZk}
Z^{(1, 0)} =
	\left(
		\frac{a^{(1, 0)}}{c^{(1, 0)}}, 
		\frac{b^{(1, 0)}}{c^{(1, 0)}}
\right) 
\ \text{ and } \ 
Z^{(0, 1)} =
	\left(
		\frac{a^{(0, 1)}}{c^{(0, 1)}}, 
		\frac{b^{(0, 1)}}{c^{(0, 1)}}
\right) 
\end{equation}
are the boundary points of $\Cyl(d_1, \dots, d_k)$ and their digit expansions are given as
\begin{equation}\label{EqZboundary}
	Z^{(1, 0)}
	= [d_1, d_2, \dots, d_k, 1^{\infty} ]_{\QQQ}
	\ \text{ and } \
	Z^{(0, 1)}
	= [d_1, d_2, \dots, d_k, 3^{\infty} ]_{\QQQ}.
\end{equation}
This is a simple consequence of Proposition~\ref{DigitsAndMatrixMultiplication} and the equations \eqref{V10} and \eqref{V01}.

\subsection{Boundaries of cylinder sets contain best approximants}
\begin{definition}\label{DefDelta}
For a rational point $Z=(\frac ac, \frac bc) \in \QQQ$ with coprime integers $a,b,c$, the \emph{height} $\Ht(Z)$ of $Z$ is defined to be 
the (common) denominator $c$ of the coordinates of $Z$.
For $P = (\alpha, \beta) \in \QQQ$ and a rational $Z \in \QQQ$ we define
\begin{equation}\label{DefDeltaPZ}
\delta(P; Z) = \Ht(Z) \cdot \sqrt{
	\left( \alpha - \frac ac \right)^2 + 
	\left( \beta - \frac bc \right)^2}.
\end{equation}
Also,  we define \emph{the Lagrange number} $L(P)$ 
of $P$ to be
\begin{equation}\label{LiminfDefinition}
	L(P) = \limsup_{c\to\infty} \delta(P; Z)^{-1}
\end{equation}
where $\{Z = (\frac ac, \frac bc) \in \QQQ \}$ is ordered by height.
\end{definition}

We define a quadratic form
\[
Q(\mathbf{x})  = x_1^2 + x_2^2 - x_3^2
\]
and let
\[
	\langle \mathbf{x}, \mathbf{y} \rangle = 
	\frac{1}{2}
	\left(
	Q(\mathbf{x} + \mathbf{y}) - Q(\mathbf{x}) - Q(\mathbf{y})
	\right)
\]
be a symmetric, bilinear pairing associated to $Q(\mathbf{x})$.
In terms of coordinates of $\mathbf{x}$ and $\mathbf{y}$,  we have
\[
	\langle \mathbf{x}, \mathbf{y} \rangle = x_1y_1 + x_2y_2 - x_3y_3.
\]
One of our key tools in computing $L(P)$ is the fact that $M_1, M_2, M_3$ are orthogonal with respect to $Q(\mathbf{x})$, or equivalently, to $\langle \,\, , \, \, \rangle$.
\begin{lemma}\label{OrthgonalityOfMd}
	Let $M_1, M_2, M_3$ be as in \eqref{DefinitionMs}.
	\begin{enumerate}[font=\upshape, label=(\alph*)]
\item Define
\[
U_1=
\begin{pmatrix}
1 & 0  & 0 \\
0 & -1 & 0 \\
0 & 0  & 1 \\
\end{pmatrix},
\quad
U_2=
\begin{pmatrix}
-1 & 0  & 0 \\
0 & -1 & 0 \\
0 & 0  & 1 \\
\end{pmatrix},
\quad
U_3=
\begin{pmatrix}
-1 & 0  & 0 \\
0 & 1 & 0 \\
0 & 0  & 1 \\
\end{pmatrix},
\]
and
\[
H=
\begin{pmatrix}
-1 & -2  & 2 \\
-2 & -1  & 2 \\
-2 & -2  & 3 \\
\end{pmatrix}.
\]
Then, $U_d$ and $H$ are of order 2 and they satisfy
\[
	M_d = HU_d, \quad \text{and} \quad M_d^{-1} = U_dH
\]
for $d = 1, 2, 3$. 
		\item $H$ and $U_d$'s are orthogonal with respect to $\langle \cdot, \cdot  \rangle$. 
			Namely, for any $\mathbf{x}$ and $\mathbf{y}$ in $\mathbb{R}^3$,
\[
	\langle H \mathbf{x},H \mathbf{y} \rangle =
	\langle U_d \mathbf{x},U_d \mathbf{y} \rangle =
	\langle \mathbf{x}, \mathbf{y} \rangle
\]
for $d = 1, 2, 3$.
Also, $M_d$'s are orthogonal with respect to $\langle \cdot, \cdot  \rangle$ as well.
	\end{enumerate}
\end{lemma}
\begin{proof}
The matrices $U_1, U_2, U_3, H$ are \emph{reflections} in a quadratic space $(\mathbb{R}^3, Q(\mathbf{x}))$.
Therefore, they are orthogonal with respect to $\langle , \rangle$ and are of order 2.
See \cite{CNT} and \cite{CK18} for more detail.
\end{proof}
\begin{proposition}\label{ck}
Let $Z$ be a rational point in the interior of 
a cylinder set $\Cyl(d_1, \dots, d_k)$.
Then, the height of $Z$ is greater than those of boundary points. That is,
\[
\max \{ \Ht(Z^{(1,0)}), \Ht(Z^{(0,1)}) \} < \Ht(Z).
\]
\end{proposition}
\begin{proof}
Write $Z = (\frac ac, \frac bc)$. Let $\mathbf{z} = (a, b, c)$ be the vector representing $Z$.
Then
\[
		\mathbf{z} =  M_{d_1} \cdots M_{d_k} 
		\mathbf{v}
\]
for some primitive Pythagorean triple $\mathbf{v} = (a_0, b_0, c_0)$.
Define $\mathbf{u}^{(1, 0)}$ and $\mathbf{u}^{(0, 1)}$ to be as in \eqref{EqDefU}.
The condition that $Z$ is an \emph{interior} point of $\Cyl(d_1, \dots, d_k)$ implies that $\mathbf{v}$ is not equal to $\mathbf{u}^{(1, 0)}$ or $\mathbf{u}^{(0,1)}$.
In particular,
	\begin{equation}\label{InitialConditionabc}
		a_0, b_0 \ge 3 \text{ and } c_0\ge 1 + \max\{a_0, b_0\}.
	\end{equation}
	We first prove a lemma, which says that if $x, y\ge 2$ and $z\ge \max\{x, y\}$ and if 
	\[ 
	\begin{pmatrix} x' \\ y' \\ z' \end{pmatrix} = M_d  \begin{pmatrix} x \\ y \\ z \end{pmatrix}
	\]
	for $d=1, 2, 3$, then the triple $(x', y', z')$ also satisfies $x',y' \ge 2$ and $z' \ge \max\{x',y'\}$.
	Here, we do not assume that $(x, y, z)$ is necessarily a Pythagorean triple.
	This lemma is easily proven by direct computation. For example, when $d=1$,
	\[ 
	\begin{pmatrix} x' \\ y' \\ z' \end{pmatrix} = M_1  \begin{pmatrix} x \\ y \\ z \end{pmatrix}
		= \begin{pmatrix} 
	-x +2y + 2z \\ 
	-2x + y + 2z \\
	-2x + 2y + 3z
\end{pmatrix},
\]
so that
\[
		x' \ge 2y + z \ge 2, \qquad
	y' \ge y\ge 2.\\
\]
Moreover, 
\[
		z' - x' = z-x \ge 0,  \qquad
		z' - y' = z+y \ge 0, %\\
\]
which completes the proof of the lemma for $d=1$. The cases $d=2$ and $d=3$ are similar.

To complete the proof of the present proposition, we define 
\[
\mathbf{w} =
		\begin{pmatrix}
		x \\ y \\ z
		\end{pmatrix}
		=
\mathbf{v} - \mathbf{u}^{(1,0)}.
\]
Then, we see from \eqref{InitialConditionabc} that $x, y\ge2$ and $z \ge \max\{ x, y \}$.
Define $\mathbf{w}' = (x', y', z') = \mathbf{z} - \mathbf{z}^{(1,0)}$. Then,
	\[
	\begin{pmatrix} x' \\ y' \\ z' \end{pmatrix} = 
	\mathbf{z} - \mathbf{z}^{(1,0)}
=
M_{d_1} \cdots M_{d_k}(\mathbf{v} - \mathbf{u}^{(1, 0)}) =
M_{d_1} \cdots M_{d_k}
\begin{pmatrix} x \\ y \\ z \end{pmatrix}.
\]
We apply our lemma $k$ times successively to conclude that $x', y'\ge 2$ and $z'\ge \max\{x', y'\}$.
This implies 
\[
 \Ht(Z) - \Ht(Z^{(1,0)}) = z' \ge \max\{ x', y'\} \ge 2.
\]
Likewise, an identical argument with $(0, 1)$ replacing $(1, 0)$ shows that $\Ht(Z) - \Ht(Z^{(0, 1)}) \ge 2$.
This completes the proof of the proposition.
\end{proof}

Suppose that $P$ is an irrational point in $\QQQ$ with its digit expansion given by an infinite sequence
\[
	P = [d_1, \dots, d_k, d_{k+1}, \dots]_{\QQQ}.
\]
 Also, for each $k\ge 1$, we let
\[
Z_k^{(1, 0)}(P) = Z^{(1, 0)}
\ \text{ and } \
Z_k^{(0, 1)}(P) = Z^{(0, 1)},
\]
which are defined in \eqref{DefinitionZk} with respect to the first $k$ Romik digits of $P$.
Also we write $\Cyl_k(P)$ to mean $\Cyl(d_1, \dots, d_k)$ for each $k \ge 1$.

\begin{proposition}
  Suppose that $P$ is an irrational point in $\QQQ$.
  Then we have
\[
	\bigcap_{k=1}^{\infty} \Cyl_k(P)
	=
	\{ P \}.
\]
Also $ Z_k^{(1, 0)}(P) \to P$ and $Z_k^{(0, 1)}(P) \to P$ as $k\to \infty$ with respect to the usual Euclidean distance in $\mathbb{R}^2$.
\label{prop:dec_cylinder_convergence}
\end{proposition}

\begin{proof}
  First, we claim that the diameter $\mathrm{diam}(\Cyl_k(P))$ of $\Cyl_k(P)$ tends to zero as $k\to\infty$.
  Because $\Cyl_k(P)$ is a (closed) subarc of $\QQQ$, its diameter is simply given by the distance between its boundary points 
  $Z_k^{(1, 0)}(P)$ 
  and
  $Z_k^{(0, 1)}(P)$. 
  Using the notations in \S\ref{SecCylinder} (except that we keep the index $k$ to emphasize its dependence), we have
  \begin{align*}
  \mathrm{diam}(\Cyl_{k}(P))^2
  &=
  \left( 
  \frac{a_k^{(1, 0)}}{c_k^{(1, 0)}}
  -
  \frac{a_k^{(0, 1)}}{c_k^{(0, 1)}}
  \right)^2 
  +
  \left( 
  \frac{b_k^{(1, 0)}}{c_k^{(1, 0)}}
  -
  \frac{b_k^{(0, 1)}}{c_k^{(0, 1)}}
  \right)^2 
  \\
  &=
  \frac{
  -2
  \langle 
  \mathbf{z}_k^{(1, 0)}, \mathbf{z}_k^{(0, 1)}\rangle
  }{
    c_k^{(1, 0)}c_k^{(0, 1)}
  }
  =
  \frac{
  -2
  \langle 
  \mathbf{u}^{(1, 0)}, \mathbf{u}^{(0, 1)}\rangle
  }{
    c_k^{(1, 0)}c_k^{(0, 1)}
  }
  %\\
  %&
  =
  \frac{
  2
  }{
    c_k^{(1, 0)}c_k^{(0, 1)}
  }.
  \end{align*}
  In the second last equality above, the orthogonality of $M_{d}$ is used.

	Writing $M = M_{d_1}\cdots M_{d_k}$, we define
	$\mathbf{y}_1, \mathbf{y}_2, \mathbf{y}_3, \mathbf{y}_4$ to be
	\[
\mathbf{y}_1 = M
\begin{pmatrix}
1 \\ 0 \\ 1
\end{pmatrix},
\quad
\mathbf{y}_2 = M
\begin{pmatrix}
4 \\ 3 \\ 5
\end{pmatrix},
\quad
\mathbf{y}_3 = M
\begin{pmatrix}
3 \\ 4 \\ 5
\end{pmatrix},
\quad
\mathbf{y}_4 = M
\begin{pmatrix}
0 \\ 1 \\ 1
\end{pmatrix}
	\]
	and let $Y_1, \dots, Y_4$ be the points in $\QQQ$ represented by $\mathbf{y}_1,\dots, \mathbf{y}_4$.
Then $Y_1, \dots, Y_4$
are boundary points of $\Cyl(d_1, \dots, d_k, d)$ with $d = 1, 2, 3$, as shown in Figure~\ref{FigCased1}.
Also, $Y_1 = Z_k^{(1, 0)}(P)$
and
$Y_4 = Z_k^{(0, 1)}(P)$
are the boundaries of $\Cyl(d_1, \dots, d_k)$.
So
$c_k^{(1, 0)}c_k^{(0, 1)} = 
\Ht(Z_k^{(1, 0)}(P)) 
\Ht(Z_k^{(0, 1)}(P)) 
=\Ht(Y_1)\Ht(Y_4)
$.
Furthermore,
\[
  c_{k+1}^{(1, 0)}c_{k+1}^{(0, 1)} = 
  \begin{cases}
    \Ht(Y_1)\Ht(Y_2) & \text{ if } d_{k+1} = 1, \\
    \Ht(Y_3)\Ht(Y_2) & \text{ if } d_{k+1} = 2, \\
    \Ht(Y_3)\Ht(Y_4) & \text{ if } d_{k+1} = 3. \\
  \end{cases}
\]
From Proposition~\ref{ck} we have $
\max\{ \Ht(Y_1), \Ht(Y_4) \}
<
\min\{ \Ht(Y_2), \Ht(Y_3) \}
$,
so that 
$c_k^{(1, 0)}c_k^{(0, 1)} < c_{k+1}^{(1, 0)}c_{k+1}^{(0, 1)}$.
This proves the claim that $\mathrm{diam}(\Cyl_k(P)) \to 0$ as $k \to \infty$.

Since the distances from $Z_k^{(1, 0)}(P)$ to $P$ and from $Z_k^{(0, 1)}(P)$ to $P$ are bounded by $\mathrm{diam}(\Cyl_k(P))$ we conclude $Z_k^{(1, 0)}(P) \to P$ and $Z_k^{(0, 1)}(P) \to P$ as $k \to \infty$.
Also,  because $\{ \Cyl_k(P) \}_{k=1}^{\infty}$ forms a decreasing sequence of compact sets of shrinking diameters, their intersection consists of a singleton, namely, $\{ P \}$. 
\end{proof}

\begin{figure}
\begin{tikzpicture}[scale=0.4,every text node part/.style={align=center}]
	\draw[very thick] (9.23, 3.8846) node [below] {$Y_1$} arc (22.61686:67.3815:10);
	\draw[fill] (8, 6) circle (0.1) node[below left] {$Y_2$};
	\draw[fill] (9.230769230769232, 3.8461538461538463) circle (0.1) node{};
	\draw[fill] (8.0, 6.0) circle (0.1) node[below left]{};
	\draw[fill] (53.13:10) circle (0.1) node[below left]{$Y_3$};
	\draw[fill] (3.8461538461538463, 9.230769230769232) circle (0.1)  node[below left]{$Y_4$};
	\draw[<->, dotted, thick] ([shift=(22.61986:10.4)]0, 0) arc (22.61986: 36.8699:10.4) node[midway,  right] {$\Cyl(d_1, \dots, d_k, 1)$};
	\draw[<->, dotted, thick] ([shift=(36.8699:10.4)]0, 0) arc (36.8699: 53.1301:10.4) node[midway, right] {$\Cyl(d_1, \dots, d_k, 2)$};
	\draw[<->, dotted, thick] ([shift=(53.1301:10.4)]0, 0) arc (53.1301: 67.38014: 10.4) node[midway, right] {$\Cyl(d_1, \dots, d_{k}, 3)$};
	\end{tikzpicture}
	\caption{The cylinder set 
	$\Cyl(d_1, \dots, d_k)$ 
	and the points $Y_1, \dots, Y_4$, which are boundaries of 
	$\Cyl(d_1, \dots, d_k, d)$ 
	with $d = 1, 2, 3$.
	\label{FigCased1}}
\end{figure}
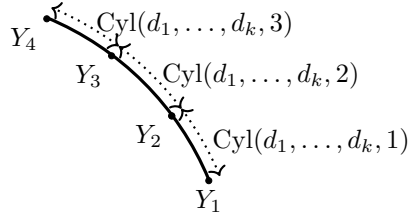

\begin{proposition}\label{PropInvariance}
Let 
\[
P = (\alpha, \beta) = [d_1, d_2, \dots ]_{\QQQ}.
\]
Fix $k$ and write $P' = (\alpha', \beta') = \TTT^k(P)$. 
Also, let 
$\mathbf{p} = (\alpha, \beta, 1)$
and
$\mathbf{p}' = (\alpha', \beta', 1)$
Then
\[
\frac{\langle \mathbf{x}_1, \mathbf{p}\rangle}%
{\langle \mathbf{x}_2, \mathbf{p}\rangle}
=
\frac{\langle M_{d_k}^{-1} \cdots M_{d_1}^{-1}\mathbf{x}_1, \mathbf{p}'\rangle}%
{\langle M_{d_k}^{-1} \cdots M_{d_1}^{-1}\mathbf{x}_2, \mathbf{p}'\rangle}
\]
for any $\mathbf{x}_1, \mathbf{x}_2 \in \mathbb{R}^3$ with 
$\langle \mathbf{x}_2, \mathbf{p}\rangle \neq 0$.
\end{proposition}
\begin{proof}
It suffices to prove this for the case $k=1$, as the general case will follow from applying this $k$ times. 
We see from Proposition~\ref{DigitsAndMatrixMultiplication} 
that there exists a positive real number, say, $\lambda$ such that $M_d^{-1} \mathbf{p} = \lambda\mathbf{p}'$.
Then the orthogonality of $M_d$ gives
\[
\frac{\langle \mathbf{x}_1, \mathbf{p}\rangle}%
{\langle \mathbf{x}_2, \mathbf{p}\rangle}
=
\frac{\langle M_d^{-1}\mathbf{x}_1, M_d^{-1}\mathbf{p}\rangle}%
{\langle M_d^{-1}\mathbf{x}_2, M_d^{-1}\mathbf{p}\rangle}
=
\frac{\langle M_d^{-1}\mathbf{x}_1, \lambda\mathbf{p}'\rangle}%
{\langle M_d^{-1}\mathbf{x}_2, \lambda\mathbf{p}'\rangle}
=
\frac{\langle M_d^{-1}\mathbf{x}_1, \mathbf{p}'\rangle}%
{\langle M_d^{-1}\mathbf{x}_2, \mathbf{p}'\rangle}.
\]
\end{proof}
\begin{definition}\label{DefAngleTheta}
    Let $P = (\alpha, \beta)$ be a point in the unit circle (that is, $\alpha^2 + \beta^2 = 1$).
    Define $\theta(P)$ to be the unique number in $[0, 2\pi)$ such that
    \[
    \alpha = \sin\theta(P), \qquad
    \beta = \cos\theta(P).
    \]
    In other words, $\theta(P)$ is the angle measured from $(0, 1)$ to $P$ clockwise.
    If $P_1$ and $P_2$ are on the unit circle, we let
    $\theta(P_1, P_2)$ 
    be the angle between $P_1$ and $P_2$,
    that is,
    $\theta(P_1, P_2)\in [0, 2\pi)$ 
    with
    \[
    \theta(P_1, P_2) \equiv \theta(P_1) - \theta(P_2) \bmod 2\pi.
    \]
\end{definition}

\begin{lemma}\label{PropDeltaSquared}
	Let $P= (\alpha, \beta) \in \QQQ$ and let $Z = (\frac ac, \frac bc)$ be a rational point in $\QQQ$. Write
	\[
		\mathbf{p} = (\alpha, \beta, 1) \quad \text{and} \quad \mathbf{z} = (a, b, c).
	\]
	Then
	\[
		\delta^2(P; Z) = -2\Ht(Z) \langle \mathbf{p} , \mathbf{z} \rangle.
	\]
\end{lemma}
\begin{proof}
	This is straightforward from Definition~\ref{DefDelta}.
	Indeed,
	\[
		\delta^2(P; Z) = 
		\left( c\alpha - a \right)^2 + \left( c\beta - b \right)^2
		= 2c^2 - 2c(a\alpha + b\beta) \\
		= -2\Ht(Z) \langle \mathbf{p} , \mathbf{z} \rangle.  \qedhere
	\] 
\end{proof}
\begin{lemma}\label{LemAngleEqDeltaPZ}
	\begin{enumerate}[font=\upshape, label=(\alph*)]
		\item Suppose that $X = (x_1, x_2)$ and $Y = (y_1, y_2)$ are two points on the unit circle and let
\[
	\mathbf{x} = (x_1, x_2, 1) \ \text{ and } \ \mathbf{y} = (y_1, y_2, 1).
\]
Then,
\[
	\langle \mathbf{x}, \mathbf{y} \rangle = -2\sin^2
	\left(
		\frac{\theta(X, Y)}2
\right).
\]
\item Let $P= (\alpha, \beta) \in \QQQ$ and let $Z$
be a rational point in $\QQQ$. Then,
	\[
	\delta(P; Z) = 2 \Ht(Z)
		\sin
		\left|
			\frac{\theta(P, Z)}2 
		\right|.
	\]
\end{enumerate}
\end{lemma}
\begin{proof}
	The statement (a) is a consequence of elementary trigonometry and (b) is obvious from (a) and Lemma~\ref{PropDeltaSquared}.
\end{proof}
Next, we show in Theorem~\ref{BestApproximation} that 
$\{Z_k^{(1, 0)}(P)\}_{k=0}^{\infty}$ 
and
$\{Z_k^{(0, 1)}(P)\}_{k=0}^{\infty}$,
that is, the endpoints of $\Cyl_k(P)$ for all $k\ge0$
contain the \emph{best approximants} of $P$.

\begin{theorem}\label{BestApproximation}
	Fix an irrational point  $P\in \QQQ$.
	For any rational point $Z\in\QQQ$,
	there exists a $k\ge 0$ such that
	\[ 
	\min\{ \delta(P; Z_k^{(1, 0)}(P)), \delta(P; Z_k^{(0, 1)}(P)) \} \le \delta(P; Z). 
	\]
\end{theorem}

\begin{proof}
Recall from  Proposition~\ref{prop:dec_cylinder_convergence} that $P = (\alpha,\beta)$ is the unique point lying in the intersection of $\{\Cyl_k(P)\}_{k=0}^{\infty}$. 
If $Z$ is one of the boundary points of $\Cyl_k(P)$ for some $k$, then the statement in the theorem is trivially true and there is nothing to prove. 
So we will assume in what follows that $Z$ is not equal to$Z_k^{(1, 0)}(P)$ or $Z_k^{(0, 1)}(P)$ for any $k$.

From now on, we fix $k$ to be the unique index such that
\[
	Z\in \Cyl(d_1, \dots, d_k) -\Cyl(d_1, \dots, d_k, d_{k+1}).
\]
Also, we write $Z = (\frac ac,\frac bc)$, so that the vector $\mathbf{z} = (a, b, c)$ representing $Z$ is given by 
\[
\mathbf{z} = 
	M_{d_1}\cdots M_{d_k}
\begin{pmatrix}
a_0 \\ b_0 \\ c_0
\end{pmatrix}
\]
for a primitive Pythagorean triple $(a_0, b_0, c_0)$.
We let $P' = (\alpha', \beta') = \TTT^k(\alpha, \beta)$ and write $\mathbf{p} = (\alpha, \beta, 1)$ and
$\mathbf{p}' = (\alpha', \beta', 1)$.
Following the notations introduced in the proof of Proposition~\ref{prop:dec_cylinder_convergence}, we will let $Y_1, Y_2, Y_3, Y_4$ be the points in $\Cyl_k(P)$ that are represented by $\mathbf{y}_1, \mathbf{y}_2, \mathbf{y}_3, \mathbf{y}_4$. 

First, we consider the case $d_{k+1}(P)= 2$.
Since $Z\not\in \Cyl_{k+1}(P)$, 
we have either 
$Z\in \Cyl(d_1,\dots, d_k, 1)$  or
$Z\in \Cyl(d_1,\dots, d_k, 3)$. 
By symmetry we may assume without loss of generality that
$Z\in \Cyl( (d_1,\dots, d_k, 1)$.
We shall prove that
\begin{equation}\label{EqY2Z}
    \delta(P; Y_2) \le \delta(P; Z),
\end{equation}
which would prove the conclusion of the theorem because $Y_2 = Z_{k+1}^{(0, 1)}(P)$ in this case.
Since $Z$ is assumed not to be equal to the boundary points of $\Cyl_k(P)$ or $\Cyl_{k+1}(P)$, we see that $Z$ is an interior point of $\Cyl(d_1, \dots, d_k, 1)$.
Therefore we can apply Proposition~\ref{ck} to $\Cyl(d_1,\dots, d_k,1)$ and obtain $\Ht(Y_2)\le \Ht(Z)$. 
Using this and Lemma~\ref{PropDeltaSquared}, we have
\[
\frac{\delta^2(P; Y_2)}{\delta^2(P; Z)}
=
\frac{
-2\Ht(Y_2) \langle \mathbf{p}, \mathbf{y}_2 \rangle
}{
-2\Ht(Z) \langle \mathbf{p}, \mathbf{z} \rangle
}
\le
\frac{
\langle \mathbf{p}, \mathbf{y}_2 \rangle
}{
\langle \mathbf{p}, \mathbf{z} \rangle
}.
\]
From Proposition~\ref{PropInvariance}, we get
\[
\frac{
\langle \mathbf{p}, \mathbf{y}_2 \rangle
}{
\langle \mathbf{p}, \mathbf{z} \rangle
}
=
\frac{
\langle \mathbf{p}', M^{-1}\mathbf{y}_2 \rangle
}{
\langle \mathbf{p}', M^{-1}\mathbf{z} \rangle
}
=
\frac{
\langle \mathbf{p}', (4, 3, 5)\rangle
}{
\langle \mathbf{p}', (a_0, b_0, c_0) \rangle
}.
\]
Since $Z\in \Cyl(d_1, \dots, d_k, 1)$ we have
$\TTT^k(Z) = (\frac{a_0}{c_0}, \frac{b_0}{c_0}) \in \Cyl(1)$.
In particular, the vector
\[
M_1^{-1}
\begin{pmatrix}
a_0\\ b_0\\c_0 
\end{pmatrix}
=
\begin{pmatrix}
-a_0 - 2b_0 + 2c_0 \\ 
2a_0 + b_0 - 2c_0 \\
-2a_0 - 2b_0 + 3c_0 
\end{pmatrix}
\]
must represent a point in $\QQQ$, so that
\begin{equation}\label{Eq24}
\begin{aligned}
-a_0 - 2b_0 + 2c_0 &>0 ,\\
2a_0 + b_0 - 2c_0 &>0, \\
-2a_0 - 2b_0 + 3c_0 &>0.
\end{aligned}
\end{equation}
On the other hand, if we define 
$(a_0',  b_0', c_0')$ to be
$(a_0', b_0', c_0') = M_2^{-1}(a_0, b_0, c_0)$, then
the point $(\frac{a_0'}{c_0'}, \frac{b_0'}{c_0'})$ is in the second quadrant, that is, 
$a_0' < 0$, $b_0'>0$, and $c_0'>0$.
Indeed, it is easy to deduce from Lemma~\ref{OrthgonalityOfMd} that 
$M_2^{-1} = U_3 M_1^{-1}$.
Therefore
\[
\begin{pmatrix}
a_0'\\ b_0'\\c_0' 
\end{pmatrix}
=M_2^{-1}
\begin{pmatrix}
a_0\\ b_0\\c_0 
\end{pmatrix}
=U_3 M_1^{-1}
\begin{pmatrix}
a_0\\ b_0\\c_0 
\end{pmatrix}
=
\begin{pmatrix}
a_0 + 2b_0 - 2c_0 \\ 
2a_0 + b_0 - 2c_0 \\
-2a_0 - 2b_0 + 3c_0 
\end{pmatrix},
\]
and the assertion follows from this and \eqref{Eq24}.
Writing $\mathbf{p}'' = (\alpha'', \beta'', 1)$ with $P'' = (\alpha'', \beta'') = \TTT(\alpha', \beta')$, we apply Proposition~\ref{PropInvariance} once again to obtain
\begin{equation}\label{EqBilinearInequality}
\begin{split}
\frac{
\langle \mathbf{p}', (4, 3, 5) \rangle
}{
\langle \mathbf{p}', (a_0, b_0, c_0) \rangle
}
&=
\frac{
\langle \mathbf{p}'', M_2^{-1}(4, 3, 5) \rangle
}{
\langle \mathbf{p}'', M_2^{-1}(a_0, b_0, c_0) \rangle
} 
=
\frac{
\langle \mathbf{p}'', (0, 1, 1) \rangle
}{
\langle \mathbf{p}'', (a_0', b_0', c_0') \rangle
} \\
&
=
\frac1{c_0}
\frac{
\langle \mathbf{p}'', (0, 1, 1) \rangle
}{
\langle \mathbf{p}'', 
\left(
\frac{a_0'}{c_0'}, \frac{b_0'}{c_0'}, 1
\right) 
\rangle
}
\le 
\frac{
\langle \mathbf{p}'', (0, 1, 1) \rangle
}{
\langle \mathbf{p}'', 
\left(
\frac{a_0'}{c_0'}, \frac{b_0'}{c_0'}, 1
\right) 
\rangle
}.
\end{split}
\end{equation}
Since the point $
(\frac{a_0'}{c_0'}, \frac{b_0'}{c_0'}) $
is in the second quadrant
we must have 
\[
0 \le 
\theta(P'', (0, 1))
\le 
\theta(P'', (\tfrac{a_0'}{c_0'}, \tfrac{b_0'}{c_0'}) )
\le \pi.
\]
(See Definition~\ref{DefAngleTheta}.)
We conclude from Lemma~\ref{LemAngleEqDeltaPZ}
that
\[
\frac{
\langle \mathbf{p}'', (0, 1, 1) \rangle
}{
\left \langle \mathbf{p}'', 
\left(
\frac{a_0'}{c_0'}, \frac{b_0'}{c_0'}, 1
\right) 
\right \rangle
}
\le 1.
\]
Combining this with \eqref{EqBilinearInequality}, we establish \eqref{EqY2Z}.

Next, we consider the cases $d_{k+1}(P) = 1$ and $d_{k+1}(P) = 3$.
By symmetry it will be sufficient for us to prove the former.
Then
$Z \in \Cyl(d_1, \dots, d_k, 2)$ or 
$Z \in \Cyl(d_1, \dots, d_k, 3)$.
Under this assumption, we shall prove 
\begin{equation}\label{EqY1Z}  
    \delta(P; Y_1) \le \delta(P; Z),
\end{equation}
which would then complete the proof of Theorem~\ref{BestApproximation}
because $Y_1 = Z_k^{(1, 0)}(P)$.
Apply Proposition~\ref{ck} to the cylinder set $\Cyl_k(P) = \Cyl(d_1, \dots, d_k)$
to obtain $\Ht(Y_1) \le \Ht(Z)$.
This and Proposition~\ref{PropInvariance} give 
\begin{equation}\label{EqFracPY1Z}
\frac{\delta^2(P; Y_1)}{\delta^2(P; Z)}
=
\frac{
-2\Ht(Y_1) \langle \mathbf{p}, \mathbf{y}_1 \rangle
}{
-2\Ht(Z) \langle \mathbf{p}, \mathbf{z} \rangle
}
\le
\frac{
\langle \mathbf{p}, \mathbf{y}_1 \rangle
}{
\langle \mathbf{p}, \mathbf{z} \rangle
}
=
\frac{
\langle \mathbf{p}', (1, 0, 1) \rangle
}{
\langle \mathbf{p}', (a_0, b_0, c_0)\rangle
}.
\end{equation}
Recall that $d_{k+1}(P)$ is assumed to be 1, so that $P'\in \Cyl(1)$. 
So we must have 
$0 \le \theta ((1, 0), P') \le \theta((1, 0),(\frac 45, \frac35)) < \pi/2$.
Then Lemma~\ref{LemAngleEqDeltaPZ} shows that
\begin{equation}\label{EqLangle15}
-  \langle \mathbf{p}', (1, 0, 1) \rangle
\le
-  
\langle (\tfrac45, \tfrac35, 1), (1, 0, 1) \rangle
=\frac15.
\end{equation}
On the other hand, we are assuming 
$Z \in \Cyl(d_1, \dots, d_k, 2)$
or
$Z \in \Cyl(d_1, \dots, d_k, 3)$.
As a consequence, 
$(\frac{a_0}{c_0}, \frac{b_0}{c_0})$ 
is in
$\Cyl(2)$
or
$\Cyl(3)$.
In either case, 
we have 
\[   
0 \le
\theta(
(\tfrac 45, \tfrac35), (\tfrac{a_0}{c_0}, \tfrac{b_0}{c_0}))
\le
\theta(
P', (\tfrac{a_0}{c_0}, \tfrac{b_0}{c_0}))
\le 
\frac{\pi}{2}
\]
because $P'\in \Cyl(1)$ and $\{(\frac 45, \frac35)\}$ is the intersection of $\Cyl(1)$ and $\Cyl(2)$
(see Figure~\ref{FigCylinderPicture}).
Again, from Lemma~\ref{LemAngleEqDeltaPZ}, we obtain
\[
-  
\langle 
(\tfrac{a_0}{c_0}, \tfrac{b_0}{c_0}, 1) ,
(\tfrac 45, \tfrac35, 1)
\rangle
\le
-  
\langle 
(\tfrac{a_0}{c_0}, \tfrac{b_0}{c_0}, 1) ,
\mathbf{p}'
\rangle,
\]
which is equivalent to
\begin{equation}\label{EqLangle151}
-  
\langle 
(a_0, b_0, c_0) ,
\mathbf{p}'
\rangle
\ge
-
\langle 
(a_0, b_0, c_0) ,
(\tfrac 45, \tfrac35, 1)
\rangle
=
-\frac15
\langle 
(a_0, b_0, c_0) ,
(4, 3, 5)
\rangle.
\end{equation}
Notice that $(a_0, b_0, c_0) \neq (4, 3, 5)$, otherwise $Z = Y_2$ which would then violate the assumption that $Z$ is not equal to any boundary point of $\Cyl_{k+1}(P)$.
Therefore
$
\langle 
(a_0, b_0, c_0) ,
(4, 3, 5)
\rangle$ is strictly negative and is at most $-1$ since $a_0, b_0, c_0$ are integers. 
Hence \eqref{EqLangle151} becomes
\begin{equation}\label{EqLangle152}
-  
\langle 
(a_0, b_0, c_0) ,
\mathbf{p}'
\rangle
\ge
\frac15.
\end{equation}
Combining \eqref{EqFracPY1Z}, \eqref{EqLangle15}, and \eqref{EqLangle152}, we obtain \eqref{EqY1Z}. 
The proof of Theorem~\ref{BestApproximation} is now completed.
\end{proof}

\begin{definition}\label{DefinitionPCheck}
	For $P = (\alpha, \beta)  = [d_1, d_2, \dots]_{\QQQ}$, define
	\[
		P^{\vee} = (\beta, \alpha).
	\]
	Also, for $d\in\{1, 2, 3\}$, define
	\[
		d^{\vee} =
		\begin{cases}
			3 & \text{ if } d = 1, \\
			2 & \text{ if } d = 2, \\
			1 & \text{ if } d = 3. \\
		\end{cases}
	\]
	Clearly, we have
	\[
		P^{\vee} = [d_1^{\vee}, d_2^{\vee},  \dots]_{\QQQ}.
	\]
\end{definition}

\begin{proposition}\label{ThmDeltaPcheck}
	We have
	\[
		\delta(P; Z_k^{(0, 1)}(P)) = \delta(P^{\vee}; Z_k^{(1, 0)}(P^{\vee})),
	\]
	and
	\[
		\delta(P; Z_k^{(1, 0)}(P)) = \delta(P^{\vee}; Z_k^{(0, 1)}(P^{\vee})).
	\]
\end{proposition}
\begin{proof}
	Let
	\[
		S = 
		\begin{pmatrix}
			0 & 1 & 0 \\
			1 & 0 & 0 \\
			0 & 0 & 1 \\
		\end{pmatrix}.
	\]
	Then it is straightforward to see that $S^2 = I_3$ (the $3\times 3$ identity) and 
	\[
		SM_dS = M_{d^{\vee}}
	\]
	for $d\in\{1,2,3\}$.
	Hence
	\begin{align*}
		M_{d_1} M_{d_2}\cdots M_{d_k} 
		\mathbf{u}^{(0, 1)} 
		&=
		S\cdot
		(SM_{d_1}S)(S M_{d_2} S) \cdots (SM_{d_k}S) 
		(S\mathbf{u}^{(0, 1)}) \\
		&=
		S\cdot
		M_{d_1^{\vee}}M_{d_2^{\vee}} \cdots M_{d_k^{\vee}} 
		\mathbf{u}^{(1, 0)}.
	\end{align*}
	This shows that 
	$\mathbf{z}_k^{(0, 1)}(d_1, d_2, \dots d_k) = S
	\cdot
	\mathbf{z}_k^{(1, 0)}(d_1^{\vee}, d_2^{\vee}, \dots d_k^{\vee})$.
	(cf.~\eqref{V10} and \eqref{V01})
	The first equality in the proposition easily follows from this.
	The second equality is proven similarly.
\end{proof}
\begin{corollary}\label{DeltaLiminfMin}
	Suppose that $P$ is an irrational point in $\QQQ$. Then,
	\[
		L(P) = 
		\limsup_{k\to\infty}
		\max
		\left\{
	\delta(P;Z_k^{(0, 1)}(P))^{-1},
	\delta(P^{\vee};Z_k^{(0, 1)}(P^{\vee}))^{-1} 
		\right\}.
\]
\end{corollary}
\begin{proof}
Thanks to Theorem~\ref{BestApproximation}, we have
\[
L(P) =
		\limsup_{k\to\infty}
		\max
		\left\{
	\delta(P;Z_k^{(0, 1)}(P))^{-1},
	\delta(P;Z_k^{(1, 0)}(P))^{-1} 
		\right\}.
\]
However, Proposition~\ref{ThmDeltaPcheck} says
\[
	\delta(P;Z_k^{(1, 0)}(P)) 
	=
	\delta(P^{\vee};Z_k^{(0, 1)}(P^{\vee}))
\]
and the proposition follows from this.
\end{proof}

\subsection{Perron's formula}\label{SecPerron}
\begin{definition}\label{DefNormRomikSequence}
	For $P=(\alpha, \beta) \in \QQQ$, we define $ \| P \|$ to be
\[
	\| P \| = 
	\frac1{\sqrt2}
	\left(
\cot
\left(
	\frac{\theta(P)}2
\right)
-1
\right).
\]
Here, $\theta(P)$ is defined in Definition~\ref{DefAngleTheta}.
\end{definition}

Our definition of $\| P \|$ is a slight modification of the standard stereographic projection ($=\cot(\theta(P)/2)$).
See \eqref{EqHalfCotangent}.
This modification is intended to make our version of \emph{Perron's formula} in Theorem~\ref{ThmPerron} closely resemble a classical Perron's formula for real numbers.
Also, see Remark~\ref{rem:stereo_graphic_projection_justification}.

Define a map $\TTT_{[0, \infty]}:  [0, \infty] \longrightarrow [0, \infty]$ to be
\begin{equation}\label{EqTreal}
\TTT_{[0, \infty]}(t) =
\begin{cases}
\frac{t}{-\sqrt2 t + 1} & \text{if } 0 \le t \le \frac1{\sqrt2},\\
\frac{-t + \sqrt2}{\sqrt2 t - 1} & \text{if } \frac1{\sqrt2} \le t \le \sqrt2,\\
t -\sqrt2  & \text{if } t \ge \sqrt2.
\end{cases}
\end{equation}
Then it is straightforward to verify that the following diagram commutes.
\begin{equation}
\begin{tikzcd}[column sep = large]
\QQQ 
\arrow{d}{\TTT}
\arrow{r}{P\mapsto\|P \|}
& {[0, \infty]} 
\arrow{d}{\TTT_{[0, \infty]}} 
\\
\QQQ 
\arrow{r}{P\mapsto\|P \|}
& {[0, \infty]}
\end{tikzcd}
  \label{EqCommuteDiagram}
\end{equation}
In other words, $P \mapsto \| P \|$ is a conjugate map from $(\QQQ, \TTT)$ to $([0,\infty], \TTT_{[0,\infty]})$.

\begin{lemma}\label{LemRefereeRequest}
Let $P = (\alpha, \beta)$ be such that $\alpha^2 + \beta^2 = 1$ and let $\mathbf{p}$ be any nonzero scalar multiple of $(\alpha, \beta, 1)$.
Define $\mathbf{p}' = (p_1', p_2', p_3')$ to be $\mathbf{p}' = H\mathbf{p}$ and let $P' = (p_1'/p_3', p_2'/p_3')$. 
(See Lemma~\ref{OrthgonalityOfMd} to recall the definition of $H$.)
Then
\[
\cot
\left(
	\frac{\theta(P)}2
\right)
+
\cot
\left(
	\frac{\theta(P')}2
\right)
=
2.
\]
\end{lemma}

\begin{proof}
Write $\mathbf{p}= (p_1, p_2, p_3)$.
Then the definition of $\mathbf{p}'$ gives
\[
p_1' = -p_1 -2p_2 + 2p_3,\quad
p_2' = -2p_1 -p_2 + 2p_3,\quad
p_3' = -2p_1 -2p_2 + 3p_3,
\]
so that
\[
\frac{p_1}{p_3 - p_2} + 
\frac{p_1'}{p_3' - p_2'} 
=
\frac{2(p_3 - p_2)}{p_3 - p_2} =2.
\]
On the other hand, using the definition of $\theta(P)$, one easily shows 
\begin{equation}\label{EqHalfCotangent}
\cot
\left(
	\frac{\theta(P)}2
\right)
=
\frac{\alpha}{1 - \beta}.
\end{equation}
From this, the statement in the lemma is easily deduced.
\end{proof}

We now state and prove our analogue of Perron's formula.
\begin{theorem}[Perron's formula]\label{ThmPerron}
	Fix an irrational point $P \in \QQQ$ with 
	\[
		P = [d_1, d_2, \dots, d_k, d_{k+1}, \dots]_{\QQQ}.
	\]
	For each $k\geq 1$, define $P_k'$ and $P_k''$ to be
	\[
	P_k' =  [d_{k+1}, d_{k+2}, \dots ]_{\QQQ}, \qquad
	P_k'' = [d_k, d_{k-1}, \dots, d_2, d_1, 3^{\infty}]_{\QQQ}.
	\]
	Then, for all large $k$,
	\[
		\delta(P; Z_k^{(0, 1)}(P)) = \frac{\sqrt2}{\| P_k' \| + \| P_k'' \|  } \cdot \epsilon_k(P).
	\]
	Here, 
\[
	\epsilon_k(P) =
	\frac{\sin(\theta(P)/2)}{\sin(\theta(Z_k^{(0, 1)}(P))/2)}.
\]
Therefore, $\epsilon_k(P)\to1$ as $k\to\infty$.
\end{theorem}
\begin{proof}
Fix an index $k\ge 1$, large enough that the sequence $\{d_1, \dots, d_k\}$ is equal to neither $1^k$ nor $3^k$.
In this proof we write (cf.~\eqref{EqDefU} and \eqref{V01})
\[
\mathbf{u} = \mathbf{u}^{(0, 1)} =
\begin{pmatrix}
0 \\ 1 \\ 1
\end{pmatrix}
\quad
\text{and}
\quad
\mathbf{z}_k = \mathbf{z}^{(0, 1)}
=
M_{d_1} \cdots M_{d_k}\mathbf{u}
\]
Also we define
\begin{equation}
\mathbf{w}_k =
\begin{pmatrix}
w_1 \\ w_2\\ w_3
\end{pmatrix}
=
(M_{d_1} \cdots M_{d_k})^{-1} \mathbf{u}
\end{equation}
and let $W_k = (w_1/w_3, w_2/w_3)$ be the corresponding point on the unit circle.
Using Lemma~\ref{PropDeltaSquared}, we have
\begin{equation}\label{ZerothTerm}
\begin{aligned}
    \delta^2(P; Z_k^{(0, 1)}(P)) &= 
    -2\Ht(Z_k^{(0, 1)}(P))
    \langle \mathbf{p}, \mathbf{z}_k \rangle \\
    &=
    -2\Ht(Z_k^{(0, 1)}(P))
    \frac{
    \langle \mathbf{p}, \mathbf{u} \rangle
    }{
    \langle \mathbf{z}_k, \mathbf{u} \rangle
    }
    \cdot
    \frac{
    \langle \mathbf{p}, \mathbf{z}_k \rangle
    }{
    \langle \mathbf{p}, \mathbf{u} \rangle
    }
    \cdot
    \langle \mathbf{z}_k, \mathbf{u} \rangle.
\end{aligned}
\end{equation}
Here, we set $\mathbf p = (\alpha, \beta, 1)$ for $P=(\alpha,\beta)$.
From Lemma~\ref{LemAngleEqDeltaPZ}, we see that
\begin{equation}\label{FirstTerm}
    \Ht(Z_k^{(0, 1)}(P))
    \frac{
    \langle \mathbf{p}, \mathbf{u} \rangle
    }{
    \langle \mathbf{z}_k, \mathbf{u} \rangle
    }
    =
	\frac{\sin^2(\theta(P)/2)}{\sin^2(\theta(Z_k^{(0, 1)})/2)}
	=
   	\epsilon_k^2(P).
\end{equation}
Write $P_k' = (\alpha_k', \beta_k')$ and let $\mathbf{p}_k' = (\alpha_k', \beta_k', 1)$.
Then Proposition~\ref{PropInvariance} shows
\begin{equation}\label{SecondTerm}
    \frac{
    \langle \mathbf{p}, \mathbf{z}_k \rangle
    }{
    \langle \mathbf{p}, \mathbf{u} \rangle
    }
    =
     \frac{
    \langle \mathbf{p}_k', M_{d_k}^{-1} \cdots M_{d_1}^{-1}\mathbf{z}_k \rangle
    }{
    \langle \mathbf{p}_k', M_{d_k}^{-1} \cdots M_{d_1}^{-1}\mathbf{u} \rangle
    }
   =
    \frac{
    \langle \mathbf{p}_k', \mathbf{u} \rangle
    }{
    \langle \mathbf{p}_k', \mathbf{w}_k \rangle
    }.
\end{equation}
Finally, we obtain from the orthogonality of $M_{d_1} \cdots M_{d_k}$ 
\begin{equation}\label{ThirdTerm}
    \langle \mathbf{z}_k, \mathbf{u} \rangle
    =
    \langle 
(M_{d_1} \cdots M_{d_k})^{-1} \mathbf{z}_k,
(M_{d_1} \cdots M_{d_k})^{-1} \mathbf{u}
\rangle
=
    \langle \mathbf{u}, \mathbf{w}_k \rangle.
\end{equation}
Combining \eqref{ZerothTerm}, \eqref{FirstTerm}, \eqref{SecondTerm}, \eqref{ThirdTerm}, we have
\begin{equation}\label{FifthTerm}
    \delta^2(P; Z_k^{(0, 1)}(P)) = 
    (-2)
    \epsilon^2_k(P)
    \frac{
    \langle \mathbf{p}_k', \mathbf{u} \rangle
    \langle \mathbf{w}_k, \mathbf{u} \rangle
    }{
    \langle \mathbf{p}_k', \mathbf{w}_k \rangle
    }
    .
\end{equation}
Use Lemma~\ref{LemAngleEqDeltaPZ} and some elementary trigonometry to obtain
\begin{equation}\label{EqPrePerron}
\begin{aligned}
    \delta^2(P; Z_k^{(0, 1)}(P)) &= 
    \epsilon^2_k(P)
    \frac{ {2^2}
    \sin^2(\theta(P_k')/2)\sin^2(\theta(W_k)/2)
    }{
    \sin^2(\theta(P_k', W_k)/2)
    } 
    \\
   &= 
   \epsilon_k^2(P)
   \left(
  \frac2{
 \cot(\theta(P_k')/2) - \cot(\theta(W_k)/2) 
  } 
   \right)^2.
    \end{aligned}
\end{equation}

To finish the proof of Theorem~\ref{ThmPerron}, we define
\begin{equation}\label{pkwk}
    \mathbf{p}_k'' 
    =
    \begin{pmatrix}
    p_1 \\ p_2 \\ p_3
    \end{pmatrix}
    = H \mathbf{w}_k.
\end{equation}
It is easy to see that $H\mathbf{u} = \mathbf{u}$ and therefore Lemma~\ref{OrthgonalityOfMd} gives
\[
\mathbf{p}_k'' = H \cdot U_{d_k}H \cdots U_{d_1} H\mathbf{u} 
=
M_{d_k} \cdots M_{d_1} \mathbf{u}.
\]
This shows that $P_k''$, whose Romik digit expansion is $[d_k, d_{k-1}, \dots, d_1, 3^{\infty}]_{\QQQ}$ by definition,
is indeed represented by the vector $\mathbf{p}_k''$.
Using \eqref{pkwk}, we can apply Lemma~\ref{LemRefereeRequest}  to obtain
\[
    \cot\left(\frac{\theta(W_k)}2\right)
    +
    \cot\left(\frac{\theta(P_k'')}2\right)
    =2.
\]
Hence \eqref{EqPrePerron} becomes
\begin{align*}
    \delta^2(P; Z_k^{(0, 1)}(P)) &= 
   \epsilon_k^2(P)
   \left(
  \frac2{
 \cot(\theta(P_k')/2) - \cot(\theta(W_k)/2) 
  } 
   \right)^2 \\
   &=
   \epsilon_k^2(P)
   \left(
  \frac2{
 \cot(\theta(P_k')/2) + \cot(\theta(P_k'')/2) - 2
  } 
   \right)^2 \\   
   &=
   \epsilon_k^2(P)
   \left(
  \frac{\sqrt2}{
\|P_k'\| + \|P_k''\|
  } 
   \right)^2.
\end{align*}
\end{proof}

\begin{corollary}\label{CorThmPerron}
	Let $P\in\QQQ$ be an irrational point.
	For each $k = 1, 2, \dots$, define $P'_k$ and $P''_k$ as in Theorem~\ref{ThmPerron}.
	Also, define $(P^{\vee})'_k$ and $(P^{\vee})''_k$ likewise for $P^{\vee}$.
	Then
	\[
		L(P) =
        \frac1{\sqrt2} 
        \limsup_{k\to\infty}
		\max
		\left\{
		    \| P_k' \| + \| P_k'' \| 
		    ,
		    \| (P^{\vee})_k' \| + \| (P^{\vee})_k'' \| 
		\right\}.
	\]
\end{corollary}
\begin{proof}
This is an immediate consequence of Perron's formula and Proposition~\ref{ThmDeltaPcheck}.
\end{proof}

\section{Doubly infinite admissible sequences}\label{SecAdmissibleSequences}
\label{SecRomikSequence}
Having established basic properties of Romik digit expansions in \S\ref{RomikWay}, we now focus on $P\in\QQQ$ with $L(P) \le 2$. 
Our presentation in \S\ref{SecRomikSequence} and \S\ref{SecCombinatorics} is a close adaptation of Bombieri's exposition in \cite{Bom07}.

\subsection{Infinite sequences and doubly infinite sequences}\label{SecInfiniteSeqDoublyInfiniteSeq}
Let $\AAA$ and $I$ be sets. 
We denote by $\AAA^I$ the set of all functions on $I$ with values in $\AAA$.
The set $I$ is understood to be an \emph{index set} and we will use the following three index sets: $\mathbb{N}$ (the set of all positive integers), $\mathbb{Z}_{\le 0}$ (the set of all nonpositive integers) and $\mathbb{Z}$.
To specify an element of
$\AAA^{\mathbb{N}}$,
$\AAA^{\mathbb{Z}_{\le0}}$, or
$\AAA^{\mathbb{Z}}$,
we list its values on the index set, that is,
\[
[a_1, a_2, \dots]\in\AAA^{\mathbb{N}},
\quad
[\dots, b_{-2}, b_{-1}, b_0]
\in\AAA^{\mathbb{Z}_{\le0}}
,
\quad
[\dots, c_{-1}, c_0, c_1, \dots]
\in
\AAA^{\mathbb{Z}}
\]
with $a_i, b_j, c_k \in \AAA$.
Any element of $\AAA^{\mathbb{N}}$ is called a \emph{sequence (with values in $\AAA$) infinite to the right}.
Likewise, an element of $\AAA^{\mathbb{Z}_{\le0}}$ is called a \emph{sequence infinite to the left}.

Suppose that $E=[e_1, e_2, \dots] \in \AAA^{\mathbb{N}}$.
The \emph{reverse} $E^*$ of $E$ is the element of $\AAA^{\mathbb{Z}_{\le0}}$ obtained by writing $E$ backwards, that is,
\[
E^*= [ \dots, e_2, e_1] \in \AAA^{\mathbb{Z}_{\le0}}.
\]
When $E=[e_1, e_2, \dots]$ and $F = [f_1, f_2, \dots]$ are elements in $\AAA^{\mathbb{N}}$, we can construct an element of $\AAA^{\mathbb{Z}}$ by concatenating $E^*$ and $F$.
Namely, we define $E^*|F$ to be the element of $\AAA^{\mathbb{Z}}$ whose value at $k$ is given by
\[
\begin{cases}
f_k & \text{ if } k > 0,\\
e_{1-k} & \text{ if } k \le 0.
\end{cases}
\]
Conversely, any element of $\AAA^{\mathbb{Z}}$ is written by $E^*|F$ for some $E, F\in \AAA^{\mathbb{N}}$.

We define an equivalence relation on $\AAA^{\mathbb{Z}}$ as follows.
Say that 
\[[\dots, c_{-1}, c_0, c_1, \dots] \sim [\dots, c_{-1}', c_0', c_1', \dots]\]
if there exists a (fixed) $k$ such that $c_j = c_{j+k}'$ for all $j \in \mathbb{Z}$.
We call an equivalence class in $\AAA^{\mathbb{Z}}$ a \emph{doubly infinite sequence} (with values in $\AAA$). 
A \emph{section} of a doubly infinite sequence will mean an element in the equivalence class.
If $C$ is a double infinite sequence, one of whose sections is $[\dots, c_{-1}, c_0, c_1, \dots]$, then we often abuse notation to write
\[
C = [\dots, c_{-1}, c_0, c_1, \dots]
\]
whenever there is no danger of confusion.
For any doubly infinite sequence $C=[\dots, c_{-1}, c_0, c_1, \dots]$ we define its \emph{reverse} $C^*$ to be $C^* = [\dots, c_1, c_0, c_{-1}, \dots]$.

In the present section, we will consider sequences with values in $\AAA = \{ 1, 2, 3 \}$, which we call \emph{Romik sequences}.

\subsection{Definitions}
To characterize $P\in \QQQ$ with $L(P)\le2$, we will study combinatorial properties of digit sequences of such $P$.
Therefore it will be convenient for us to identify $P$ with an element in $\{ 1, 2, 3\}^\mathbb{N}$ 
using the digit expansion of $P$ we introduced in \S\ref{RomikWay}, so that
\[
	P= [d_1, d_2, \dots] \in \{ 1, 2, 3 \}^{\mathbb{N}}.
\]
For each such $P$, we define
\[
	P^{\vee} = [d_1^{\vee}, d_2^{\vee}, \dots]\in \{ 1, 2, 3 \}^{\mathbb{N}}
\]
where
\begin{equation}\label{ConjugateDigit}
	d_j^{\vee} = \begin{cases}
		3 & \text{ if } d_j = 1,\\ 
		2 & \text{ if } d_j = 2,\\ 
		1 & \text{ if } d_j = 3.\\ 
	\end{cases}
\end{equation}
Note that this definition is compatible with Definition~\ref{DefinitionPCheck}.
Also, for a doubly infinite Romik sequence $T$, we can define $T^{\vee}$ in a similar way. 
That is, if $T = [\dots, t_{-1}, t_0,t_1,\dots]$, then $T^{\vee}$ is defined to be
\[
T^{\vee} = [\dots, t_{-1}^{\vee}, t_0^{\vee}, t_{1}^{\vee}, \dots].
\]
For $P, Q \in \{ 1, 2, 3\}^{\mathbb{N}}$, 
we define $L(P^*|Q)$ to be
\begin{equation}\label{EqDefLSection}
L(P^*|Q) =
		\frac{\| P \| + \| Q \|  }{\sqrt2}.
\end{equation}
(See Corollary~\ref{CorThmPerron}.)
Finally, for a doubly infinite Romik sequence $T$, we define
\emph{the Lagrange number of} $T$ to be
\begin{equation}\label{DefinitionL}
	L(T) = 
	\sup_{P^*|Q } 
	 \left\{
	 \max\left(
	L(P^* | Q),
	L((P^{\vee})^* | Q^{\vee})\right)
	\right\}
\end{equation}
where $\{P^* | Q\}$ runs over all sections of $T$.
It is not difficult to show that the definition \eqref{DefinitionL} is equivalent to
\begin{equation}\label{DefinitionL2}
	L(T) = 
	\max \left\{
	\sup_{P_1^*|Q_1 } L(P_1^* | Q_1),
	\sup_{P_2^*|Q_2 } L(P_2^* | Q_2),
	\right\}
\end{equation}
where $\{P_1^* | Q_1\}$ and $\{P_2^*|Q_2 \}$ run over all sections of $T$ and $T^{\vee}$, respectively.
Also, the definition of $L(P^*| Q)$ is symmetric in $P$ and $Q$.
Therefore, 
\begin{equation}\label{LInvariance}
	L(T) = L(T^*) = L(T^{\vee}) = L((T^{\vee})^*).
\end{equation}

We say that $T$ is \emph{admissible} if $L(T) \le 2$, 
and \emph{strongly admissible} if $L(T) < 2$.
From \eqref{LInvariance}, we see that $T$ is admissible (or strongly admissible) if and only if any one of $\{ T,T^*, T^{\vee}, (T^{\vee})^*\}$ is admissible (or strongly admissible).

\begin{proposition}[Bombieri's trick]\label{PropBombieriTrick}
   Suppose that $P \in \{ 1, 2, 3\}^{\mathbb{N}}$. Then there exists a doubly infinite Romik sequence $T$ with $L(P) = L(T)$. 
\end{proposition}
\begin{proof}
The proof is nearly identical to Bombieri's argument in page 191 of \cite{Bom07} and we omit it.
\end{proof}

\subsection{The Romik system on nonnegative real numbers}

We let $\mathrm{GL}_2(\mathbb{R})$ act on $\mathbb{R} \cup \{ \infty \}$ via fractional linear action, that is,
\begin{equation}\label{EqFractionalLinearAction}
\begin{pmatrix}
a & b \\ c & d
\end{pmatrix}
\cdot t 
=
\frac{at + b}{ct + d}.
\end{equation}
Recall that we defined the dynamical system $([0,\infty], \TTT_{[0,\infty]})$ in \eqref{EqTreal}, which is conjugate to $(\QQQ, \TTT)$ via the map $P \mapsto \| P \|$ as in \eqref{EqCommuteDiagram}.
The following proposition shows that the actions of
\begin{equation}\label{EqDefNd}
    N_1 =
    \begin{pmatrix}
    1 & 0 \\
    \sqrt2 & 1 \\
    \end{pmatrix},
    \quad
    N_2 =
    \begin{pmatrix}
    1 & \sqrt2 \\
    \sqrt2 & 1 \\
    \end{pmatrix},
    \quad
    N_3 =
    \begin{pmatrix}
    1 & \sqrt2 \\
    0 & 1 \\
  \end{pmatrix}.
\end{equation}
on $[0, \infty]$ correspond to the inverse branches of $\TTT_{[0, \infty]}$, playing the role of $M_d$ on $\QQQ$ (cf.~Proposition~\ref{DigitsAndMatrixMultiplication}).

\begin{proposition}\label{PropNormShift}
Let $P\in \{1,2,3\}^{\mathbb{N}}$ and let $[d, P]$ be the element of $\{ 1,2,3\}^{\mathbb{N}}$ obtained by concatenating a digit $d$ and $P$. Then
\[
\| [d, P] \| = N_d \cdot  \| P \|
\]
for $d= 1, 2, 3$.
\end{proposition}
\begin{proof}
This is a simple consequence 
of straightforward calculation based on the commutative diagram in \eqref{EqCommuteDiagram} and we omit the proof.
\end{proof}

One easily checks that $\| (\frac 35, \frac45) \| = 1/\sqrt2$
and $\| (\frac 45, \frac35) \| = \sqrt2$, thus the cylinder sets $\Cyl(1), \Cyl(2), \Cyl(3)$ are projected homeomorphically onto the intervals 
$[0, \frac1{\sqrt2}]$,
$[\frac1{\sqrt2}, \sqrt{2}]$,
$[\sqrt{2}, \infty]$, as shown in 
Figure~\ref{FigCylindersOnReals}.
\begin{figure}
    \begin{center}
        \begin{tikzpicture}[xscale = 2.3]
            \draw[|-|] (0, 0) 
            node[below]{$0$}
            -- (0.707, 0)
            node[midway, above]{$\| \Cyl(1) \|$}
            node[below]{$\tfrac{1}{\sqrt2}$};
            \draw[|-|] (0.707, 0)
            -- (1.414, 0)
            node[midway, above]{$\| \Cyl(2) \|$}
            node[below]{$\tfrac{\sqrt2}1$};
            \draw[|-|] (1.414, 0)
            -- (4.243, 0)
            node[midway, above]{$\| \Cyl(3) \|$}
            node[below]{$\infty$}
            ;
        \end{tikzpicture}
    \end{center}
    \caption{Images of cylinder sets $\Cyl(1), \Cyl(2), \Cyl(3)$ under the stereographic projection $P \mapsto \| P \|$.}
    \label{FigCylindersOnReals}
\end{figure}
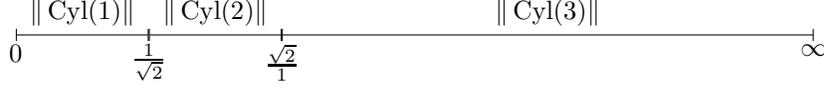
In particular, if $P_d \in \Cyl(d)$ for $d = 1, 2, 3$, it follows that
\begin{equation}\label{EqMonotonicity}
\| P_1 \| \le \| P_2 \| \le \| P_3 \|.
\end{equation}
Another easy observation is that
$\| P \| \le \| Q \|$ if and only if any one of the following three inequalities holds:
\begin{equation}\label{EqDigitMonotonicity}
\begin{gathered}
N_1 \cdot 
\| P \| \le 
N_1 \cdot 
\| Q \|,
\qquad
N_3 \cdot 
\| P \| \le 
N_3 \cdot 
\| Q \|, \\
N_2 \cdot 
\| P \| \ge 
N_2 \cdot 
\| Q \|.
\end{gathered}
\end{equation}
This follows from \eqref{EqFractionalLinearAction} and the definition \eqref{EqDefNd} of $N_d$.
\begin{proposition}\label{PropDigitComparison}
Suppose $P, Q\in \{ 1,2, 3\}^{\mathbb{N}}$
with 
\[
P = [d_1, \dots, d_k, d_P, \dots]
\text{ and }
Q = [d_1, \dots, d_k, d_Q, \dots]
\]
where $d_P < d_Q$ and $\| P \| \neq \|Q \|$.
If  $[d_1,\dots, d_k]$ contains an even number of $2$'s,  then $\| P \| < \| Q \|$. 
If $[d_1,\dots, d_k]$ contains an odd number of $2$'s, then $\| P \| > \| Q \|$. 
\end{proposition}

\begin{proof}   
This is immediate from
Proposition~\ref{PropNormShift}, 
\eqref{EqMonotonicity}, and
\eqref{EqDigitMonotonicity}.
\end{proof}

\begin{proposition}
  Suppose $P, Q\in \{ 1,2, 3\}^{\mathbb{N}}$.
  We have $\| P \| \le \| Q \|$ if and only if $\| Q^{\vee} \| \le \| P^{\vee} \|$.
  \label{prop:check_order_reversing}
\end{proposition}
\begin{proof}
  If $\| P \| = \| Q \|$ then $\| P^{\vee} \| = \| Q^{\vee} \|$ and the proposition is trivially true.
  Assume now $\| P \| \neq \| Q \|$.
  Let $[d_1, \dots, d_k]$ be the longest common prefix (possibly empty) of $P$ and $Q$, so that 
  we write
\[
P = [d_1, \dots, d_k, d_P, \dots]
\ \text{ and } \
Q = [d_1, \dots, d_k, d_Q, \dots]
\]
with $d_P \neq d_Q$, as in Proposition~\ref{PropDigitComparison}.
Then we have $d_P < d_Q$ if and only if $d_P^{\vee}> d_Q^{\vee}$ (cf.~Definition~\ref{DefinitionPCheck}).
Since $[d_1, \dots, d_k]$ and $[d_1^{\vee}, \dots, d_k^{\vee}]$ contain the same number of 2's, the statement of the proposition now follows from this and Proposition~\ref{PropDigitComparison}.
\end{proof}

\begin{proposition}\label{PropRealCylinder}
Fix a finite sequence of Romik digits $[d_1, \dots, d_k]$ and let
\[
    \begin{pmatrix}
    p & p' \\
    q & q' \\
    \end{pmatrix}
    =
    N_{d_1} \cdots N_{d_k}.
\]
Assume that $[d_1, \dots, d_k]$ contains an even number of 2's.
Then $P \in \Cyl(d_1, \dots, d_k)$ if and only if
\[
\frac{p'}{q'} \le \| P \| \le \frac pq.
\]
Now, assume that $[d_1, \dots, d_k]$ contains an odd number of 2's.
Then $P \in \Cyl(d_1, \dots, d_k)$ if and only if
\[
\frac pq \le \| P \| \le \frac{p'}{q'}.
\]
\end{proposition}

\begin{proof}
As in \eqref{EqZboundary} of \S\ref{SecCylinder}, let
\[
Z^{(1, 0)} = [d_1, \dots, d_k, 1^{\infty}], \qquad
Z^{(0, 1)} = [d_1, \dots, d_k, 3^{\infty}].
\]
Since $\| [1^{\infty}] \|  =  \| (1, 0) \| = 0$ and $\| [3^{\infty}] \|= \| (0, 1) \| = \infty$, Proposition~\ref{PropNormShift} yields
\begin{align*}
\| Z^{(1, 0)} \| &= 
N_{d_1}\cdots N_{d_k} \cdot \| [1^{\infty}] \| 
=
N_{d_1}\cdots N_{d_k} \cdot 0
= \frac{p'}{q'}, \\
\| Z^{(0, 1)} \| &= 
N_{d_1}\cdots N_{d_k} \cdot \| [3^{\infty}] \| 
=
N_{d_1}\cdots N_{d_k} \cdot \infty
= \frac{p}{q}. 
\end{align*}
On the other hand,  
we see from Proposition~\ref{PropDigitComparison} that, 
if $[d_1,\dots, d_k]$ contains an even number of 2's,
then 
$
\| Z^{(1, 0)} \|
\le 
\| Z^{(0, 1)} \|,
$
therefore,
$\frac{p'}{q'} \le \frac pq$.
Since the stereographic projection is a homeomorphism, the cylinder set $\Cyl(d_1, \dots, d_k)$ must be mapped onto the interval $
[\frac{p'}{q'} , \frac pq]
$
and the proposition follows from this.
The proof for the case when $[d_1,\dots, d_k]$ contains an odd number of 2's is similar.
\end{proof}
Let 
\begin{equation}\label{EqDefJ}
J = 
    \begin{pmatrix}
    0 & 1 \\
    1 & 0 \\
    \end{pmatrix}.
\end{equation}
Then $J^2 = I_2$ (the $2\times 2$ identity matrix) and 
\begin{equation}\label{LemJ}
J 
    \begin{pmatrix}
    a & b \\
    c & d \\
    \end{pmatrix}
J
=
    \begin{pmatrix}
    d & c \\
    b & a \\
    \end{pmatrix}
\end{equation}
for any $\left( \begin{smallmatrix}
a & b \\ c & d 
\end{smallmatrix} \right)$.
In particular, $JN_d J = N_{d^{\vee}}$ for $d = 1, 2, 3$. (cf.~\eqref{EqDefNd})
\begin{proposition}\label{PropNdCheck}
Fix a finite sequence of digits $[d_1, \dots, d_k]$ and let
\[
    \begin{pmatrix}
    p & p' \\
    q & q' \\
    \end{pmatrix}
    =
    N_{d_1} \cdots N_{d_k}.
\]
Then we have
\[
    N_{d^{\vee}_1} \cdots N_{d^{\vee}_k}
    =
    \begin{pmatrix}
    q' & q \\
    p' & p \\
    \end{pmatrix}
\quad
\text{ and }
\quad
    N_{d_k} \cdots N_{d_1}
    =
    \begin{pmatrix}
    q' & p' \\
    q & p \\
    \end{pmatrix}.
\]
\end{proposition}
\begin{proof}
Using \eqref{LemJ}, we have
\[
N_{d_1^{\vee}}
\cdots
N_{d_k^{\vee}}
= J
N_{d_1}
\cdots
N_{d_k}
J
=
J
    \begin{pmatrix}
    p & p' \\
    q & q' \\
    \end{pmatrix}
    J
    =
    \begin{pmatrix}
    q' & q \\
    p' & p \\
    \end{pmatrix}
    ,
\]
which proves the first assertion.
For the second, we take the transpose of the above equation to obtain
\[
    \begin{pmatrix}
    q' & p' \\
    q & p \\
    \end{pmatrix}
=
(
N_{d_1^{\vee}}
\cdots
N_{d_k^{\vee}}
)^t
=
N_{d_k^{\vee}}^t
\cdots
N_{d_1^{\vee}}^t.
\]
However, \eqref{EqDefNd} shows $N_{d^{\vee}}^t = N_d$ for $d = 1, 2, 3$ and this establishes the second assertion of the proposition.
\end{proof}

\begin{remark}
  One of the reasons why we choose to modify the standard stereographic projection as in Definition~\ref{DefNormRomikSequence}
  is to obtain the identity $N_{d^{\vee}}^t = N_d$ for $d = 1, 2, 3$.
  This also lets us obtain a cleaner formula in Proposition~\ref{PropLpurelyPeriodic} below.
  \label{rem:stereo_graphic_projection_justification}
\end{remark}

To finish the subsection, we prove a formula for $L(P^*|Q)$ when $P^*|Q$ is a section of a \emph{purely periodic} doubly infinite sequence.
Fix a finite digit sequence $\Pi = [d_1, \dots, d_k]$
and let $T$ be a doubly infinite Romik sequence obtained by repeating $\Pi$ to both directions, that is,
\[
T = [\dots, d_k, d_1, d_2, \dots, d_k, d_1, \dots, d_k, d_1, \dots].
\]
Additionally, we let 
\begin{align*}
P &= [d_k, d_{k-1}, \dots, d_1, d_k, \dots, d_1, \dots] ,\\
Q &= [d_1, d_2, \dots, d_k, d_1, \dots, d_k, \dots].
\end{align*}
Then $P^*|Q$ is a section of $T$, which we denote by ${}^{\infty}\Pi| \Pi^{\infty}$.
\begin{proposition}\label{PropLpurelyPeriodic}
Fix a finite sequence of Romik digits $\Pi = [d_1, \dots, d_k]$ and let
\[
    N
    =
    \begin{pmatrix}
    p & p' \\
    q & q' \\
    \end{pmatrix}
    =
    N_{d_1} \cdots N_{d_k}.
\]
Then
\[
L({}^{\infty}\Pi |\Pi^{\infty})
=
\frac{\sqrt{
\Tr(N)^2 - 4\det(N)
}}{\sqrt2 q}.
\]
\end{proposition}
\begin{proof}
For 
$
A =
\left(
\begin{smallmatrix}
a & b \\ c & d 
\end{smallmatrix}
\right)
$,
we let 
\[
\Delta(A) = \Tr(A)^2 - 4\det(A) = (a - d)^2 + 4bc.
\]
Then, as an easy exercise, one can prove that 
$A\cdot x_{\pm} = x_{\pm}$
whenever
\begin{equation}\label{EqEigenvector}
x_{\pm} =\frac{(a-d) \pm\sqrt{\Delta(A)}}{2c}.
\end{equation}
Also, if $bc > 0$, we have $x_+>0$ and $x_-<0$.

As before, we let
\begin{align*}
P &= [d_k, d_{k-1}, \dots, d_1, d_k, \dots, d_1, \dots], \\
Q &= [d_1, d_2, \dots, d_k, d_1, \dots, d_k, \dots].
\end{align*}
Then Proposition~\ref{PropNormShift} gives 
\[
N_{d_k} \cdots N_{d_1} \cdot \| P \| = \| P \|,  \qquad
N_{d_1} \cdots N_{d_k} \cdot \| Q \| = \| Q \|. 
\]
Also, Proposition~\ref{PropNdCheck} implies 
$
\Delta(N_{d_1}\cdots N_{d_k})
=
\Delta(N_{d_k}\cdots N_{d_1})
$.
So we obtain from \eqref{EqEigenvector} that
\[
\| P \|
=
\frac{q' - p + \sqrt{\Delta(N)}}{2q}, \qquad
\| Q \|
=
\frac{p - q' + \sqrt{\Delta(N)}}{2q}.
\]
Therefore
\[
L(P^*|Q) =
\frac{\| P \| + \| Q \|}{\sqrt2}
=
\frac{\sqrt{\Delta(N)}}{\sqrt2 q}.  \qedhere
\]
\end{proof} 

\subsection{Combinatorial properties of doubly infinite admissible sequences}
To ease notation, let us omit commas between digits, whenever there is no fear of confusion.  
For example, we write $33$ instead of $3, 3$.
Also, the concatenated sequence $[d, P]$ shall be shortened as $dP$.

\begin{proposition}\label{ForbiddenWords}
Suppose $T$ is a doubly infinite admissible Romik sequence. Then, the following blocks cannot appear in $T$:
\[ 
	33, 11, 232, 212.
\]
\end{proposition}
\begin{proof}
	Assume that $T$ contains $33$, so that it contains a section $P^* | 33Q$.
	Since $
N_3N_3 = 
\left(
\begin{smallmatrix}
1 & 2\sqrt2 \\
0 & 1 \\
\end{smallmatrix}
\right)
$,
Proposition~\ref{PropRealCylinder} gives
$\| 33 Q \| \ge 2\sqrt2$,
so that
	\[
		L(P^*| 33 Q) = \frac{\| P \| + \| 33Q \|}{\sqrt2} \ge2,
	\]
	where the equality holds only if $\|P\| = \|Q\| = 0$.
	But, if either $\| P \|$ or $\| Q \|$ is zero, then $T$ is clearly not admissible. 
	Therefore the inequality above must be strict. 
	This shows that an admissible $T$ cannot contain $33$. 
	Applying the same argument to $T^{\vee}$, we see that $11$ is also forbidden.
	
	For $232$, we apply Proposition~\ref{PropRealCylinder} with 
	\[
N_2 = 
\begin{pmatrix}
1 & \sqrt2 \\
\sqrt2 & 1 
\end{pmatrix}
\quad
\text{and}
\quad
N_3N_2 = 
\begin{pmatrix}
3 & 2\sqrt2 \\
\sqrt2 & 1 
\end{pmatrix}
	\]
	to obtain
	\[
\| 2P \| \ge 
\frac1{\sqrt2}
\quad
\text{and}
\quad
\| 32 Q \| 
\ge 
\frac3{\sqrt2},
	\]
	so that
	\[
		L(P^*2| 32 Q) = 
		\frac1{\sqrt2}
		\left(
\| 2P \| 
+
\| 32 Q \| 
		\right)
		\ge 
		\frac1{\sqrt{2}}
		\left(
\frac1{\sqrt2}
+
\frac3{\sqrt2}
		\right)
		=2.
	\]
	Again, this inequality must be strict and we see that $232$ is forbidden.
	Likewise, by applying the same argument to $T^\vee$, we see that $212$ is also forbidden.
\end{proof}

\begin{proposition}\label{AdmissibleComparison}
	Let $P, Q \in \{1, 2,3\}^{\mathbb{N}}$.
	Then,
	$L(P^*2 | 3 1 Q) \le 2$ if and only if $\|P\| \ge \|Q\|$. Moreover, $L(P^*2 | 3 1 Q) = 2$ if and only if $P = Q$. 
\end{proposition}
\begin{proof}
Using Proposition~\ref{PropNormShift}, we have
\[
\| 2P \| =  \frac{\| P \| + \sqrt2}{\sqrt2 \| P \| + 1}
\ \text{ and } \ 
 	\| 31Q \| = 
	\frac{3\| Q \| + \sqrt2}{\sqrt2 \| Q \|+ 1},
\]
so that
\begin{align*}
L(P^*2 | 31Q)
&=
\frac1{\sqrt2}
\left( 
\frac{\| P \| + \sqrt2}{\sqrt2 \| P \| + 1}
+
\frac{3\| Q \| + \sqrt2}{\sqrt2 \| Q \|+ 1}   
\right) \\
&=
2 + 
\frac{\| Q \|-\| P \|}{\sqrt2(\sqrt2\| P \| + 1)(\sqrt2\| Q \|+1)}.  
\end{align*}
The statements in the proposition follow from this.
\end{proof}

\begin{proposition}\label{AdmissibleComparison2}
For any $P_1, P_2, Q_1, Q_2 \in \{1, 2,3\}^{\mathbb{N}}$, we have 
\[ 
L(P_1^* 2 | 3 1 Q_1)  \ge L(P_2^*2| 1 3  Q_2).
\]
\end{proposition}

\begin{proof}
Proposition~\ref{PropRealCylinder} shows that 
\[
	\frac1{\sqrt2} \le \| 2P \| \le \sqrt2
\]
for any $P \in \{ 1, 2, 3 \}^{\mathbb{N}}$. 
So
\[
	\| 2P_1 \| - \| 2P_2 \|
	\ge \frac1{\sqrt2} - \sqrt2 = -\frac1{\sqrt2}.
\]
Likewise, we use
\[
N_1N_3 =
\begin{pmatrix}
1 & \sqrt2 \\ \sqrt2 & 3 
\end{pmatrix}
\quad
\text{and}
\quad
N_3N_1 =
\begin{pmatrix}
3 & \sqrt2 \\ \sqrt2 & 1 
\end{pmatrix}
\]
to deduce
\[
\| 31Q_1 \| - \| 13 Q_2 \| 
\ge \sqrt2 - \frac1{\sqrt2} 
=\frac1{\sqrt2}.
\]
Combining these two inequalities,
\[
	 L(P_1^* 2 | 3 1 Q_1)  - L(P_2^*2| 1 3  Q_2)
	 \ge 0
	 . \qedhere
\]
\end{proof}

\begin{proposition}\label{UnitWord}
Suppose that a doubly infinite Romik sequence $T$ is admissible.
Then neither $2 (31)^k  3 2$ nor $2 (13)^k 1 2$ can appear in $T$ for any $k\ge0$.
\end{proposition}

\begin{proof}
It is sufficient to prove only the first kind is forbidden because $ 2 (31)^k  3 2 =(2 (13)^k 1 2)^{\vee}$.
Assume that $T$ contains such a sequence. 
Choose a minimal $k$. 
Note that $k=0$ would produce a forbidden block $232$ (Proposition~\ref{ForbiddenWords}).
So, $k$ is at least 1.
	
Choose a section $\cdots 2 | 31 (31)^{k-1}32 \cdots$ of $T$ and write it as
\[
	P^* 2|31 Q,
\]
with $P = [p_1, p_2, \dots]$ and $Q = (31)^{k-1} 3 2\dots$.
From Proposition~\ref{AdmissibleComparison}, we have $\|P\|\ge \|Q\|$. 
We will use this and Proposition~\ref{PropDigitComparison} to reveal the digits of $P$ successively. 
Consider the following table:
\[
	\begin{array}{lllcll}
		P: & p_1 & p_2 & \cdots & & \\
		Q: & 3 & 1 & \underbrace{3 \, 1 \, \cdots 3\, 1 }_{\text{($k-2$)-times}} & 3 & 2 \\
	\end{array}
\]
First, $\|P\|\ge \|Q\|$ implies that $p_1=3$. 
Then notice that neither $p_2 = 3$ nor $p_2=2$ is possible because it would produce a forbidden block $33$ or $232$ in $T$, so $p_2=1$.
Next, apply the condition $\|P\|\ge \|Q\|$ again to force $p_3=3$. Then $p_4 = 3$ is impossible because this would create a $33$.
Also, $p_4 = 2$ is impossible because this would violate the minimality of $k$. 
This way, we continue to reveal the digits of $P$ until we arrive at
\[
	P = (31)^{k-1} 3 \dots.
\]
Then there is no possible value for the next digit of $P$; 3 would produce a forbidden block $33$, 2 would violate the minimality of $k$, and 1 is not possible because $\|P\|\ge \|Q\|$. 
This finishes the proof that $T$ cannot contain $2 (31)^k 3 2$.
\end{proof}

\begin{proposition}\label{TwoBlock}
Suppose that a doubly infinite Romik sequence $T$ is admissible. 
Then, the following blocks are forbidden in $T$:
\[
	1\,2^{2n+1}\,3, \quad
	3\,2^{2n+1}\,1, \quad
	3\,2^{2n}\,3, \quad
	1\,2^{2n}\,1,
\]
for any $n\ge0$.
\end{proposition}

\begin{proof}
	We will prove that $1 2^k 3$ for an odd $k$ and $1 2^k 1$ for an even $k$ are both forbidden, as the other cases will follow from this with $T^\vee$ replacing $T$.
	Assume the contrary and suppose that an admissible $T$ contains one of these. 
	Take a minimal $k$.

	We know that $k$ is at least 1 because $11$ is already forbidden.
	Suppose $k\ge1$ is odd.
	In particular, $T$ contains $12^k3$. 
	Since 232 and 33 are forbidden (Proposition~\ref{ForbiddenWords}) the block $12^k3$ can be extended to the right only as $12^k31$, resulting in a section 
	\[
		P^*12^{k-1}2|31 Q
	\]
	of $T$ for some $P$ and $Q$.
	Apply Proposition~\ref{AdmissibleComparison} to obtain $\| 2^{k-1}1P\|  \ge \| Q\| $.
	We compare the digits of $P$ and $Q$ in the following table:
	\[
		\begin{array}{rcll}
			2^{k-1}1P: & \overbrace{2 \, \cdots \, 2}^{\text{$k-1$ times}} & 1 & p_1 \cdots \\
			Q: & q_1  \, q_2 \cdots  & q_k & q_{k+1}\cdots    \\
		\end{array}
	\]
	First, $q_1 = 1$ is not possible because it would produce a forbidden block $11$ in $T$. Also, $q_1 = 3$ is impossible because of the inequality $\|2^{k-1}1P\| \ge \|Q\|$, (cf.~Proposition~\ref{PropDigitComparison}) so we conclude $q_1 = 2$. 
	For $q_2$, we see that $q_2 = 3$ violates the minimality of $k$ and that $q_2 = 1$ would invalidate $\| 2^{k-1}1P\| \ge \| Q\| $.
	So, the only possible choice is $q_2 = 2$. 
	For $q_3$, the inequality $\| 2^{k-1}3P\| \ge \|Q\| $ says this time that $q_3 = 3$ is impossible. If $q_3 = 1$, then it will produce $12^l1$ with $l<k$, violating the minimality of $k$. So, $q_3 =2$.
	Continuing this way using the minimality of $k$ and the inequality $\|2^{k-1}3P\| \ge \|Q\|$, we obtain $q_1 = \dots = q_{k-1} = 2$.
	Finally, for $q_k$, the inequality $\|2^{k-1}3P \|\ge \|Q\|$ implies $q_k = 3$, but this would violate the minimality of $k$.
	This proves that $T$ cannot have a sequence $12^k3$ for an odd $k\ge 1$. 
	
	Similarly, one can derive a contradiction for the case when the minimal $k$ is even.
	We leave the detail for the reader.
\end{proof}

\begin{proposition}\label{LBounds}
	Suppose that $P$ and $Q$ are Romik sequences.
	\begin{enumerate}[font=\upshape, label=(\alph*)]
		\item $L(P^* 3 | 1 Q) = L(P^* | 3  1  Q)$ and $L(P^* 1 | 3 Q) = L(P^* 13|  Q)$.
		\item $L(P^* 2 | 2  Q) \le 2$.
		\item $L(P^* 1 3 | 1  Q) \le 2$.
	\end{enumerate}
\end{proposition}

\begin{proof}
	For the first equality of (a), it suffices to notice
	from Proposition~\ref{PropNormShift} that
	\[
\| 3P \| + \| 1Q \| = \| P \| +\sqrt2 + \| 1Q \| = \| P \| + \| 31Q\|.
	\]
The second equality in (a) is proven similarly.  
For (b) and (c), we apply Proposition~\ref{PropRealCylinder} to utilize upper bounds of cylinder sets:
\[
\| 1P_1 \| \le \frac1{\sqrt2},\quad
\| 2P_2 \| \le \sqrt2,\quad
\| 31P_3 \| \le \frac3{\sqrt2} %\\
\]
for any $P_1, P_2, P_3 \in \{ 1, 2,3 \}^{\mathbb{N}}$.
So,
\[
L(P^* 2 | 2  Q) 
\le 
\frac1{\sqrt2}
\left(
\sqrt2
+
\sqrt2
\right)
= 2
\]
and
\[
L(P^* 1 3 | 1  Q) 
\le 
\frac1{\sqrt2}
\left(
\frac1{\sqrt2}
+
\frac3{\sqrt2}
\right)
= 2. \qedhere
\]
\end{proof}
\begin{proposition}\label{Axiom}
	A doubly infinite Romik sequence $T$ is admissible if and only if:
	\begin{enumerate}[font=\upshape, label=(\roman*)]
		\item none of the blocks $33, 11, 232, 212$ appears in $T$, and 
		\item for every section $P^*2 | 31 Q$ of $T$, $T^*$, $T^{\vee}$ and $(T^*)^{\vee}$, we have $\|P\|\ge \|Q\|$.
	\end{enumerate}
\end{proposition}
\begin{proof}
	Suppose $T$ is admissible. 
	The statements (i) and (ii) are consequences of Propositions~\ref{ForbiddenWords} and \ref{AdmissibleComparison}.
	
	Conversely, assume that $T$ satisfies (i) and (ii). We will show that, for every section 
	$P_1^* d_1|d_2 P_2$
	of $T$ and $T^{\vee}$, 
	we have
	$L(P_1^* d_1|d_2 P_2) \le 2$.
	Of the nine choices of $(d_1, d_2)$ arising from all possible values of $d_1$ and $d_2$,
	we can exclude $(d_1,d_2) = (1, 1)$ and $(3, 3)$ because 33 and 11 are forbidden in $T$.
	In the case $(d_1, d_2) = (2, 2)$, we already have 
	the inequality $L(P_1^* 2|2 P_2) \le 2$
	from  Proposition~\ref{LBounds}.
	If $(d_1, d_2) = (2, 3)$, then the block $d_1d_2$ can only be extended to the right as $231$, so that the section becomes
	\[
		P_1^*2|31P_2',
	\]
	and we already know from (ii) of Proposition~\ref{AdmissibleComparison} that $L(P^*2|31P_2') \le 2$.
	When $(d_1, d_2) = (2, 1)$, it extends to the right as $213$, resulting in a section $P_1^*2|13P_2'$.
	Then we use Proposition~\ref{AdmissibleComparison} to obtain
	\[
		L(P_1^*2|13P_2') \le 
		L((P_1^{\vee})^*2|31(P_2')^{\vee}),
	\]
	the right hand side of which we already know is $\le 2$.
	Considering reverses and using the fact that $L(P^*|Q)$ is symmetric in $P$ and $Q$, we see that the only remaining case to consider is $(d_1, d_2) = (3, 1)$. 

	The block $d_1d_2 = 31$, when extended to the left, can result in only the following two types of sections
	\[
		P^*23|1Q, 
		\quad
		P^*13|1Q.
	\]
	The part (c) of Proposition~\ref{LBounds} says that $L(P^*13|1Q) \le 2$, covering the second type of section.
	For the first type, we use the part (a) of Proposition~\ref{LBounds} to obtain
	\[
		L(P^*23|1Q) = 
		L(P^*2|31Q),
	\]
	which has already been shown to be $\le2$.
	Therefore we have $L(P_1^*d_1|d_2P_2) \le 2$ for all possible $(d_1, d_2)$ and consequently  $T$ is admissible.
\end{proof}

\section{Combinatorial characterization of admissible sequences}\label{SecCombinatorics}

\subsection{Definitions}\label{SecWordTerminology}
In this section we will study sequences with values in the following two sets: $\AAA_2 = \{a, b\}$ and $\AAA_3 = \{ a, b, a^{\vee} \}$. (See \S\ref{SecInfiniteSeqDoublyInfiniteSeq}.)
We call $\AAA_2 $ and $\AAA_3$ a \emph{two-letter alphabet} and a \emph{three-letter alphabet}, respectively, and a sequence with values in $\AAA_2$  or $\AAA_3$ will be called a \emph{word}.
Additionally, we consider a \emph{finite word} on an either alphabet, which means a finite sequence with values in $\mathscr{A}_2$ or $\mathscr{A}_3$.
The set of all finite words (of any length $< \infty$) forms a free monoid under concatenation. 
\emph{The empty word} shall be regarded as the identity element of such a monoid.

For a finite word $w$ on $\{a, b\}$, we will say $w$ is \emph{odd} if $w$ contains an odd number of $b$'s, and is \emph{even} otherwise.
For any word $w$ on $\{a, b, a^{\vee}\}$, finite, infinite, or doubly infinite,
we define $w^{\vee}$ to be the word obtained by attaching $\vee$ to each letter in $w$, subject to the rule
\begin{equation}\label{EqCheckRule}
    b^{\vee} = b \quad \text{ and } \quad (a^{\vee})^{\vee} = a.
\end{equation}
The empty word will be regarded unchanged under ${}^{\vee}$.
By definition, taking $\vee$ commutes with concatenation, that is,
\begin{equation}\label{CheckCommutesWithConcatenation}
    (w_1 w_2)^{\vee} = w_1^{\vee}w_2^{\vee}
\end{equation}
for $w_1$ and $w_2$, which can be finite or infinite to one direction.

\subsection{Words and digit sequences}
To connect words and digit sequences,
we define a ``digit-substitution'' map $\Pi: \{ a, b, a^{\vee} \}^{\mathbb{N}} \longrightarrow \{ 1, 2, 3\}^{\mathbb{N}}$.
Namely, for each $E\in \{ a, b, a^{\vee}\}^{\mathbb{N}}$, we define $\Pi(E)$ to be the digit sequence, obtained by applying the substitution rule
\begin{equation}\label{AlphabetRule}
    \Pi(a) = 3 \, 1, \quad \Pi(b) = 2,\quad \Pi(a^{\vee}) = 1\, 3
\end{equation}
to each letter in $E$ successively.

Suppose that $T$ is a doubly infinite Romik sequence and $B$ is a doubly infinite word on $\{a, b, a^{\vee}\}$.
If we can recover $T$ from $B$ by the above digit substitution rule, we will say that $B$ is an \emph{associate of} $T$. 
More rigorously, if $T$ has a section $P^*|Q$ and if $B$ has a section $E^*|F$ such that
\[
\Pi(E^{\vee}) = P,  \qquad
\Pi(F) = Q,
\]
then we say that $B$ is an \emph{associate of $T$}.
Notice that we require 
$\Pi(E^{\vee}) = P$,
not
$\Pi(E) = P$.
This is because the letter $a$, when read ``backwards'', should represent $1\, 3$ not $3\, 1$.
For example, if $T$ is a doubly infinite Romik sequence having a section 
\[
P^* | Q = \cdots 3\, 1 | 3\, 1 \cdots
\]
and if a doubly infinite word $B$ with a section $E^*| F = \cdots a | a \cdots$ is an associate of $T$,
then we see that
\begin{align*}
\Pi(E^{\vee}) &=  \Pi(a^{\vee}\cdots ) = 1\, 3 \cdots = P, \\
\Pi(F) 
&= \Pi(a\cdots ) = 3\, 1 \cdots 
= Q.
\end{align*}
It is clear that $B$ is an associate of $T$ if and only if $B^{\vee}$ is an associate of $T^{\vee}$.

We will define a Lagrange number of a doubly infinite word $B$ on $\{ a, b, a^{\vee}\}$ as before. 
Namely, for $E,F \in \{ a, b, a^{\vee}\}^{\mathbb{N}}$, we let 
\begin{equation}\label{EqLwordSection}
L(E^*|F) =
		\frac{\| \Pi(E^{\vee}) \| + \| \Pi(F) \|  }{\sqrt2}
\end{equation}
and define
\begin{equation}\label{DefinitionLB}
	L(B) = 
	\sup_{E^*|F } 
	 \left\{
	 \max\left(
	L(E^* | F),
	L((E^{\vee})^* | F^{\vee})\right)
	\right\}
\end{equation}
where $\{E^* | F\}$ runs over all sections of $B$.
From the definition, we see that
\[
L(B) = 
L(B^*) = 
L(B^{\vee}) = 
L((B^{\vee})^*).
\]
We say that $B$ is \emph{admissible} if $L(B) \le 2$ and is \emph{strongly admissible} 
if $L(B) < 2$.

\begin{proposition}
    Suppose that $B$ is a doubly infinite word on $\{ a, b, a^{\vee} \}$ and $T$ is a doubly infinite Romik sequence.
    If $B$ is an associate of $T$, then $L(B) = L(T)$.
\end{proposition}
\begin{proof}
The two Lagrange numbers $L(T)$ and $L(B)$ are defined in an identical way (cf.~\eqref{EqDefLSection}, \eqref{DefinitionL}, \eqref{EqLwordSection}, \eqref{DefinitionLB}), except for the fact that not all sections of $T$ correspond to those of $B$.
Namely, a section of $T$ may ``break'' a letter $a$ or $a^{\vee}$ in $B$ as $\cdots 3|1 \cdots$ or $\cdots 1|3 \cdots$. 
However the part (a) of Proposition~\ref{LBounds} shows that we do not need to consider such a section of $T$, and
as a result, we conclude $L(T) = L(B)$.
\end{proof}

\subsection{Oriented words}\label{SecOrientation}
We will say that a word $w$ on $\{a, b, a^{\vee}\}$, finite or infinite, is \emph{$\vee$-oriented} if $w$ does not contain any of the words
\begin{equation}\label{ForbiddenSubword}
ab^{2k}a^{\vee}, \quad
a^{\vee} b^{2k}a,\quad
ab^{2k+1}a,\quad
a^{\vee}b^{2k+1}a^{\vee}
\end{equation}
for any $k\ge0$.
Any word of the form \eqref{ForbiddenSubword} is said to be a \emph{forbidden word}.
If any one of  $w, w^* , w^{\vee}, (w^*)^\vee$ is $\vee$-oriented, then the other three are also $\vee$-oriented.

The reason for introducing the terminology ``$\vee$-oriented words'' is as follows.
Suppose that a doubly infinite Romik sequence $T$ is admissible. 
Then Propositions~\ref{ForbiddenWords} and \ref{UnitWord} show that the digits 1 and 3 in $T$ can appear only 
in blocks of $1\, 3$ and $3\, 1$
and therefore $T$ can be rewritten as a doubly infinite word $B$ in $\{ a, b, a^{\vee}\}$,
according to the substitution rule \eqref{AlphabetRule}.
Moreover, such a word $B$ cannot contain any of the forbidden words \eqref{ForbiddenSubword}, thanks to Proposition~\ref{TwoBlock}.
We summarize this in the following proposition.
\begin{proposition}\label{PropAdmissibleAssociate}
Suppose that a doubly infinite Romik sequence $T$ is admissible. 
Then there exists a doubly infinite word $B$ on $\{ a, b, a^{\vee}\}$, which is a $\vee$-oriented associate of $T$. 
\end{proposition}
Now, we will define $\jmath(E)\in \{a, b, a^{\vee}\}^{\mathbb{N}}$
for $E \in \{ a, b \}^{\mathbb{N}}$, so that $\jmath(E)$ is $\vee$-oriented.
First, we write $E = l_1l_2\cdots$ with $l_k \in \{ a, b \}$.
For each $k = 1, 2, \dots$, we let $t = t(k)$ be the 
the number of occurrences of $b$ in the sequence $l_1, \dots, l_{k-1}$. 
Then we define $\jmath(E) = l'_1 l'_2 \cdots$ with $l_k' \in \{ a, b, a^{\vee} \}$ where 
\begin{equation}\label{EqOrientation}
l'_k = 
\begin{cases}
a & \text{ if $l_k = a$ and $t(k)$ is even},\\
a^{\vee} & \text{ if $l_k = a$ and $t(k)$ is odd},\\
b & \text{ if } l_k = b.\\
\end{cases}
\end{equation}
Simply put, $\jmath(E)$ is obtained by attaching, or not attaching, $\vee$ to each occurrence of $a$ in $E$ successively avoiding all the forbidden words in \eqref{ForbiddenSubword}.
Likewise, if $w$ is a finite word in $\{ a, b \}$, we define $\jmath(w)$ in the same way using \eqref{EqOrientation}.
Then it is not difficult to see that 
    \begin{enumerate}[font=\upshape, label=($\jmath$-\Roman*)]
    \item the (concatenated) word $a\jmath(E)$ is also $\vee$-oriented,
    \item $\jmath$ is injective, 
    \item when $w$ is a finite word in $\{ a, b \}$, 
    \[
    \jmath(wE)
    =
    \begin{cases}
    \jmath(w)\jmath(E) & \text{ if $w$ is even,}\\
    \jmath(w)\jmath(E)^{\vee} & \text{ if $w$ is odd,}\\
    \end{cases}
    \]
    \item and, we have
    \[
    \jmath(w)^*
    =
    \begin{cases}
    \jmath(w^*) & \text{ if $w$ is even,}\\
    \jmath(w^*)^{\vee} & \text{ if $w$ is odd.}\\
    \end{cases}
    \]
    \end{enumerate}
The properties ($\jmath$-I), ($\jmath$-III), and ($\jmath$-IV) follow immediately from definition.
For ($\jmath$-II), let 
\[
\iota: 
\{a, b, a^{\vee}\}^{\mathbb{N}}
\longrightarrow 
\{ a, b \}^{\mathbb{N}} 
\]
be the ``forget-the-$\vee$'' map.
Then the composition $\iota\circ\jmath$ is clearly the identity map on $\{ a, b \}^{\mathbb{N}}$
and $\jmath$ is therefore injective.

\begin{definition}[order on words]\label{DefinitionOrder}
We define orders in the two sets 
$\{a, b\}^{\mathbb{N}}$ 
and 
$\{a, b, a^{\vee}\}^{\mathbb{N}}$ 
in the following way. 
First, 
$\{a, b\}^{\mathbb{N}}$ 
is given the lexicographic order ``$\prec$'' with $a\prec b$.
For $E_1, E_2 \in \{ a, b, a^{\vee}\}^{\mathbb{N}}$, we define
\[
	E_1 \prec E_2 \ \Leftrightarrow \ \| \Pi(E_1) \| > \| \Pi(E_2)\|.  
\]
\end{definition}
\begin{proposition}\label{PropOrderPreserving}
With respect to the orders in Definition~\ref{DefinitionOrder},
the map 
\[
\jmath: 
\{a, b\}^{\mathbb{N}}
\longrightarrow
\{ a, b, a^{\vee}\}^{\mathbb{N}}
\]
is an order-preserving injection.
\end{proposition}
\begin{proof}
We already showed that $\jmath$ is an injection. 
To prove $\jmath$ is order-preserving, let $F_1, F_2\in \{ a, b \}^{\mathbb{N}}$.
Assume $F_1 \prec F_2$. 
Let $w$ be the longest (possibly empty) common prefix of $F_1$ and $F_2$ and let $t$ be the number of occurrences of $b$ in $w$. 
Since $F_1 \prec F_2$, 
we must have 
\[
	F_1 = w a \cdots \quad \text{and} \quad F_2 = w b \cdots.
\]
Suppose that $t$ is even.
From the definition of $\jmath$ (see the equation \eqref{EqOrientation} above),
we have $\jmath(F_1) = \jmath(w)a \cdots$ and $\jmath(F_2) = \jmath(w)b \cdots$.
Therefore, 
\[
\Pi(\jmath(F_1)) = [d_1, \dots, d_k, 3, 1, \dots]
\quad
\text{and}
\quad
\Pi(\jmath(F_2)) = [d_1, \dots, d_k, 2, \dots]
\]
where $\Pi(\jmath(w)) = [d_1, \dots, d_k]$ contains an even number of $2$.
Then Proposition~\ref{PropDigitComparison} implies that $\|\Pi(\jmath(F_1)) \| \ge  \| \Pi(\jmath(F_2)) \| $.
Similarly, if $t$ is odd, then $\jmath(F_1) = \jmath(w)a^{\vee} \cdots$ and $\jmath(F_2) = \jmath(w)b \cdots$. In this case,
\[
\Pi(\jmath(F_1)) = [d_1, \dots, d_k, 1, 3, \dots]
\quad
\text{and}
\quad
\Pi(\jmath(F_2)) = [d_1, \dots, d_k, 2, \dots]
\]
where $\Pi(\jmath(w)) = [d_1, \dots, d_k]$ contains an odd number of $2$.
Again, Proposition~\ref{PropDigitComparison} gives $\| \Pi(\jmath(F_1))  \| \ge \| \Pi( \jmath(F_2) ) \| $.
\end{proof}
\begin{proposition}\label{PropMonotonicity}
    Suppose that $w_1$ and $w_2$ are finite words of the same length on $\{ a, b \}$.
    Then $w_1 \preceq w_2$ (lexicographically) if and only if
    $\| \Pi( \jmath(w_1^{\infty})) \| \ge \| \Pi( \jmath(w_2^{\infty}) )\|$.
\end{proposition}
\begin{proof}
The condition $w_1 \preceq w_2$ is equivalent to $w_1^{\infty} \preceq w_2^{\infty}$.
Then the conclusion of this proposition follows 
from Proposition~\ref{PropOrderPreserving}. 
\end{proof}

\subsection{Prerequisite: Christoffel words}\label{SecChristoffel}
In this subsection, we review some definitions and basic properties of \emph{Christoffel words}. 
Our main reference for this is \cite{BLRS}.

Fix two positive coprime integers $t$ and $s$ and draw a line segment $\ell$ connecting $(0, 0)$ and $(s, t)$.
Next, draw a path $C$ from $(0,0)$ to $(s, t)$ not intersecting  $\ell$ and located below $\ell$ by joining line segments of unit length as follows. Beginning with $(0, 0)$, we successively connect adjacent points in $\mathbb{Z}\times \mathbb{Z}$ either horizontally from $(p, q)$ to $(p+1, q)$, or vertically from $(p, q)$ and $(p, q+1)$, in a way that the interior of the resulting polygon formed by $\ell$ and $C$ contains no points in $\mathbb{Z}\times \mathbb{Z}$.
Then the \emph{lower Christoffel word} \emph{of slope} $\frac ts\in \mathbb{Q}$ is the word on $\{ a, b\}$ encoding $C$, where $a$ represents a horizontal segment in $C$ and $b$ represents a vertical segment in $C$. If we draw $C$ similarly but lying \emph{above} $\ell$, then the word representing $C$ is called the \emph{upper Christoffel word} \emph{of slope} $\frac ts$.
See Figure~\ref{ChristoffelPicture} for an example.
\begin{figure}
\begin{center}
\begin{tikzpicture}
	\draw[->] (-0.2, 0) -- (7.5, 0) node[right]{$a$} ;
	\draw[->] (0, -0.2) -- (0, 4.5) node[above]{$b$} ;

	\draw[thin, dotted] (0, 0) grid (7, 4);
	\draw[thick, dashed] (0, 0) -- (7, 4) node[above right] {$(7, 4)$};
	\draw[ultra thick] 
	(0, 0) -- (1, 0) node[midway, below]{$a$}
	-- ++ (1, 0)  node[midway, below]{$a$}
	-- ++ (0, 1)  node[midway, left]{$b$}
	-- ++ (1, 0)  node[midway, below]{$a$}
	-- ++ (1, 0)  node[midway, below]{$a$}
	-- ++ (0, 1)  node[midway, left]{$b$}
	-- ++ (1, 0)  node[midway, below]{$a$}
	-- ++ (1, 0)  node[midway, below]{$a$}
	-- ++ (0, 1)  node[midway, left]{$b$}
	-- ++ (1, 0)  node[midway, below]{$a$}
	-- ++ (0, 1)  node[midway, left]{$b$}
	;
	\draw[ultra thick]
	(0, 0) -- (0, 1) node[midway, right]{$b$}
	-- ++ (1, 0)  node[midway, above]{$a$}
	-- ++ (0, 1)  node[midway, right]{$b$}
	-- ++ (1, 0)  node[midway, above]{$a$}
	-- ++ (1, 0)  node[midway, above]{$a$}
	-- ++ (0, 1)  node[midway, right]{$b$}
	-- ++ (1, 0)  node[midway, above]{$a$}
	-- ++ (1, 0)  node[midway, above]{$a$}
	-- ++ (0, 1)  node[midway, right]{$b$}
	-- ++ (1, 0)  node[midway, above]{$a$}
	-- ++ (1, 0)  node[midway, above]{$a$}
	;
\end{tikzpicture}
	\[ 
	w_{\mathrm{lower}} = aabaabaabab, \text{ and } 
	w_{\mathrm{upper}} = babaabaabaa. 
	\]
\end{center}
	\caption{Lower and upper Christoffel words of slope $\frac 47$. This picture is from Figure 1.2 in \cite{BLRS}.}
   \label{ChristoffelPicture}
\end{figure}
Also, we refer readers to Chapter 1 of \cite{BLRS} for other equivalent definitions of Christoffel words. Also, see Chapters 7 and 8 in \cite{Aig13}. Note that Aigner's lower Christoffel word $\mathrm{ch}_{\frac pq}$ is the same as the lower Christoffel word with slope $\frac{p}{q-p}$ in \cite{BLRS}.

The words $w = a$ and $w = b$ are regarded as \emph{trivial} Christoffel words (of slope 0 and $\infty$, respectively). Any Christoffel word of lengh 2 or greater is called \emph{nontrivial}.

As $t$ is the number of occurrence of $b$ in a Christoffel word $w$ of slope $t/s$, we see that the parity of $w$ is the same as that of $t$.

The properties of Christoffel words we will need later are summarized in the following proposition. For proof, see Proposition 4.2 in \cite{BLRS} and  Proposition 7.27 and Remark 7.28 in \cite{Aig13}.

\begin{proposition}\label{PropertiesChristoffel}
Let $w$ be a nontrivial lower Christoffel word in $\{ a, b \}$.
\begin{enumerate}[font=\upshape, label=(\alph*)]
    \item We have $w=aub$ for a palindrome $u$,
    namely, $u=u^*$.
	\item The reverse $w^*$ of $w$ is the upper Christoffel word of the same slope.
	\item Let $w'$ be any conjugate of $w$ (that is, $w' = w_2w_1$ for some $w_1$ and $w_2$ such that $w=w_1w_2$) and let $w'^*$ be its reverse.
		Then,
		\[
			w \preceq w'
			\quad
			\text{and}
			\quad
			w \preceq w'^*.
		\]
		Furthermore,
		\[
			w'
			\preceq w^*
			\quad
			\text{and}
			\quad
			 w'^*
			\preceq w^*.
		\]
\end{enumerate}
\end{proposition}

\subsection{Doubly infinite words having Christoffel words as periods}\label{SecPeriodicWords}
For a finite word $w$ in $\{ a, b \}$, we define $E(w)$ to be 
\begin{equation}\label{EqDefinitionEw}
    E(w) = \jmath(w^{\infty}) \in \{ a, b, a^{\vee}\}^{\mathbb{N}}.
\end{equation}
Further, we define $B(w)$ to be the doubly infinite word in $\{ a, b, a^{\vee} \}$, one of whose sections is
\begin{equation}\label{EqMainSection}
E(w^*)^* | E(w).
\end{equation}
It is easy to see that,
if every $a^{\vee}$ in $B(w)$ is replaced with $a$,
then the resulting doubly infinite word (in $\{ a, b \}$) would be
\[
\cdots www\cdots.
\]
For example, let $w = abbb$. Then
\begin{align*}
E(w^*) &= \jmath((w^*)^{\infty}) = 
bbba^{\vee}bbb a \cdots ,\\
E(w) &= \jmath(w^{\infty}) = 
abbb a^{\vee} bbb \cdots,
\end{align*}
so that
\[
E(w^*)^*|E(w)=
\cdots 
abbb a^{\vee} bbb 
|
abbb a^{\vee} bbb \cdots 
\\
\]
and
\[
B(w) =
\cdots 
abbb a^{\vee} bbb 
abbb a^{\vee} bbb \cdots. 
\]

\begin{theorem}\label{ThmLagrangeNumberChristoffelWord}
Let $w$ be a nontrivial lower Christoffel word in $\{ a, b \}$.
Then
\[
L(B(w)) = 
L(E(w^*)^*|E(w)).
\]
\end{theorem}
The rest of this subsection is devoted to proving this theorem.
We will need to show that 
$
L(E(w^*)^*|E(w))
$
is greater than or equal to $L(E^*|F)$ for any section $E^*|F$ of $B(w)$ or $B(w)^{\vee}$ (cf.~\eqref{DefinitionLB}).
First, we investigate all possible sections of $B(w)$ and $B(w)^{\vee}$.
In particular, we claim that, for any finite word $w$ in $\{a,b \}$ (not necessarily a Christoffel word for now), 
every section of $B(w)$ or $B(w)^{\vee}$ is of the form
\begin{enumerate}[font=\upshape, label=(\roman*)]
   \item $E(w'^*)^* | E(w')$, or
   \item $(E(w'^*)^{\vee})^* | E(w')^{\vee}$,
\end{enumerate}
where $w'$ is a conjugate of $w$, that is, $w' = w_2w_1$ for some $w_1$ and $w_2$ such that $w = w_1w_2$.
To prove this claim, let us first assume that
$w_1$ is even.
Then we see 
from the property ($\jmath$-III) in \S\ref{SecOrientation} that
\[
E(w) = \jmath(w^{\infty})
=\jmath(w_1w_2w_1w_2 \cdots)
=\jmath(w_1)\jmath(w_2w_1w_2 \cdots).
\]
Therefore, if we delete the prefix $\jmath(w_1)$ from $E(w)$, 
we obtain
\[
E(w') 
= \jmath(w'^{\infty}))
=\jmath(w_2w_1w_2w_1 \cdots).
\]
Likewise, 
\[
E(w'^*)
= \jmath(
w_1^*w_2^*w_1^*w_2^* \cdots
)
\]
can be obtained from
\[
E(w^*)
= \jmath(
w_2^*w_1^*w_2^*w_1^* \cdots
)
\]
by concatenating it with $\jmath(w_1)^* = \jmath(w_1^*)$
(see ($\jmath$-IV) in \S\ref{SecOrientation}).
Combining them together, we see that the section
$
E(w'^*)^*|E(w')
$
is obtained from 
$
E(w^*)^*|E(w)
$
by shifting it forward by the length of $w_1$.
On the other hand, if $w_1$ is odd, we see from ($\jmath$-III) that
\[
E(w) = \jmath(w^{\infty}) 
=\jmath(w_1w_2w_1w_2 \cdots)
=\jmath(w_1)\jmath(w_2w_1w_2 \cdots)^{\vee}.
\]
So, by deleting the prefix $\jmath(w_1)$ from $E(w)$, we obtain $E(w')^{\vee}$.
Also, if we concatenate $\jmath(w_1)^* = \jmath(w_1^*)^{\vee}$ with 
\[
E(w^*)
= \jmath(
w_2^*w_1^*w_2^*w_1^* \cdots
),
\]
then we obtain
\[
\jmath(w_1^*)^{\vee}
\jmath(
w_2^*w_1^*w_2^*w_1^* \cdots
)
=
\jmath(w_1^*w_2^*w_1^*w_2^*w_1^* \cdots
)^{\vee}
=
(E(w'^*))^{\vee}.
\]
This shows that, by shifting 
$
E(w^*)^*|E(w)
$
forward by the length of $w_1$, we obtain
$(E(w'^*)^{\vee})^* | E(w')^{\vee}$.
\begin{table}
    \caption{Sections of $B(w)$ and $B(w)^{\vee}$ with $w = abbb$.}
    \centering
    \begin{tabular}{cr@{$\mid$}lr@{$\mid$}l}
    \toprule
       $w'$ & 
       $E(w'^*)^*$&$E(w')$ &
       $(E(w'^*)^{\vee})^*$&$E(w')^{\vee}$ 
       \\
      \midrule
$abbb$ &  $\cdots abbb a^{\vee} bbb$  & $abbb a^{\vee} bbb \cdots $ 
&
$\cdots a^{\vee} bbb a bbb$  &  $a^{\vee} bbb a bbb \cdots $ 
\\
$bbba$ &  
$\cdots bbb a^{\vee} bbba$  & $bbb a^{\vee} bbba \cdots $ 
&
$\cdots bbb a bbba^{\vee} $  & $bbb abbba^{\vee}  \cdots $ 
\\
$bbab$ &  
$\cdots bb abbba^{\vee} b$  & $bb a bbba^{\vee} b \cdots $
&
$\cdots bb a^{\vee} bbbab$  & $bb a^{\vee} bbbab \cdots $
\\
$babb$ &  
$\cdots b a^{\vee} bbbabb$  & $b a^{\vee}  bbba bb \cdots $
&
$\cdots b abbba^{\vee} bb$  & $b a bbba^{\vee} bb \cdots $
\\
    \bottomrule
    \end{tabular}
    \label{TabSections}
\end{table}
This concludes proving the claim that every section of $B(w)$ and $B(w)^{\vee}$ is of the form (i) and (ii) above.
An example of this with $w = abbb$ is shown in Table~\ref{TabSections}.

We are now ready to prove that
\begin{equation}\label{EqLoptimized}
L(E(w^*)^*|E(w)) \ge L(E^*| F)
\end{equation}
for any section $E^*|F$ of $B(w)$ or of $B(w)^{\vee}$.
Suppose that $E^*|F$ is of the form (i) in the claim, that is,
\[
   E^*|F = E(w'^*)^* | E(w')
\]
for a conjugate $w'$ of a nontrivial lower Christoffel word $w$.
Then Proposition~\ref{PropertiesChristoffel} gives
the inequalities
\[ 
w \preceq w'
\quad
\text{and}
\quad
w'^* \preceq w^*.
\]
So we can apply Proposition~\ref{PropMonotonicity} to obtain
\[
\| \Pi(E(w')) \| \le
\| \Pi(E(w)) \|
\quad
\text{and}
\quad
\| \Pi(E(w^*)) \|
\le
\| \Pi(E(w'^*)) \|
.
\]
We use these inequalities together with Proposition~\ref{prop:check_order_reversing} to obtain
\begin{align*}
L(E(w'^*)^* | E(w')) &=
\frac{
\| \Pi(E(w'^*))^{\vee} \|
+
\| \Pi(E(w')) \|  
}{
\sqrt2
} 
\\
& \le
\frac{
\| \Pi(E(w^*))^{\vee} \|
+
\| \Pi(E(w)) \| 
}{
\sqrt2
} \\
&=
L(E(w^*)^*|E(w)),
\end{align*}
which proves \eqref{EqLoptimized} when $E^*|F$ is of the form (i).
Next, assume that the section $E^*|F  =
(E(w'^*)^{\vee})^* | E(w')^{\vee}
$, that is, it is of the form (ii) in the claim.
In this case, we use the inequalities
\[ 
w \preceq w'^*
\quad
\text{and}
\quad
w' \preceq w^*
\]
to obtain
\[
\| \Pi(E(w)) \| \ge  \| \Pi(E(w'^*)) \| 
\quad
\text{and}
\quad
\| \Pi(E(w')) \|
\ge
\| \Pi(E(w^*)) \|.
\]
Likewise, we use them together with Proposition~\ref{prop:check_order_reversing} to get
\begin{align*}
L((E(w'^*)^{\vee})^* | E(w')^{\vee})
&=
\frac{
\| \Pi(E(w'^*))\|
+
\| \Pi(E(w'))^{\vee}  \|  
}{
\sqrt2
} 
\\
& \le
\frac{
\| \Pi(E(w)) \| 
+
\| \Pi(E(w^*))^{\vee} \|
}{
\sqrt2
} \\
&=
L(E(w^*)^*|E(w)).
\end{align*}
This establishes \eqref{EqLoptimized} for all sections $E^*|F$ of $B(w)$ and of $B(w)^{\vee}$, finishing the proof of
Theorem~\ref{ThmLagrangeNumberChristoffelWord}.

\subsection{Characterization of doubly infinite admissible words}
\begin{proposition}\label{AdmissibleInequality}
	Assume that $B$ is $\vee$-oriented.
	Then $B$ is admissible if and only if every section $E^*b|aF$ of 
	$B$, $B^*$, $B^{\vee}$ or $(B^*)^{\vee}$ 
	satisfies $E^{\vee} \preceq F$.
	Moreover, $L(E^*b|aF) = 2$ if and only if $E^{\vee} = F$.
\end{proposition}

\begin{proof}
	Let $E^*b|aF$ be a section of 
	$B$, $B^*$, $B^{\vee}$ or $(B^*)^{\vee}$.
	Then, we compute $L(E^*b|aF)$ using \eqref{EqLwordSection} to conclude 
	from Proposition~\ref{AdmissibleComparison}
	that
	\begin{equation*}
		L(E^*b|aF) 
		= 
		\frac{\| 2 \, \Pi(E^{\vee}) \| + \| 3 1\,  \Pi(F) \|  }{\sqrt2}
		\le 2
	\end{equation*}
	if and only if $\| \Pi(E^{\vee}) \| \ge \| \Pi(F)\|$.
	The last inequality is equivalent to $E^{\vee} \preceq F$ by Definition~\ref{DefinitionOrder}.
	Hence, we obtain both assertions in this proposition from this
	and Proposition~\ref{Axiom}.
\end{proof}

Later in Theorem~\ref{BombieriThm15}, we will characterize all doubly infinite words on $\{ a, b, a^{\vee} \}$ that are admissible.
First, we review Bombieri's results, which do this for doubly infinite words on the \emph{two letters alphabet} $\{ a, b \}$.
This will constitute a core ingredient in our proof.  

Let $B$ be a doubly infinite word on $\{ a, b \}$.
As in \cite{Bom07}, define $\lambda(B)$ to be the supremum of the length $l(w)$ of words $w$ such that $w^*b|aw$ occurs as a subword of $B$.
Then we can rephrase Bombieri's Lemma 11 and Theorem 15 in \cite{Bom07} as follows.

\begin{theorem}[Lemma 11 and Theorem 15 in \cite{Bom07}]\label{ThmBombieriOriginal15}
Suppose that a doubly infinite word $B$ in $\{a, b\}$ satisfies the following condition:
every section $E^*b|aF$ of either $B$ or $B^*$ satisfies $E\preceq F$ (cf.~Definition~\ref{DefinitionOrder}). 
Then either $\mathrm{(i)}$ or $\mathrm{(ii)}$ is true:
	\begin{enumerate}[font=\upshape, label=(\roman*)]
	\item $\lambda(B) = \infty$, or
	\item $\lambda(B) < \infty$ and
	    $B$ is periodic with a Christoffel word $w$.
	\end{enumerate}
\end{theorem}

Now, we proceed to characterizing doubly infinite admissible words in the \emph{three letters alphabet} $\{a, b, a^{\vee} \}$.
Let $B$ be a doubly infinite word on $\{ a, b, a^{\vee} \}$
and, as before, we define $\lambda(B)$ to be the supremum of the length $l(w)$ of 
$w$ such that $(w^{\vee})^*b|a\,w$ occurs as a subword of $B$.
\begin{theorem}\label{BombieriThm15}
	If $B$ is admissible and $\lambda(B) = \infty$, then $L(B) = 2$. 
	If $B$ is admissible and $\lambda(B) < \infty$, then either $B$ or $B^{\vee}$ is equal to  
$B(w)$ for a Christoffel word $w$.
\end{theorem}

\begin{proof}
Suppose that there is a sequence of sections $E_j^*(w_j^{\vee})^*b|aw_jF_j$ of $B$ with $l(w_j) \to\infty$. Notice from Proposition~\ref{PropAdmissibleAssociate} that $B$ is $\vee$-oriented.
Therefore, the infinite words $w_jF_j$ and $w_jE_j^{\vee}$ are also $\vee$-oriented for all $j\ge0$.
Apply Proposition~\ref{PropOrderPreserving} and use the compactness of $\{a, b\}^{\mathbb{N}}$ to conclude that there exists a subsequence $\{j_k\}$ along which both $w_{j_k}F_{j_k}$ and $w_{j_k}E_{j_k}^{\vee}$ converge to a common infinite word $W$, which is also $\vee$-oriented. 
(For instance, we first choose a convergent subsequence $\{w_{j_n}F_{j_n}\}$ of $\{w_jF_j\}$, then choose a subsequence of $\{j_n\}$ along which $w_jE_j^{\vee}$ converges.)
Then, 
\[
	L(B) \ge L( E_j^*(w_j^{\vee})^*b|aw_jF_j) \to L((W^{\vee})^*b|aW) = 2
\]
by Proposition~\ref{AdmissibleInequality}.

Now, assume that $\lambda(B)<\infty$ and that $B$ is not equal to any of the following words:
\[
\cdots aaa \cdots, \quad
\cdots a^{\vee}a^{\vee}a^{\vee} \cdots, \quad
\cdots bbb \cdots.
\]
Let $C$ be the doubly infinite word on $\{a, b\}$ obtained by ``forgetting $\vee$'s'' from $B$.
Suppose that $E_1^*b|aE_2$ is a section of either $C$ or $C^*$.
Define $F_1 = \jmath(E_1)$ and $F_2= \jmath(E_2)$.
Since both $F_1$ and $F_2$ are $\vee$-oriented, a concatenation $(F_1^{\vee})^* b|a F_2$ gives a section of a $\vee$-oriented doubly infinite word.
Moreover, if we drop all $\vee$'s from $(F_1^{\vee})^* b|a F_2$, 
then we must get $E_1^*b|aE_2$ back.
This shows that $(F_1^{\vee})^* b|a F_2$	is a section of $B$, $B^{\vee}$, $B^*$ or $(B^*)^{\vee}$.
Then we apply Proposition~\ref{AdmissibleInequality} and conclude that $F_1\preceq F_2$, which in turn gives $E_1 \preceq E_2$ from Proposition~\ref{PropOrderPreserving}.

With this, we invoke Bombieri's result above (Theorem~\ref{ThmBombieriOriginal15}), according to which $C$ is periodic with period given by a lower Christoffel word $w$ with length $l(w) = \lambda(C)$.
If we write $ C_1= (w^*)^{\infty}$ and $C_2 = w^{\infty}$,
then $C_1^*| C_2$ is a section of $C$. 
Utilizing notations from \S\ref{SecPeriodicWords}, we have 
$\jmath(C_1) = E(w^*)$
and 
$\jmath(C_2) = E(w)$,
which shows that either $B$ or $B^{\vee}$ is equal to $B(w)$.
\end{proof}

\section{Proof of Main Theorem}\label{SecProofOfMainTheorem}
\subsection{Cohn matrices}\label{SecCohn}
\begin{definition}\label{DefCohnMatrix}
Define $A,  A^{\vee}, B \in\mathrm{SL}_2(\mathbb{R})$ to be
(cf.~\eqref{EqDefNd} and \eqref{EqDefJ})
\begin{equation*}
A = 
N_3N_1 =
\begin{pmatrix} 
3 & \sqrt2 \\ 
\sqrt2 & 1 \end{pmatrix}, 
\qquad
A^\vee = 
N_1N_3 =
\begin{pmatrix} 
1 & \sqrt2 \\ 
\sqrt2 & 3 \end{pmatrix},
\end{equation*}
\begin{equation*}
B = N_2 J = J N_2 = \begin{pmatrix} 
\sqrt2 & 1 \\ 
1 & \sqrt2 
\end{pmatrix}.
\end{equation*}
For a finite word $w$ on $\{a, b \}$, we define \emph{a Cohn matrix} $N(w)$ associated to $w$ to be the matrix obtained by substituting $a$ with $A$ and $b$ with $B$ and then performing matrix multiplication.
In other words, if $w = a^{e_1} b^{f_1} \cdots a^{e_k} b^{f_k}$ with nonnegative integers $\{e_j, f_j\}$,  $j = 1 , \dots , k$, then
\[
N(w) = A^{e_1} B^{f_1}\cdots A^{e_k}B^{f_k}.
\]
\end{definition}
\begin{proposition}\label{Prop:Mw}
Let $w$ be a finite word on $\{a, b \}$ and let
\[
N(w)
=
\begin{pmatrix} 
p & p' \\ 
q & q' \end{pmatrix}
\]
be its Cohn matrix.
Also, let $[d_1, \dots, d_k]$ be the digit sequence such that
\[
\Pi(\jmath(w)) = [d_1, \dots,d_k].
\]
If $w$ is even, then
\[
N_{d_1} \cdots N_{d_k}
=
\begin{pmatrix} 
p & p' \\ 
q & q' \end{pmatrix}
=
N(w).
\]
If $w$ is odd, then
\[
N_{d_1} \cdots N_{d_k}
=
\begin{pmatrix} 
p' & p \\ 
q' & q \end{pmatrix}
=N(w)J. 
\]
\end{proposition}
\begin{proof}
Direct calculations and the induction show that for any integer $e$ that
\[
A^e = J(A^{\vee})^eJ,
\quad
\text{and}
\quad
(A^{\vee})^e = JA^eJ.
\]
Also, if $e$ is odd, then
\[
B^e = JN_2^e = N_2^e J
\]
and, for an even $e$, 
\[
B^e = N_2^e.
\]
Note that $w$ is written uniquely as
\[
w = a^{e_1}b^{e_2}a^{e_3}\cdots  b^{e_{2n}}a^{e_{2n+1}}
\]
where $e_1,  e_{2n+1} \ge 0$ and $e_2, \dots,  e_{2n}>0$.
If $e_2$ is even, we obtain
\[
N(w) = A^{e_1}N_2^{e_2}A^{e_3}\cdot B^{e_4} \cdots
= (N_3N_1)^{e_1}N_2^{e_2}(N_3N_1)^{e_3} \cdot B^{e_4} \cdots,
\]
and if $e_2$ is odd, then
\[
N(w) = A^{e_1}N_2^{e_2}(A^{\vee})^{e_3} J \cdot B^{e_4}  \cdots
= (N_3N_1)^{e_1}N_2^{e_2}(N_1N_3)^{e_3} J \cdot B^{e_4}  \cdots.
\]
Continuing this way, we see that
\begin{equation}\label{EqNwToNdk}
N(w) = 
\begin{cases}
N_{d_1} \cdots N_{d_k} & \text{ if  $e_2 + e_4 + \dots +  e_{2n}$ is even},\\
N_{d_1} \cdots N_{d_k} J & \text{ if  $e_2 + e_4 + \dots +  e_{2n}$ is odd}.\\
\end{cases}
\end{equation}
Since $w$ is even if and only if $e_2 + e_4 + \dots + e_{2n} $ is even, \eqref{EqNwToNdk} completes the proof of the proposition.
\end{proof}
\begin{corollary}\label{CorCohnForm}
For a finite word $w$ in $\{a, b \}$, 
let
\[
N(w)
=
\begin{pmatrix} 
p & p' \\ 
q & q' \end{pmatrix}.
\]
be the Cohn matrix of $w$. 
Then $p, q, p', q'> 0$.
Furthermore, if $w$ is even, then $p, q'\in \mathbb{Z}$ and $p', q \in \sqrt2\mathbb{Z}$.
If $w$ is odd, then $p, q'\in \sqrt2\mathbb{Z}$ and $p', q \in \mathbb{Z}$.
\end{corollary}
\begin{proof}
It will be sufficient to show that $N_{d_1} \cdots N_{d_k}$ is of the form
\begin{equation}\label{EqDiaAntiDia}
D + \sqrt2 K
\end{equation}
for a diagonal matrix $D$ and an anti-diagonal matrix $K$ with nonnegative integer coefficients. 
It is not difficult to show that a product of any two matrices of the form \eqref{EqDiaAntiDia} is again of the form \eqref{EqDiaAntiDia}.
Since $N_1$, $N_2$, $N_3$ are of the form \eqref{EqDiaAntiDia}, we prove the claim.

If $w$ is even, Proposition~\ref{Prop:Mw} implies the corollary directly. 
When $w$ is odd, we use Proposition~\ref{Prop:Mw} again and make the same observation above, together with an easy fact $(D + \sqrt2K)J = K' + \sqrt2D'$ for some diagonal $D'$ and anti-diagonal $K'$. 
\end{proof}
\begin{proposition}[Lemme 3.2 in \cite{Reu06}]\label{Lemma32Reu}
Let $W$ be any $2\times 2$ symmetric matrix
and let $q$ be the lower left entry of the $2\times 2$ matrix $AWB$.
Then
\[
q = \frac{\Tr(AWB)}{2\sqrt2}.
\]
In particular, if $w$ is a lower Christoffel word in $\{ a, b\}$ and if $q$ is the lower left entry of the Cohn matrix $N(w)$, then
\[
q =
\frac{\Tr(N(w))}{2\sqrt2}.
\]
\end{proposition}
\begin{proof}
Letting $W = \left( 
\begin{smallmatrix}
p & q \\
q & r \\
\end{smallmatrix} 
\right)$, we get
\begin{align*}
AWB &= 
\begin{pmatrix}
3 & \sqrt2 \\
\sqrt2 & 1 
\end{pmatrix}
\begin{pmatrix}
p & q \\
q & r \\
\end{pmatrix}
\begin{pmatrix}
\sqrt2 & 1 \\
1 & \sqrt2 \\
\end{pmatrix}
\\
&= 
\begin{pmatrix}
3\sqrt2 p + 5q + \sqrt2 r & 3p + 4\sqrt2 q + 2r \\
2p + 2\sqrt2 q + r & \sqrt2 p + 3q + \sqrt2 r
\end{pmatrix},
\end{align*}
so that 
\[
\Tr(AWB) = 4\sqrt2 p + 8 q + 2\sqrt2 r= 2\sqrt2(2p + 2\sqrt2 q + r).
\]
For the last assertion, we need to show that
$N(w)$ is of the form $AWB$ for some symmetric $W$.
This easily follows from the fact that any (nontrivial) lower Christoffel word $w$ is of the form $w = aub$ for some palindrome $u$ and the observation that the matrices $A$ and $B$ themselves are symmetric.
\end{proof}
\begin{theorem}\label{ChristoffelL}
Let $w$ be a lower Christroffel word and let $q(w)$ be the lower left entry of the Cohn matrix $N(w)$. 
Also, let $B(w)$ be the doubly infinite word in $\{ a, b, a^{\vee}\}$ defined in \S\ref{SecPeriodicWords}.
Then 
\[
L( B(w) ) = \sqrt{4 - \frac2{q(w)^2}}.
\]
\end{theorem}
\begin{proof}
As before, we write the Cohn matrix $N(w)$ as
\[
N(w) = 
\begin{pmatrix}
p & p' \\
q & q'
\end{pmatrix}.
\]

When $w$ is even, 
it is easy to see that
$\Pi(E(w))$ 
(recall the notation from \S\ref{SecPeriodicWords})
is a purely periodic sequence with period $\Pi: = \Pi(\jmath(w))$, so that
\[
\Pi(E(w)) = \Pi^{\infty}.
\]
Likewise,
$\Pi(E(w^*))^{\vee}$
is purely periodic with period $\Pi^*$ (the reverse of $\Pi)$ and 
\[
\Pi(E(w^*))^{\vee} = (\Pi^*)^{\infty}.
\]
Therefore, from Theorem~\ref{ThmLagrangeNumberChristoffelWord}, we have
\[
L(B(w)) =
L({}^{\infty}\Pi |\Pi^{\infty}).
\]
To compute this expression, 
we use
Propositions~\ref{PropLpurelyPeriodic},
\ref{Prop:Mw},
and
\ref{Lemma32Reu},
together with the fact $\det(N(w)) = 1$ to obtain
\[
L(B(w))
=
\sqrt{\frac{8q^2 - 4}{2q^2}}
=
\sqrt{4 -\frac2{q^2}}.
\]
This proves the theorem in the even case.

Suppose that $w$ is odd.
Similar calculations show that
$
\Pi(E(w))
$
and
$
\Pi(E(w^*))^{\vee}
$
are purely periodic with periods $\Pi 
=[d_1, \dots, d_k] 
:= \jmath(w)\jmath(w)^{\vee}$
and $\Pi^*$.
So, we will need to compute  
$N_{d_1}\cdots N_{d_k}N_{d_1^{\vee}} \cdots N_{d_k^{\vee}}$.
Apply Propositions~\ref{Prop:Mw} and \ref{PropNdCheck} to get
\[
N_{d_1} \cdots N_{d_k}
=
\begin{pmatrix}
p' & p \\
q' & q
\end{pmatrix}
\quad
\text{and}
\quad
N_{d_1^{\vee}} \cdots N_{d_k^{\vee}}
=
\begin{pmatrix}
q & q' \\
p & p'
\end{pmatrix}.
\]
So,
\[
N_{d_1}\cdots N_{d_k}N_{d_1^{\vee}} \cdots N_{d_k^{\vee}}
=
\begin{pmatrix}
p' & p \\
q' & q
\end{pmatrix}
\begin{pmatrix}
q & q' \\
p & p'
\end{pmatrix}
=
\begin{pmatrix}
p'q + p^2 & p'(p + q') \\
q(p + q') & p'q + (q')^2 
\end{pmatrix}.
\]
Now we use Proposition~\ref{PropLpurelyPeriodic} and Theorem~\ref{ThmLagrangeNumberChristoffelWord} as before to obtain
\[
L(B(w)) =
\sqrt{
\frac{
(2p'q + p^2 + (q')^2)^2 - 4
}{
2q^2(p + q')^2
}}.
\]
We simplify this expression using $\det(N(w)) = pq' - p'q = 1$ and $p + q' = 2\sqrt2 q$ (cf.~Proposition~\ref{Lemma32Reu}):
\begin{align*}
L(B(w)) &=
\sqrt{
\frac{
(2pq' + p^2 + (q')^2 - 2)^2 - 4
}{
2q^2(p + q')^2
}}
\\
&=
\sqrt{
\frac{
(p + q')^4 - 4(p + q')^2 
}{
2q^2(p + q')^2
}}
=
\sqrt{
\frac{
(p + q')^2 - 4
}{
2q^2
}} 
=
\sqrt{4 -\frac2{q^2}}.
\end{align*}
This completes the proof of the theorem in both cases.
\end{proof}

\subsection{The Christoffel tree and the Markoff tree}\label{SubSectionMarkoffTriples}
In this subsection, we review some known results on \emph{the Christoffel tree} and \emph{the Markoff tree}. 
The upshot of this discussion will be Theorem~\ref{ThmTheorem1Reutenauer}, which says that these two have the same tree structure.
This fact will be a key ingredient in our proof of Theorem~\ref{MainTheoremDescDiscretePart}.

Our discussion on the two trees in this subsection should not be regarded as original, as these results are either directly cited from existing literature or they are slight modifications of corresponding statements in the classical case. 
See \cite{Aig13} for a much more general exposition on several isomorphic trees defined in a variety of ways.

We review the construction of the Christoffel tree.
Let $w$ be a nontrivial Christoffel word on $\{a, b\}$.
Then a theorem of Borel and Laubie (see \cite{BL93} and Theorem 3.3 in \cite{BLRS}) says that $w$ admits a factorization $w = uv$ (by which we mean that the concatenation of $u$ and $v$ is equal to $w$) with $u$ and $v$ being Christoffel words themselves and that such a factorization is unique.
The factorization $w = uv$ is called the \emph{standard factorization of} $w$.
By a \emph{Christoffel pair}, we mean an ordered pair $(u, v)$  which gives the standard factorization of a Christoffel word $w$.
We define a \emph{branching rule for the Christoffel tree} as follows. 
For a given Christoffel pair $(u, v)$, we add $(u, uv)$ and $(uv, v)$, and connect them with $(u, v)$.
\begin{center}
\begin{tikzcd}
  & (u, v) \arrow[-, rd] \arrow[-, ld] &  \\
  (u, uv) && (uv, v)\\ 
\end{tikzcd}
\end{center}
The newly added pairs $(u, uv)$ and $(uv, v)$ are also known to be Christoffel pairs. 
Furthermore, if we construct a tree by beginning with $(a, b)$ and applying the branching rule recursively,
we obtain an infinite complete binary tree containing every Christoffel pair exactly once.
This tree is called \emph{the Christoffel tree} (see Figure~\ref{Ctree}).
All these facts are explained in Chapter 3 of \cite{BLRS}. 
\begin{figure}
\begin{center}
\tikzset{triarrow/.pic={
    \draw (-0.1, 0) -- (-0.2, -0.5) node[below]{$\vdots$};
    \draw (0.1, 0) -- (0.2, -0.5) node[below]{$\vdots$};
  }
}
\begin{tikzpicture}[xscale=0.9]%[->, >=stealth', auto, xscale=1.2]

  \node (0) at (0, 4) {$(a, b)$};

  \node (f) at (-2.8, 2.8) {$(a, ab)$};
  \node (g) at (2.8, 2.8) {$(ab, b)$};

\node (f0) at (-4.3, 1.5) {$(a, a^2b)$};
\node (f1) at (-1.5, 1.5)  {$(a^2b, ab)$};

\node (g0) at (1.4, 1.5) {$(ab, ab^2)$};
\node (g1) at (4.5, 1.5) {$(ab^2, b)$};

\node (f01) at (-5.6, -.3) {$(a, a^3b)$};
\pic at (-5.6, -0.5) {triarrow};
\node (f00) at (-3.8, .3) {$(a^3b, a^2b)$};
\pic at (-4, 0.1) {triarrow};

\node (f10) at (-2.4, -.3)  {$(a^2b, a^2bab)$};
\pic at (-2.4, -0.5) {triarrow};
\node (f11) at (-.8, .3)  {$(a^2bab, ab)$};
\pic at (-0.8, 0.1) {triarrow};

\node (g00) at (.8, -.3) {$(ab, abab^2)$};
\pic at (0.8, -0.5) {triarrow};
\node (g01) at (2.4, .3) {$(abab^2, ab^2)$};
\pic at (2.4, 0.1) {triarrow};

\node (g10) at (3.8, -.3) {$(ab^2, ab^3)$};
\pic at (4, -0.5) {triarrow};
\node (g11) at (5.2, .3) {$(ab^3, b)$};
\pic at (5.2, 0.1) {triarrow};

\draw (0) to node[font=\footnotesize]{} (f);
\draw (0) to node[below right, font=\footnotesize]{} (g);

\draw (f) to node[font=\footnotesize]{} (f0); 
\draw (f) to node[right, font=\footnotesize]{} (f1); 

\draw (g) to node[font=\footnotesize]{} (g0); 
\draw (g) to node[right, font=\footnotesize]{} (g1); 

\draw (f0) to node[font=\footnotesize]{} (f00); 
\draw (f0) to node[right, font=\footnotesize]{} (f01); 

\draw (f1) to node[font=\footnotesize]{} (f10); 
\draw (f1) to node[right, font=\footnotesize]{} (f11); 

\draw (g0) to node[font=\footnotesize]{} (g00); 
\draw (g0) to node[right, font=\footnotesize]{} (g01); 

\draw (g1) to node[font=\footnotesize]{} (g10); 
\draw (g1) to node[right, font=\footnotesize]{} (g11); 
\end{tikzpicture}
\end{center}\caption{The Christoffel tree based on the alphabet $\{ a, b \}$. \label{Ctree}}
\end{figure}

Now we move on to the Markoff tree.
Let us call $(x; y_1, y_2)$ 
\emph{a Markoff triple} if they are positive integers satisfying 
$2x^2 + y_1^2 + y_2^2 = 4xy_1y_2$.
The notation $(x; y_1, y_2)$ is meant to distinguish $x$ from $y_1$ and $y_2$, while $(y_1, y_2)$ is considered to be an unordered pair. 
For example, $(x; y_1, y_2) = (x'; y_1', y_2')$ shall mean $x = x'$ and, as unordered pairs, $(y_1, y_2) = (y_1', y_2')$.
The triple $(1;1, 1)$ will be called the \emph{trivial} Markoff triple and all others will be called \emph{nontrivial} ones. 

\begin{proposition}[Chapter 2 of \cite{CF89}]
  Suppose that $(x; y_1, y_2)$ is a Markoff triple with $(x; y_1, y_2) \neq (1; 3, 1), (1; 1, 1)$.
  Then $x, y_1, y_2$ are distinct.
  \label{prop:distinctMarkoff}
\end{proposition}
\begin{proof}
  If $y_1 = y_2$, then $2x^2 + 2y_1^2 = 4xy_1^2$, or $x^2 = (2x-1)y_1^2$. 
  So $y_1^2$ divides $x^2$ and $x = ky_1$ for some $k$.
  Cancelling out $y_1^2$, we have $k^2 = 2ky_1 - 1$.
  Being an integral solution of this quadratic equation, $k$ must be 1. 
  This results in $(x; y_1, y_2) = (1; 1, 1)$.
  
  If $x = y_2$,  a similar calculation gives $y_1^2 = x^2(4y_1 - 3)$.
  So $y_1 = kx$ for some $k$, which gives $k^2 = 4kx - 3$. 
  Likewise, the possible values of $k$ are $k = 1, 3$, which give $(x; y_1, y_2) = (1; 1, 1)$ and $(x; y_1, y_2) = (1; 3, 1)$. 
\end{proof}

For a Markoff triple $(x; y_1, y_2)$ we define
\[
  x' = 2y_1y_2 - x,\quad
  y_1' = 4xy_2 - y_1,\quad
  y_2' = 4xy_1 - y_2.
\]
Then 
$(x'; y_1, y_2)$, 
$(x; y_1', y_2)$, 
$(x; y_1, y_2')$ are easily seen to be Markoff triples and we call them \emph{neighbors} of $(x; y_1, y_2)$.

\begin{proposition}[Chapter 2 of \cite{CF89}]
  Let $(x; y_1, y_2)$ be a nontrivial Markoff triple. 
  Then a maximum element of exactly one of its neighbors is less than $\max(x, y_1, y_2)$ and the remaining two neighbors have greater maximum elements than $\max(x, y_1, y_2)$.
  \label{prop:smaller_max}
\end{proposition}
\begin{proof}
  When $(x; y_1, y_2) = (1; 3, 1)$, its neighbors are $(1; 1, 1)$, $(5; 3, 1)$ and $(1; 3, 11)$, and the statement in the proposition is trivially true.
  So, we will now assume $(x; y_1, y_2) \neq (1; 3, 1)$ and therefore $x, y_1, y_2$ are distinct (see Proposition~\ref{prop:distinctMarkoff}).

  First, we prove the proposition when $x = \max(x, y_1, y_2)$. 
  Define $f(z) = 2z^2 - 4y_1y_2z + y_1^2 + y_2^2$.
  Then $x$ and $x'$ are the two zeros of $f(z)$.
  If $y_1 > y_2$, then
\begin{equation*}
  f(y_1) = 3y_1^2 - 4y_1^2 y_2 + y_2^2 
   < 4y_1^2 - 4y_1^2 y_2 = 4y_1^2 (1 - y_2) 
   \le 0.
\end{equation*}
So $y_1$ is strictly between $x$ and $x'$. 
Because $x =\max(x; y_1, y_2)$ we obtain $ x > y_1 > x'$. 
If $y_2 > y_1$ then a similar argument gives $x>y_2>x'$.
So we conclude $x > \max(y_1, y_2) > x'$. 
In particular, $\max(x'; y_1, y_2)<x$.
Additionally, 
\begin{align*}
  y_1' &= 4xy_2 - y_1 > 4x - x >x, \\
  y_2' &= 4xy_1 - y_2 > 4x - x >x.
\end{align*}
So the maxima of $(x; y_1, y_2')$ and $(x; y_1', y_2)$ are greater than $x = \max(x; y_1, y_2)$.

We move to the case $y_1 = \max(x, y_1, y_2)$. 
(If $y_2=\max(x, y_1, y_2)$, then the proof will be essentially identical, so we omit it.)
In this case, we define $f(z) = z^2 - 4xy_2z + 2x^2 + y_2^2$, whose zeros are easily seen to be $y_1$ and $y_1'$.
Suppose $x > y_2$ first. Then
\begin{equation*}
  f(x) = 3x^2 - 4x^2 y_2 + y_2^2 
   < 4x^2 - 4x^2 y_2 = 4x^2 (1 - y_2) 
   \le 0.
\end{equation*}
As before, this yields $y_1 > x > y_1'$. 
If $y_2 > x$, then 
\[
  f(y_2) = 2y_2^2 - 4x y_2^2 + 2x^2  < 4y_2^2(1 - x) \le 0,
\]
which gives $y_1 > y_2 > y_1'$. 
In conclusion, $y_1 > \max(x, y_2) > y_1'$.
This proves $\max(x; y_1', y_2)<\max(x; y_1, y_2)$. 
Also, 
\begin{align*}
  x' &= 2y_1y_2 - x > 2y_1 - y_1 = y_1, \\
  y_2' &= 4xy_1 - y_2 > 4y_1 - y_1 >y_1.
\end{align*}
Hence the maxima of $(x'; y_1, y_2)$ and $(x; y_1, y_2')$ are greater than $\max(x; y_1, y_2)$.
\end{proof}
To describe our construction of the Markoff tree, we define a \emph{branching rule for Markoff triples} using Proposition~\ref{prop:smaller_max} as follows.
For a given nontrivial Markoff triple $(x; y_1, y_2)$, add the two neighbors whose maxima are greater than $\max(x, y_1, y_2)$, and connect $(x; y_1, y_2)$ with them.
We construct the Markoff tree by beginning with $(1; 3, 1)$ and by applying the branching rule recursively.
See Figure~\ref{Mtree}.

It is not difficult to see that the Markoff tree contains every nontrivial Markoff triple exactly once.
Indeed, for an arbitrary nontrivial Markoff triple $(x; y_1, y_2)$, we can move to the neighbor with a smaller maximum than $\max(x; y_1, y_2)$.
By repeating this process finitely many times, we see that $(x; y_1, y_2)$ is connected to $(1; 3, 1)$ by a sequence of successive neighbors, which characterizes $(x; y_1, y_2)$ uniquely.

\begin{definition}\label{DefMarkoffNumber}
    For a finite word $w$ on $\{a, b \}$, we define the \emph{Markoff number} $m(w)$ of $w$ to be
    \[
    m(w) = 
    \begin{cases}
    \Tr(N(w))/4 & \text{ if $w$ is even}, \\
    \Tr(N(w))/(2\sqrt2)& \text{ if $w$ is odd}.
        \end{cases}
    \]
    Here, $N(w)$ is the Cohn matrix of $w$ (cf.~Definition~\ref{DefCohnMatrix}).
\end{definition}
\begin{lemma}
  Suppose that $u$ and $v$ are finite words on $\{ a, b \}$.
  If both $u$ and $v$ are odd, then
  \[
      m(u^2 v) = 4m(u)m(uv) - m(v), \quad
      m(u v^2) = 4m(uv)m(v) - m(u). 
  \]
  If $u$ is odd and $v$ is even, then
  \[
      m(u^2 v) = 2m(u)m(uv) - m(v), \quad
      m(u v^2) = 4m(uv)m(v) - m(u). 
  \]
  If $u$ is even and $v$ is odd, then
  \[
      m(u^2 v) = 4m(u)m(uv) - m(v), \quad
      m(u v^2) = 2m(uv)m(v) - m(u). 
 %   \end{cases}
  \]
  \label{lem:trace_identities}
\end{lemma}
\begin{proof}
  We begin with an identity
  \[
    \Tr(M_1M_2) = \Tr(M_1)\Tr(M_2)-\Tr(M_1^{-1}M_2) =
    \Tr(M_1)\Tr(M_2)-\Tr(M_1M_2^{-1})
  \]
  for any $M_1, M_2\in \mathrm{SL}_2(\mathbb{R})$ (see Lemma 4.2 in \cite{Aig13}).
  This is easily established by direction calculation, so we omit the proof. 
  Using this, we obtain
  \begin{equation}
    \begin{cases}
      \Tr(U^2V) = \Tr(U)\Tr(UV) - \Tr(V), \\
      \Tr(UV^2) = \Tr(UV)\Tr(V) - \Tr(U)
    \end{cases}
    \label{eq:matrix_induction}
  \end{equation}
  for any $U, V \in \mathrm{SL}_2(\mathbb{R})$. 
  All the equalities in the lemma now follow from \eqref{eq:matrix_induction} and Definition~\ref{DefMarkoffNumber}.
\end{proof}
\begin{proposition}
  For any finite words $w_1$ and $w_2$ on $\{ a, b\}$, not both even, we have
  \[
    m(w_1w_2) > \max(m(w_1), m(w_2)).
  \]
  \label{prop:max_Markoff_number}
\end{proposition}
\begin{proof}
  Let us write 
  \[
    N(w_1) =
    \begin{pmatrix}
      a_1 & b_1 \\
      c_1 & d_1 \\
    \end{pmatrix}
    \quad
    \text{and}
    \quad
    N(w_2) =
    \begin{pmatrix}
      a_2 & b_2 \\
      c_2 & d_2 \\
    \end{pmatrix}.
  \]
  We note from Corollary~\ref{CorCohnForm} that $a_j, b_j, c_j, d_j \ge1$ for $j = 1, 2$ and, if $w_j$ is odd, $a_j, d_j\ge\sqrt2$.
  So,
  \begin{equation}
     \Tr(N(w_1)N(w_2)) 
     = (a_1a_2 + d_1d_2) + (b_1c_2 + c_1b_2) 
    > \min(a_2, d_2) \Tr(N(w_1)).
    \label{eq:Tr_bound1}
  \end{equation}
  Reversing the role of $N(w_1)$ and $N(w_2)$, we have
  \begin{equation}
     \Tr(N(w_1)N(w_2)) 
    > \min(a_1, d_1) \Tr(N(w_2)).
    \label{eq:Tr_bound2}
  \end{equation}
  To prove the statement in the proposition, we first consider the case when $w_1w_2$ is odd.
  If $w_1$ is odd and $w_2$ is even, we use \eqref{eq:Tr_bound1} to obtain
\begin{align*}
    m(w_1w_2) = 
    \frac{\Tr(N(w_1w_2))}{2\sqrt2} 
    &> \min(a_2, d_2) \frac{\Tr(N(w_1))}{2\sqrt2} \\
    &\ge
    \frac{\Tr(N(w_1))}{2\sqrt2}
    =m(w_1)
\end{align*}
  and from \eqref{eq:Tr_bound2} 
\begin{align*}
    m(w_1w_2) = 
    \frac{\Tr(N(w_1w_2))}{2\sqrt2}  
    &> \sqrt 2  \min(a_1, d_1) \frac{\Tr(N(w_2))}{4} \\
    &> \frac{\Tr(N(w_2))}{4}
    =m(w_2).
\end{align*}
  The case for $w_1$ even and $w_2$ odd is similar and we omit it. 
  Next, suppose $w_1w_2$ is even, which necessarily implies that $w_1$ and $w_2$ are odd.
  Because $\min(a_2, d_2) \ge \sqrt2$, we obtain from \eqref{eq:Tr_bound1} 
  \[
    m(w_1w_2) = 
    \frac{\Tr(N(w))}{4} > 
    \frac{1}{\sqrt2}
    \min(a_2, d_2) \frac{\Tr(N(w_1))}{2\sqrt2}\ge 
    m(w_{1}).
  \]
  The same argument produces $m(w_1w_2) > m(w_2)$ as well.
\end{proof}

Note that, for any Christoffel pair $(u, v)$, exactly one in $\{ uv, u, v\}$ is an even word and the remaining two are odd words.
Indeed, such a property is trivially true for $(a, b)$ and, if $(u, v)$ has this property, then the two new pairs $(u, uv)$ and $(uv, v)$ created by the branching rule also satisfy this property.
By writing $(m(uv), m(u), m(v))$ in the following theorem, we shall mean $(m(z); m(z_1), m(z_2))$ where $z$ is the even word in $\{ uv, u, v \}$ and $z_1$ and $z_2$ are the two odd words.
\begin{theorem}\label{ThmTheorem1Reutenauer}
  The Christoffel tree and the Markoff tree are isomorphic as graphs, which means that there is a bijection between the set of all Christoffel pairs and the set of all nontrivial Markoff numbers and adjacent Christoffel pairs are mapped to neighbors under this bijection.
  This graph isomorphism is given by
  \[
    (u, v) \mapsto (m(uv), m(u), m(v)).
  \]
\end{theorem}
\begin{proof}
  We note from Definition~\ref{DefCohnMatrix} that
  \[
    A =
    \begin{pmatrix}
      3 & \sqrt 2 \\ \sqrt2 & 1
    \end{pmatrix},
    \quad
    B =
    \begin{pmatrix}
      \sqrt 2 & 1\\ 1& \sqrt2 
    \end{pmatrix},
    \quad
    AB =
    \begin{pmatrix}
      4\sqrt2 & 5 \\ 3 & 2\sqrt2
    \end{pmatrix},
  \]
  so that $(m(ab), m(a), m(b)) = (1; 3, 1)$, which is the root of the Markoff tree.

  Next, we let $(u, v)$ be a Christoffel pair and assume that $(m(uv), m(u), m(v))$ is a nontrivial Markoff triple.
  Because of the branching rule for the Christoffel tree, it suffices to show that
  $(m(u^2v), m(u), m(uv))$ and $(m(uv^2), m(uv), m(v))$ are the two neighbors of
  $(m(uv), m(u), m(v))$ whose maxima are greater than 
  $\max(m(uv), m(u), m(v))$.

  Consider the case when $uv$ is even. 
  This necessarily implies that both $u$ and $v$ are odd. 
  Writing $x = m(uv)$, $y_1 = m(u)$, $y_2 = m(v)$, we know from Proposition~\ref{prop:max_Markoff_number} that $x = \max(x, y_1, y_2)$.
  Further, Lemma~\ref{lem:trace_identities} gives $m(u^2v) = 4xy_1 - y_2 = y_2'$ and $m(uv^2) = 4xy_2 -y_1 = y_1'$.
  This shows that 
  \begin{equation*}
    \begin{cases}
      (m(u^2v), m(u), m(uv)) = (x; y_1, y_2') \\
      (m(uv^2), m(uv), m(v)) = (x; y_1', y_2)
    \end{cases}
  \end{equation*}
  showing that they are 
  indeed the neighbors of $(x; y_1, y_2)$, whose maxima are greater than $x$. (cf.~Proposition~\ref{prop:max_Markoff_number})
  
  Now assume that $uv$ is odd.
  Then we have either (i) $u$ is odd and $v$ is even, or (ii) $u$ is even and $v$ is odd.
  For (i), we let $x = m(v)$, $y_1 = m(uv)$, $y_2 = m(u)$. 
  Then Propoisition~\ref{prop:max_Markoff_number} gives $y_1 = \max(x, y_1, y_2)$. Also $m(u^2v) = 2y_1y_2 - x = x'$ and $m(uv^2) = 4xy_1 - y_2 = y_2'$, so that
  \begin{equation}
    \begin{cases}
      (m(u^2v), m(u), m(uv)) = (x'; y_1, y_2) \\
      (m(uv^2), m(uv), m(v)) = (x; y_1, y_2').
    \end{cases}
    \label{eq:the_same_equation}
  \end{equation}
  For (ii), we let $x = m(u)$, $y_1 = m(uv)$, $y_2 = m(v)$ and we have $y_1 = \max(x, y_1, y_2)$ again.
  A similar calculation shows that \eqref{eq:the_same_equation} holds in this case as well.
  So for both (i) and (ii), we see that
  $(m(u^2v), m(u), m(uv))$ and
  $(m(uv^2), m(uv), m(v))$ are the desired neighbors of $(x; y_1, y_2)$.
  This completes the proof of the theorem.
\end{proof}

Recall from \eqref{DefMx} and \eqref{DefMy} of \S\ref{SecIntro} that we defined the two sets
\[
    \mathcal{M}_x = \{ x \mid
(x; y_1, y_2) \text{ is a Markoff triple}
\} 
=\{ 1, 5, 11, 29, 65, 349, \dots\}
\]
and
\begin{align*}
\mathcal{M}_y &= \{ \max\{y_1, y_2\} \mid (x; y_1, y_2) \text{ is a Markoff triple }
\}\\
&=\{ 1, 3, 11, 17, 41, 59, \dots\}. \notag
\end{align*}
We are now ready to complete the proof of our main theorem.
\begin{proof}[Proof of Theorem~\ref{MainTheoremDescDiscretePart}]
Let $P \in \QQQ$ be such that $L(P) < 2$.
Choose a doubly infinite Romik sequence $T$ with $L(T) = L(P)$ (cf.~Proposition~\ref{PropBombieriTrick}),
which then gives a strongly admissible doubly infinite word in $\{ a, b, a^{\vee} \}$.
We conclude from Theorem~\ref{BombieriThm15} that either $B$ or $B^{\vee}$ is equal to
\begin{enumerate}[font=\upshape, label=(\roman*)]
	\item $\cdots bbb \cdots$,
	\item $\cdots aaa \cdots$, 
	\item $B(w)$ for a nontrivial lower Christoffel word $w$.
\end{enumerate}
For the case (i), we have $T = \cdots 222 \cdots$.
Use Proposition~\ref{PropLpurelyPeriodic} with $N_2 = \left( \begin{smallmatrix}
1 & \sqrt2 \\ \sqrt2 & 1
\end{smallmatrix}
\right)$:
\[
L(P) = L(T) =
L( \cdots 2|2 \cdots ) = \frac{
\sqrt{
2^2 - 4(-1)
}
}{
\sqrt2 \sqrt2
}
=\sqrt2
=
\sqrt{
4 - \frac2{1^2}
}.
\]
Similarly, for the case (ii), we have $T = \cdots 3131 \cdots $.
Again, using Proposition~\ref{PropLpurelyPeriodic} with $N_3N_1 = 
\left(
\begin{smallmatrix}
3 & \sqrt2 \\ \sqrt2 & 1 
\end{smallmatrix}
\right)
$, we have
\[
L(P) = L(T) =
L( \cdots 31|31 \cdots ) = \frac{
\sqrt{
4^2 - 4
}
}{
\sqrt2 \sqrt2
}
=\sqrt3
=
\sqrt{
4 - \frac1{1^2}
}.
\]
(Proposition~\ref{PropLpurelyPeriodic} easily shows
$ 
L( \cdots 31|31 \cdots )
>  L( \cdots 13|13 \cdots )$, 
thus $L(T)$ is equal to
$
L( \cdots 31|31 \cdots )
$,
not
$
L( \cdots 13|13 \cdots )
$.)

It remains to consider the case (iii), when $B$ or $B^{\vee}$ is equal to $B(w)$ for a nontrivial Chritoffel word $w$. 
In this case, we have already calculated $L(B(w))$ in Theorem~\ref{ChristoffelL}:
\[
L(B(w)) = 
\sqrt{
4 - \frac2{q(w)^2}
}.
\]
However, when $w$ is a lower Christoffel word, we see from Proposition~\ref{Lemma32Reu} and Corollary~\ref{CorCohnForm} that $m(w)$ is a positive integer and that
\[
q(w) = 
\begin{cases}
\sqrt2 m(w)& \text{ if $w$ is even}, \\
m(w)    & \text{ if $w$ is odd}. \\
\end{cases}
\]
Therefore, we have
\[
L(P) = L(B(w)) = 
\begin{cases}
\sqrt{4 - (1/m(w)^2)}
& \text{ if $w$ is even}, \\
\sqrt{4 - (2/m(w)^2)}
& \text{ if $w$ is odd}. \\
\end{cases}
\]
Now, we let $(u, v)$ be the Christoffel pair giving the standard factorization of $w = uv$ and use Theorem~\ref{ThmTheorem1Reutenauer} to get $(m(w), m(u), m(v)) = (x; y_1, y_2)$. Then we conclude from this that $L(P) = \sqrt{4-2/x^2}$ 
or 
$L(P) = \sqrt{4-1/y_j^2}$ with $y_j = \max(y_1, y_2)$. 

Conversely, for any $x\in \mathcal{M}_x$ or $y\in \mathcal{M}_y$, we aim to find $B$ such that $L(B) = \sqrt{4-2/x^2}$ or $\sqrt{4-1/y^2}$. When $x = 1$ or $y = 1$, we've already shown that $B = \cdots aaa \cdots$ or $\cdots bbb \cdots$ gives the desired result.
Assume now that $x>1$ and let $(x; y_1, y_2)$ be a nontrivial Markoff triple containing $x$. By moving to a neighbor, if necessary, we may assume that $x = \max(x; y_1, y_2)$. By Theorem~\ref{ThmTheorem1Reutenauer}, there exists a Christoffel pair $(u, v)$ with $(m(uv), m(u), m(v)) = (x; y_1, y_2)$. In particular, $m(uv) = x$, thus $L(B(uv)) = \sqrt{4-2/x^2}$. Likewise, when $y >1$, there is a nontrivial Markoff triple $(x; y_1, y_2)$ with $y = \max(x, y_1, y_2)$. Then, by Theorem~\ref{ThmTheorem1Reutenauer} again, we have $(m(uv), m(u), m(v)) = (x; y_1, y_2)$ for a Christoffel pair $(u, v)$ and we obtain $L(B(uv)) = \sqrt{4-1/y^2}$.
This finishes the proof of Theorem~\ref{MainTheoremDescDiscretePart}.
\end{proof}

\end{document}